%% file: arxiv_v2.tex
\documentclass{article}
\usepackage[cm]{fullpage}
\usepackage{graphicx}
\usepackage{amsmath, amssymb, amsfonts, amsthm}
\usepackage{mathdots}
\usepackage{caption}
\usepackage{subcaption}
\usepackage{xcolor}
\usepackage{enumitem}
\usepackage[numbers]{natbib}
\usepackage{algorithm}
\usepackage{algpseudocode}
\usepackage{tikz}
\usetikzlibrary{calc}
\usetikzlibrary{arrows.meta}
\usepackage{bm}
\usepackage{array,booktabs,tabularx}
\usepackage{makecell}
\usepackage{nicefrac}

\renewcommand{\arraystretch}{1.25}
\newcolumntype{C}[1]{>{\centering\arraybackslash}m{#1}}

\theoremstyle{definition}
\newtheorem{definition}{Definition}
\theoremstyle{plain}
\newtheorem{theorem}{Theorem}
\newtheorem{lemma}{Lemma}
\newtheorem{proposition}{Proposition}
\newtheorem{corollary}{Corollary}

\usepackage[hidelinks]{hyperref}
\usepackage[capitalise]{cleveref}

\makeatletter
\renewcommand*\env@matrix[1][*\c@MaxMatrixCols c]{
  \hskip -\arraycolsep
  \let\@ifnextchar\new@ifnextchar
  \array{#1}}
\makeatother

\makeatletter
\if@titlepage\else
  \renewenvironment{abstract}{%
    \small
    \paragraph{\abstractname}
  }{\par\bigskip}
\fi
\makeatother

\allowdisplaybreaks

\title{A Theory of Composition and Duality of\\ Extremal Optimal Fixed-Point Algorithms}
\author{TaeHo Yoon\\
\texttt{tyoon7@jh.edu} 
\and Benjamin Grimmer\\
\texttt{grimmer@jhu.edu}}
\date{Department of Applied Mathematics and Statistics, Johns Hopkins University}

\input{macros}

\begin{document}

\maketitle

\begin{abstract}
In this work, we reveal a rich combinatorial structure underlying exact minimax optimal algorithms for classical nonexpansive fixed-point problems. This viewpoint unifies all extremal optimal methods and provides a systematic and practical framework for designing new algorithms via diagrams. Specifically, we study fixed-step algorithms represented by a lower triangular matrix $H$, and show that the set of optimal $(N-1)$-step algorithms has exactly $(N-1)!$ vertices (extremal algorithms), each of which naturally corresponds to an arc diagram, a graph that encodes its convergence proof. Using these arc diagrams, we can compose, decompose, and analyze the properties of distinct optimal vertex algorithms. Furthermore, we determine when the H-dual operation, given by taking the anti-diagonal transpose of $H$, preserves the optimality of a vertex algorithm, and in such cases we characterize the convergence proof of the dual algorithm. Based on this machinery, we develop new optimal algorithms with quasi-anytime guarantees; that is, they admit an increasing integer sequence $\{m_k\}_{k=0}^\infty$ such that each $y_{m_k}$ has the optimal residual guarantee, and are additionally robust to fixed-point operators that violate nonexpansiveness.

\vspace{.1cm}
\noindent
\textbf{Keywords\phantom{.} } Fixed-point problems $\cdot$ Acceleration $\cdot$ Rates of convergence $\cdot$ Computational complexity $\cdot$ Minimax optimality

\vspace{.1cm}
\noindent
\textbf{Mathematics Subject Classification\phantom{.} } 47H09 $\cdot$ 47J26 $\cdot$ 65J15 $\cdot$ 68Q25 $\cdot$ 90C47 $\cdot$ 90C60
\end{abstract}

\section{Introduction}

Consider the class of nonexpansive operators $\opT \colon \cX \to \cX$ on Hilbert space $\cX$, and the associated fixed-point problems
\begin{align}
\label{eqn:fixed-point}
    \underset{y \in \cX}{\text{find}} \quad y = \opT(y) .
\end{align}
In first-order optimization, the problem of solving \eqref{eqn:fixed-point} efficiently is often formulated in the form
\begin{align}
\label{eqn:minimax-optimality}
    \underset{\cA \in \fA}{\textrm{minimize}} \,\,\, \underset{(\opT, y_0)}{\textrm{maximize}} \,\,\, \cM \left( \cA, \opT, y_0, N \right)
\end{align}
where the maximization is taken over all nonexpansive operators $\opT\colon \cX\to\cX$ with at least one fixed point $y_\star$ and over initial points $y_0 \in \cX$, $\fA$ is the algorithm class of interest, and $\cM$ is a measure of error after $N$ evaluations of the operator $\opT$, where $\cM(\cdot) = 0$ indicates convergence.

A natural, standard error metric for nonexpansive fixed-point problems is the squared norm of the fixed-point residual $\sqnorm{y - \opT(y)}$. 
For this setting, the optimal value of~\eqref{eqn:minimax-optimality} is well-known; \citep{ParkRyu2022_exact} shows that with $\cM \left( \cA, \opT, y_0, N \right) = \frac{\sqnorm{y_{N-1} - \opT (y_{N-1})}}{\sqnorm{y_0 - y_\star}}$, one has $\underset{(\opT,y_0)}{\sup} \cM \left( \cA, \opT, y_0, N \right) \ge \frac{4}{N^2}$ for any deterministic first-order algorithm $\cA$, and \citep{Lieder2021_convergence, Kim2021_accelerated, YoonKimSuhRyu2024_optimal} present algorithms exactly matching this lower bound (we use $y_{N-1}$ because computing the fixed-point residual requires an additional evaluation of $\opT$).
However, identifying the \textit{solutions} to \eqref{eqn:minimax-optimality}, i.e., understanding which $\cA$ achieve \textit{minimax optimal} convergence is a deeper question.
This has only recently been answered by \citep{YoonRyuGrimmerInvariance2025} for the class $\fA_\textrm{H}$ of \textit{H-matrix representable} algorithms, which use fixed step-size updates
\begin{align}
\label{eqn:H-matrix-representation}
    y_{k+1} = y_k - \sum_{j=0}^k h_{k+1,j+1} \left( y_j - \opT (y_j) \right) 
\end{align}
where $h_{k,j}\in \reals$ are predetermined for $k=1,\dots,N-1$ and $j=1,\dots,k$.
More precisely, \citep{YoonRyuGrimmerInvariance2025} characterized the set $\cH_\star(N-1)$ of all $(N-1)\times (N-1)$ lower triangular matrices $H = [h_{k,j}]$ of these step-sizes whose corresponding algorithms are minimax optimal, solving~\eqref{eqn:minimax-optimality}.
The class $\fA_H$ and the optimal set $\cH_\star(N-1)$ can be transformed into the class of Mann iterations \citep{Mann1953_mean} and its optimal subset, which has been long studied in the literature on fixed-point iterations.

Among the infinitely many optimal fixed-point algorithms, known methods including the Optimal Halpern Method (OHM) \cite{Lieder2021_convergence} and Dual-OHM \cite{YoonKimSuhRyu2024_optimal} stand out as \textit{extremal} instances in two key senses: they can be expressed using convenient recursive forms, and their convergence proofs are readily understandable.
While the characterization provided in \citep{YoonRyuGrimmerInvariance2025} is exhaustive, their description of $\cH_\star(N-1)$ is based on algebraic invariance and inequalities rather than an explicit parametrization, obscuring the structure of $\cH_\star(N-1)$.
In particular, it leaves the problem of concretely understanding the extremal optimal algorithms with comprehensible structure and simple proofs largely open.

In this work, we present a study of the optimal \textit{vertex} algorithms, which are the geometric extremal points of $\cH_\star(N-1)$ possessing the simplest possible proof structure. 
We discover a useful combinatorial representation of vertex algorithms and a broader structure underlying them.
This new systematic language enables the complete enumeration and principled design of previously unknown extremal algorithms with optimal guarantees and yet new properties, which had not been directly possible based on the abstract framework of \cite{YoonRyuGrimmerInvariance2025}.
Our results also fully reveal the effect of H-duality \citep{KimOzdaglarParkRyu2023_timereversed, YoonKimSuhRyu2024_optimal}, a recently popularized notion of involutory relationship between first-order methods, for optimal fixed-point algorithms.

\begin{table}[t]
\centering
\caption{Arc diagrams of distinct optimal vertex algorithms for $N=5$, together with their H-matrices and definitions by recurrences. RDO(2) stands for Repeated Dual-OHM with period $2$, which we introduce in Section~\ref{section:repeated-dual-ohm}.}
\label{table:algorithms_arc_diagrams}
\setlength{\tabcolsep}{4pt}
\renewcommand{\arraystretch}{1.2}

\begin{tabular}{C{0.10\textwidth} C{0.20\textwidth} C{0.20\textwidth} C{0.38\textwidth}}
\toprule
Algorithm & Arc diagram & H-matrix & Recurrence \\
\midrule
\raisebox{-1.3\height}{OHM}
&
\raisebox{-0.7\height}{%
\begin{tikzpicture}[baseline=(current bounding box.center), scale=0.72,
    dot/.style={circle,fill,inner sep=1.5pt},
    every node/.style={font=\small}
]
    \foreach \i in {1,...,5}{
        \coordinate (L\i) at (\i,0);
        \fill (L\i) circle (2pt);
        \node[below=2pt] at (L\i) {\i};
    }
    \draw[line width=0.8pt, color=gray, densely dashed] (L1) -- (L5);
    \foreach \i/\j in {1/2,2/3,3/4,4/5}{
        \draw[line width=0.8pt] (L\i) arc[start angle=180,end angle=0,radius=0.5];
    }
\end{tikzpicture}%
}
&
$\begin{bmatrix}
\frac12 \\
-\frac16 & \frac23 \\
-\frac1{12} & -\frac16 & \frac34 \\
-\frac1{20} & -\frac1{10} & -\frac{3}{20} & \frac45
\end{bmatrix}$
&
\raisebox{-0.5\height}{$\displaystyle y_{k+1} = \frac{k+1}{k+2}\opT y_k + \frac{1}{k+2}y_0$}
\\
\midrule

\raisebox{-1.9\height}{Dual-OHM}
&
\raisebox{-0.4\height}{%
\begin{tikzpicture}[baseline=(current bounding box.center), scale=0.72,
    dot/.style={circle,fill,inner sep=1.5pt},
    every node/.style={font=\small}
]
    \foreach \i in {1,...,5}{
        \coordinate (R\i) at (\i,0);
        \fill (R\i) circle (2pt);
        \node[below=2pt] at (R\i) {\i};
    }
    \draw[line width=0.8pt, color=gray, densely dashed] (R1) -- (R5);
    \foreach \i/\rad in {1/2,2/1.5,3/1,4/0.5}{
        \draw[line width=0.8pt] (R\i) arc[start angle=180,end angle=0,radius=\rad];
    }
\end{tikzpicture}%
}
&
\raisebox{-0.2\height}{$\begin{bmatrix}
\frac{4}{5} \\
-\frac{3}{20} & \frac{3}{4} \\
-\frac{1}{10} & -\frac{1}{6} & \frac{2}{3} \\
-\frac{1}{20} & -\frac{1}{12} & -\frac{1}{6} & \frac{1}{2}
\end{bmatrix}$}
&
\raisebox{-0.8\height}{%
$\displaystyle
y_{k+1} = y_k + \frac{N-k-1}{N-k}\bigl(\opT y_k - \opT y_{k-1}\bigr)$
}
\\
\midrule

\raisebox{-1.4\height}{RDO(2)}
&
\raisebox{-0.8\height}{%
\begin{tikzpicture}[baseline=(current bounding box.center), scale=0.72,
    dot/.style={circle,fill,inner sep=1.5pt},
    every node/.style={font=\small}
]
    \foreach \i in {1,...,5}{
        \coordinate (P\i) at (\i,0);
        \fill (P\i) circle (2pt);
        \node[below=2pt] at (P\i) {\i};
    }
    \draw[line width=0.8pt, color=gray, densely dashed] (P1) -- (P5);
    \foreach \i/\r in {1/1,2/0.5,3/1,4/0.5}{
        \draw[line width=0.8pt] (P\i) arc[start angle=180,end angle=0,radius=\r];
    }
\end{tikzpicture}%
}
&
$\begin{bmatrix}
\frac{2}{3}  \\
-\frac{1}{6} & \frac{1}{2} \\
-\frac{1}{5} & -\frac{1}{5} & \frac{6}{5} \\
0 & 0 & -\frac{3}{10} & \frac{1}{2}
\end{bmatrix}$
&
\parbox[c]{\linewidth}{\small\centering
$\begin{aligned}
    y_{2k+1} & = y_{2k} - \frac{2(2k+1)}{2k+3} (y_{2k} - \opT y_{2k}) \\
    & \qquad + \frac{2}{2k+3}\,(y_0-y_{2k}) \\
    y_{2k+2} & = \frac{1}{2} \left( y_{2k+1} + \opT y_{2k+1} + \frac{2k+1}{2k+3} (y_{2k} - \opT y_{2k}) \right) 
\end{aligned}$}
\\
\bottomrule
\end{tabular}
\end{table}

\paragraph{Main contributions and organization.}
We organize our main technical contributions as follows.

\begin{itemize}
    \item \textbf{Section~\ref{section:general-theory-of-optimal-vertex-algorithms}.} We develop the foundational theory of optimal vertex algorithms. 
    We establish a complete correspondence between vertices of $\cH_\star(N-1)$ and \textit{arc diagrams}---a specific type of tree graphs, and equip this representation with a gluing operation and decomposition, which explain how vertex algorithms are constructed and how their intermediate guarantees arise.
    \begin{itemize}
        \item \textbf{Sections~\ref{subsection:vertex-algorithm-characterization-existence-computation} and \ref{subsection:arc-diagram-representation}.}
        We show that $\cH_\star(N-1)$ has precisely $(N-1)!$ vertices, each uniquely determined by the set of inequalities used in its convergence proof (\cref{theorem:all-sparsity-patterns-give-optimal-vertex}). 
        Each vertex algorithm is identified with an arc diagram on $N$ nodes, a specific form of undirected tree graph encoding the sparsity pattern of the convergence proof (\cref{definition:diagram-of-algorithm}). 
        Given an arc diagram, the associated optimal vertex algorithm's H-matrix $H \in \cH_\star(N-1)$ can be computed via Meta Algorithms~\ref{meta-alg:find-H-from-Q} and \ref{meta-alg:find-Q-from-sparsity-pattern}.
    
        \item \textbf{Sections~\ref{subsection:gluing-operation} and \ref{subsection:decomposable-and-basic}.}
        We introduce a gluing operation that composes arc diagrams or H-matrices to produce a larger vertex algorithm whose update rule and convergence proof inherit those of their components (Definitions~\ref{definition:gluing-of-diagrams}--\ref{definition:gluing-of-lambdas} \& \cref{theorem:gluing}). 
        Conversely, we define decomposable algorithms and diagrams, which arise from gluing smaller components (Definitions~\ref{definition:decomposable-diagram}--\ref{definition:decomposable-algorithm}). 
        A decomposable vertex algorithm enjoys optimal guarantees at intermediate iterates, whose indices can be read directly from its arc diagram (\cref{corollary:intermediate-iterate-guarantee-for-decomposable-algorithm}).
    \end{itemize}

    \item \textbf{Section~\ref{section:duality-via-diagrams}.} We identify the condition for an optimal vertex algorithm to remain optimal under the so-called H-dual operation \citep{KimOzdaglarParkRyu2023_timereversed, YoonKimSuhRyu2024_optimal}, which is the anti-diagonal transpose of $H$.
    While it is known that H-dual preserves a fixed-point algorithm's performance in special cases, it has been only partially understood when that symmetry manifests.
    We provide a clear and complete characterization of it for optimal vertex algorithms: H-dual optimality is equivalent to the arc diagram being non-crossing (\cref{theorem:dual-optimality-and-non-crossing-arc-diagram}), a purely combinatorial property. 
    Furthermore, we show that when this is the case, the H-dual is another vertex algorithm whose arc diagram can be computed via Meta Algorithm~\ref{meta-alg:basic-diagram-dualization}.

    \item \textbf{Section~\ref{section:deducing-algorithm-design-and-analysis}.} We introduce two classes of novel algorithms designed via arc diagrams: \textit{Repeated Dual OHM (RDO)} and \textit{Fractal Self-Dual Method (FSDM)}.
    Both algorithms are quasi-anytime optimal, i.e., there exists an increasing integer sequence $\{m_k\}_{k=0}^\infty$ such that each $y_{m_k}$ is optimal.
    We additionally show that having more intermediate-iterate guarantees comes at a cost of reducing an algorithm's robustness, and OHM and Dual-OHM are at the opposite extremes of this trade-off spectrum (Propositions~\ref{proposition:intermediate-guarantee-costs-robustness}--\ref{proposition:OHM-is-least-robust}), while the new algorithms, RDO and FSDM, lie in between.
\end{itemize}

We defer proofs involving substantial linear algebra or algebraic verification to the appendix. 
While these are necessary for completeness, they are not the main conceptual contributions of the work.
This organization keeps the main text focused on the central ideas and preserves the flow of the exposition without interruptions from lengthy technical proofs.

\section{Background and prior results on optimal algorithms}

This section covers preliminary material that is needed for developing new results in subsequent sections.

\subsection{Basic definitions and notations}

For $y \in \cX$ and an operator $\opT\colon \cX \to \cX$, we denote $\opT(y) = \opT y$ for simplicity.
We say that an operator $\opT$ is nonexpansive if $\|\opT x - \opT y\| \le \|x-y\|$ for all $x,y\in \cX$.
If $y_\star \in \cX$ satisfies $y_\star = \opT y_\star$, it is a fixed point of $\opT$.
For $y \in \cX$, its fixed-point residual is the vector $y-\opT y$, which is $0 \in \cX$ if and only if $y$ is a fixed point of $\opT$.
Hence, the norm of the fixed-point residual is a natural measure of convergence for the fixed-point problem \eqref{eqn:fixed-point}.

Define $\opA\colon \cX \rightrightarrows \cX$ as the set-valued operator (i.e., $\opA x := \opA(x) \subseteq \cX$ for $x\in \cX$) uniquely characterized by $\opT = 2(\opI + \opA)^{-1} - \opI$, where $\opI$ is the identity operator on $\cX$. 
Then, $\opA$ is maximally monotone: $\inprod{x-y}{u-v} \ge 0$ for any $x,y \in \cX$ and $u\in \opA x, v\in \opA y$, and the graph $\gra \opA = \{(x,u)\,|\, x\in \cX, u\in \opA x\}$ of $\opA$ cannot be extended to a strictly larger graph of a monotone operator, which is a standard fact in convex analysis as shown in \citep[Proposition~23.8]{bauschke2017convex} or \citep{BauschkeMoffatWang2012_firmly}.
This establishes a one-to-one correspondence between maximally monotone operators and nonexpansive operators.
Furthermore, given $y \in \cX$, if we write $x = \frac{y + \opT y}{2} = \frac{1}{2}(\opI + \opT) (y)$, then
\[
    x = \frac{y + \opT y}{2} = \frac{1}{2}(\opI + \opT) (y) \iff
    x \in (\opI + \opA)^{-1} (y) \iff y \in x + \opA x \iff y-x \in \opA x .
\]
In particular, in the H-matrix representation \eqref{eqn:H-matrix-representation} of an iterative algorithm, we let $x_{k+1} = \frac{1}{2}(\opI + \opT) (y_k)$ for $k=0,\dots,N-1$ and $g_{k+1} = y_k - x_{k+1} \in \opA x_{k+1}$.
Because $y_j - \opT y_j = 2g_{j+1}$, we can rewrite \eqref{eqn:H-matrix-representation} as
\begin{align}
\label{eqn:H-matrix-representation-via-g-vectors}
    y_{k+1} = y_k - \sum_{j=0}^k 2h_{k+1,j+1} g_{j+1} .
\end{align}
Furthermore, the nonexpansivity inequality between the iterates $y_k$ and $y_j$ is a rescaled monotonicity inequality between $x_{k+1}$ and $x_{j+1}$:
\begin{align}
\label{eqn:nonexpansivity-and-monotonicity}
    \sqnorm{y_k - y_j} - \sqnorm{\opT y_k - \opT y_j} = 4 \inprod{x_{k+1} - x_{j+1}}{g_{k+1} - g_{j+1}} \ge 0 
\end{align}
where $k,j=0,\dots,N-1$.

\subsection{Complexity lower bound}

The following result from \citep{ParkRyu2022_exact} establishes a tight lower bound on the dimension-independent worst-case rate with respect to the squared fixed-point residual.

\begin{lemma}[Theorem~4.6 of \citep{ParkRyu2022_exact}]
\label{lemma:lower-bound}
Let $N \in \mathbb{N}$ be given, and let $R>0$. Then, for any iterative algorithm making deterministic updates using $N-1$ operator evaluations, there exists a nonexpansive operator $\opT\colon \reals^d \to \reals^d$ with $d\ge 2N-2$ that has a fixed point $y_\star \in \reals^d$, and $y_0 \in \reals^d$ such that $\|y_0 - y_\star\| = R$, for which $\sqnorm{y_{N-1} - \opT y_{N-1}} \ge \frac{4\sqnorm{y_0 - y_\star}}{N^2}$.
\end{lemma}

The rest of the paper will focus on algorithms that match this lower bound without any gap, i.e., exact minimax optimal algorithms for our nonexpansive fixed-point problem setting.

\subsection{Optimal Halpern Method (OHM) and its H-dual}

The Halpern iteration \citep{Halpern1967_fixed} is a classical algorithm that interpolates the operator evaluation $\opT y_k$ with the initial point $y_0$:
it has the general update rule $y_{k+1} = \beta_k y_0 + (1-\beta_k) \opT y_k$.
With suitable choices of $\beta_k$, Halpern iteration converges at an up-to-constant-optimal rate $\sqnorm{y_{N-1} - \opT y_{N-1}} = \cO\left(\frac{\sqnorm{y_0 - y_\star}}{N^2}\right)$, which was first shown by \citep{SabachShtern2017_first} for fixed-point problems in normed spaces.
Since then, the Halpern iteration has been studied as a key mechanism for optimal acceleration of fixed-point problems and related settings in optimization \citep{Diakonikolas2020_halpern, Lieder2021_convergence, ParkRyu2022_exact, bravoUniversalBoundsFixed2022, contrerasOptimalErrorBounds2023, leeAcceleratingValueIteration2023, bravo2026minimax}.
In particular, \citep{Lieder2021_convergence, Kim2021_accelerated} showed\footnote{
To be precise, \cite{Kim2021_accelerated} analyzed an algorithm named \textit{Accelerated Proximal Point Method (APPM)}, which was expressed using a difference recurrence of the form \eqref{eqn:H-matrix-representation-via-g-vectors}. It was later noted in the literature \cite{contrerasOptimalErrorBounds2023, ryuLargescaleConvexOptimization2022} that APPM is in fact equivalent to \ref{eqn:OHM}.
} 
that the following version of the Halpern iteration using $\beta_k = \frac{1}{k+2}$:
\begin{align}
\tag{OHM}
\label{eqn:OHM}
    y_{k+1} = \frac{1}{k+2} y_0 + \frac{k+1}{k+2} \opT y_k
\end{align}
which we refer to as \textit{Optimal Halpern Method (OHM)}, converges at the rate
\begin{align}
\label{eqn:opt-rate}
\tag{opt-rate}
    \sqnorm{y_{N-1} - \opT y_{N-1}} \le \frac{4\sqnorm{y_0 - y_\star}}{N^2}    
\end{align}
matching the lower bound of \cref{lemma:lower-bound} exactly; hence it is minimax optimal.

The H-matrix for OHM can be obtained by rewriting its first $N-1$ iterations in the form \eqref{eqn:H-matrix-representation} (as done in \citep{Kim2021_accelerated}):
\begin{align*}
    \left(H_{\text{OHM}}(N-1)\right)_{k,j} =
    \begin{cases}
        -\frac{j}{k(k+1)} & \text{if } j < k \\
        \frac{k}{k+1}     & \text{if } j = k
    \end{cases}
    = \begin{bmatrix}
        \nicefrac{1}{2} \\
        \nicefrac{-1}{6} & \nicefrac{2}{3} \\
        \vdots & \vdots & \ddots \\
        -\nicefrac{1}{N(N-1)} & -\nicefrac{2}{N(N-1)} & \dots & \nicefrac{(N-1)}{N}
    \end{bmatrix} \in \reals^{(N-1) \times (N-1)} .
\end{align*}
Given any lower triangular H-matrix $H = [h_{k,j}] \in \reals^{(N-1)\times (N-1)}$ and the associated fixed-point algorithm defined via \eqref{eqn:H-matrix-representation}, we consider its \textit{H-dual}, obtained by reflecting $H$ along its anti-diagonal:
\begin{align}
\tag{H-dual}
\label{eqn:H-dual}
    H^\at = \begin{bmatrix} h_{N-j,N-k} \end{bmatrix}_{\substack{k=1,\dots,N-1 \\ j=1,\dots,k}} 
    = \begin{bmatrix}
        h_{N-1,N-1} \\
        h_{N-1,N-2} & h_{N-2,N-2} \\
        \vdots & \vdots & \ddots \\
        h_{N-1,1} & h_{N-2,1} & \cdots & h_{1,1}
    \end{bmatrix} .
\end{align}
This is also an $(N-1) \times (N-1)$ lower triangular matrix, and therefore defines a fixed-step algorithm.
Furthermore, it is known in special cases that the H-dual operation preserves the worst-case performance of certain fixed-point algorithms \citep{YoonKimSuhRyu2024_optimal}.
OHM is an example of such a correspondence---its H-dual, called \textit{Dual-OHM}, has the H-matrix 
\begin{align*}
    \left(H_{\text{Dual-OHM}}(N-1)\right)_{k,j} =
    \begin{cases}
        -\frac{N-k}{(N-j)(N-j+1)} & \text{if } j < k \\
        \frac{N-k}{N-k+1}         & \text{if } j = k
    \end{cases} 
    = \begin{bmatrix}
        \nicefrac{(N-1)}{N} \\
        \nicefrac{-(N-2)}{N(N-1)} & \nicefrac{(N-2)}{(N-1)} \\
        \vdots & \vdots & \ddots \\
        -\nicefrac{1}{N(N-1)} & -\nicefrac{1}{(N-1)(N-2)} & \dots & \nicefrac{1}{2}
    \end{bmatrix} 
\end{align*}
and the update rule can be rewritten into a simple recurrence: for $k=0,\dots,N-2$,
\begin{align}
\label{eqn:Dual-OHM} 
\tag{Dual-OHM}
    y_{k+1} = y_k + \frac{N-k-1}{N-k} \left( \opT y_k - \opT y_{k-1} \right) 
\end{align}
with the convention $\opT y_{-1} = y_0$. 
Dual-OHM also achieves \eqref{eqn:opt-rate} and thus is minimax optimal \citep{YoonKimSuhRyu2024_optimal}.

\subsection{Characterization of optimal H-matrices and Q-functions}

Denote by $\cH_\star(N-1)$ the set of all lower triangular $(N-1)\times (N-1)$ H-matrices whose associated algorithm defined by~\eqref{eqn:H-matrix-representation} satisfies \eqref{eqn:opt-rate} for any nonexpansive $\opT\colon\cX \to \cX$ with a fixed point $y_\star$.
While the previous discussion shows
\[
    H_\text{OHM}(N-1), H_\text{Dual-OHM}(N-1) \in \cH_\star(N-1) ,
\]
a more holistic characterization of $\cH_\star(N-1)$ was provided in \citep{YoonRyuGrimmerInvariance2025} via equality and inequality constraints in $H$.
For $m,j=1,\dots,N-1$, the \textit{Q-functions} of a lower triangular $H \in \reals^{(N-1)\times (N-1)}$ are defined as
\begin{align*}
    Q(m,j; H) = \sum_{\substack{ j = j(1) \le i(1) < j(2) \le i(2) < \cdots \le i(m-1) < j(m) \le i(m) \le N-1}} \prod_{r=1}^m h_{i(r),j(r)} ,
\end{align*}
following \cite{YoonRyuGrimmerInvariance2025}.
This is the sum of all products of $m$ entries of $H$, formed by starting at some $h_{i(1),j}$ with $j \le i(1)\le N-1$ and appending the product with another entry with column index $j(2) > i(1)$, and repeating a similar process until one gets $m$ entries.
When $m+j > N$, the summation is vacuous (it is not possible to form such a product), so $Q(m,j; H) = 0$.
When $m+j = N$, the only admissible choice of indices is when $i(r)=j(r)=j+r-1$, so $Q(m,j;H) = Q(N-j,j;H) = h_{j,j} h_{j+1,j+1} \cdots h_{N-1,N-1}$.
For a more concrete illustration, when $N=4$ so $H$ is $3\times 3$, we have
\begin{align*}
    & Q(1,1;H) = h_{1,1} + h_{2,1} + h_{3,1} , \quad Q(1,2;H) = h_{2,2} + h_{3,2} , \quad Q(1,3;H) = h_{3,3} \\
    & Q(2,1;H) = h_{1,1} h_{2,2} + h_{1,1} h_{3,2} + h_{1,1} h_{3,3} + h_{2,1} h_{3,3} , \quad Q(2,2;H) = h_{2,2} h_{3,3} , \quad Q(2,3;H) = 0 \\
    & Q(3,1;H) = h_{1,1} h_{2,2} h_{3,3} , \quad Q(3,2;H) = 0 , \quad Q(3,3;H) = 0 .
\end{align*}
Given these notions, \citep[Theorem~3]{YoonRyuGrimmerInvariance2025} proved that $H \in \cH_\star(N-1)$ if and only if it satisfies the equality constraints
\begin{align}
\label{eqn:H-invariance}
    \sum_{j=1}^{N-1} Q(k, j; H) = \sum_{j=1}^{N-k} Q(k, j; H) = \frac{1}{N} \binom{N}{k+1} \qquad (k=1,\dots,N-1),
\end{align}
called the \textit{H-invariance condition}, together with the inequality constraints
\begin{alignat}{2}
    & \lambda_{N,j}^\star (H) = N\sum_{m=1}^{N-1} (-1)^{m-1} Q(m,j; H) = N\sum_{m=1}^{N-j} (-1)^{m-1} Q(m,j; H) \ge 0 \qquad & & (j=1,\dots,N-1) \label{eqn:H-certificate-N} \\
    &
    \begin{aligned}
        \lambda_{k,j}^\star (H) & = N\sum_{\ell=1}^{N-1} \sum_{m=1}^{N-1} (-1)^{\ell+m-1} \binom{\ell+m}{m} Q(\ell,j; H) Q(m,k; H) \\
        & = N\sum_{\ell=1}^{N-j} \sum_{m=1}^{N-k} (-1)^{\ell+m-1} \binom{\ell+m}{m} Q(\ell,j; H) Q(m,k; H) \ge 0
    \end{aligned}
    \quad & & (1 \le j < k \le N-1) \label{eqn:H-certificate-others}
\end{alignat}
where $\lambda^\star(H)$ are called \textit{H-certificates}.

\paragraph{H-certificates are unique multipliers in convergence proof.}
In \citep{YoonRyuGrimmerInvariance2025}, it was also shown that for $H$ satisfying~\eqref{eqn:H-invariance}, $\lambda_{k,j}^\star(H)$ defined by \eqref{eqn:H-certificate-N} and \eqref{eqn:H-certificate-others} are unique values satisfying the following ``universal proof template'' for optimal algorithms:
\begin{align}
\label{eqn:optimal-family-proof-core-identity}
\begin{aligned}
    0 & = N \sqnorm{g_N} + \inprod{g_N}{x_N - y_0} + \sum_{k=1}^N \sum_{j=1}^{k-1} \frac{\lambda_{k,j}^\star (H)}{4} \left( \sqnorm{y_{k-1} - y_{j-1}} - \sqnorm{\opT y_{k-1} - \opT y_{j-1}} \right) \\
    & = N \sqnorm{g_N} + \inprod{g_N}{x_N - y_0} + \sum_{k=1}^N \sum_{j=1}^{k-1} \lambda_{k,j}^\star (H) \inprod{x_k - x_j}{g_k - g_j}
\end{aligned}
\end{align}
where the second identity uses \eqref{eqn:nonexpansivity-and-monotonicity}.
Therefore, if $\lambda_{k,j}^\star(H) \ge 0$ for all $1\le j < k \le N$, the rightmost double summation is nonnegative, which implies
\begin{align*}
    & 0 \ge N \sqnorm{g_N} + \inprod{g_N}{x_N - y_0} \ge N \sqnorm{g_N} + \inprod{g_N}{y_\star - y_0} \ge N \sqnorm{g_N} - \left( \frac{N}{2} \sqnorm{g_N} + \frac{1}{2N} \sqnorm{y_0 - y_\star} \right) \\
    & \implies \sqnorm{g_N} \le \frac{\sqnorm{y_0 - y_\star}}{N^2} \iff \sqnorm{y_{N-1} - \opT y_{N-1}} \le \frac{4\sqnorm{y_0 - y_\star}}{N^2} 
\end{align*}
where we use $\inprod{g_N}{x_N - y_\star} = \frac{1}{4} \left( \sqnorm{y_{N-1} - y_\star} - \sqnorm{\opT y_{N-1} - \opT y_\star} \right) \ge 0$ for the second inequality.

\section{Optimal vertex algorithms and diagram representations}
\label{section:general-theory-of-optimal-vertex-algorithms}

The core objects of interest in this paper are the extremal points---i.e., \textit{vertices}---of the optimal H-matrix set $\cH_\star(N-1)$.
Each vertex is an intersection of a maximal number of zero sets of H-certificates $\lambda_{k,j}^\star(H)$, which are the boundary surfaces of $\cH_\star(N-1)$.
In other words, vertices represent the algorithms with sparsest/simplest convergence proofs, where a minimal number of inequalities are active (positive) within the proof template~\eqref{eqn:optimal-family-proof-core-identity} showing that $H$ satisfies \eqref{eqn:opt-rate}.

\subsection{Characterization, existence and computation of vertex algorithms}
\label{subsection:vertex-algorithm-characterization-existence-computation}

We start by characterizing the minimal proofs associated with the vertex algorithms.
In the following, we drop the $H$-dependence if it is clear from the context.

\begin{lemma}
\label{lemma:nonzero-lambda-exists}
If $H \in \cH_\star(N-1)$, then for $j=1,\dots,N-1$, there exists at least one $k \in \{j+1,\dots,N\}$ such that $\lambda_{k,j}^\star > 0$.
\end{lemma}

This result, whose proof is presented in Appendix~\ref{section:proof-of-sparsity-lemma}, shows that minimal proofs are the ones that have exactly one $k(j) \in \{j+1,\dots,N\}$ satisfying $\lambda_{k(j), j}^\star > 0$ for each $j=1,\dots,N-1$. 
Then a natural question is: do all $(N-1)!$ possible combinations of $(k(1), \dots, k(N-1))$, each of which specifies a sparsity pattern of H-certificates, have a corresponding optimal vertex algorithm $H \in \cH_\star(N-1)$?
We answer this question in the affirmative, as follows.

\begin{theorem}
\label{theorem:all-sparsity-patterns-give-optimal-vertex}
Let $k(1),k(2),\dots,k(N-1)$ be given, where each $k(j)$ is an integer satisfying $j<k(j)\le N$.
For any such choice of $k(j)$'s, there uniquely exists an optimal vertex algorithm $H \in \cH_\star(N-1)$ such that
\begin{align}
\label{eqn:lambda-sparsity-pattern}
    \lambda_{k(j),j}^\star(H) > 0 \text{ and } \lambda_{\ell,j}^\star(H) = 0 \text{ for all } \ell \ne k(j) , \,\, j < \ell \le N
\end{align}
for each $j=1,\dots,N-1$.
\end{theorem}

\begin{algorithm}[t]
\caption{Meta algorithm for recovering $H$ from values of $Q$}
\label{meta-alg:find-H-from-Q} 
\textbf{Input: } $q_{k,j}$ for $k=1,\dots,N-1$ and $j=1,\dots,N-k$, satisfying $q_{N-j,j} \ne 0$ for $j=1,\dots,N-1$
\begin{algorithmic}
\State $h_{N-1,N-1} = q_{1,N-1}$
\For{$k=N-2,N-3,\dots,1$}
    \State $h_{k,k} = \frac{q_{N-k,k}}{q_{N-k-1,k+1}}$
    \For{$i=k+1,\dots,N-2$}
        \State $h_{i,k} = \frac{1}{q_{N-i-1,i+1}} \left[ q_{N-i,k} - \sum_{\ell=k+1}^i \left(\sum_{m=k}^{\ell-1} h_{m,k}\right) q_{N-i-1,\ell} \right] - \sum_{m=k}^{i-1} h_{m,k}$ 
    \EndFor
    \State $h_{N-1,k} = q_{1,k} - \sum_{m=k}^{N-2} h_{m,k}$
\EndFor
\end{algorithmic}
\textbf{Output: } Lower triangular matrix $H = (h_{i,k})_{\substack{i=1,\dots,N-1, k=1,\dots,i}}$
\end{algorithm}

The proof of this theorem involves heavy linear algebraic arguments; it is deferred to Appendix~\ref{section:vertex-algorithms-existence-proof}.
At a high level, we show that the selection of a sparsity pattern uniquely determines the values of $Q(\cdot,\cdot)$, which in turn determine $H$ and from which we can verify the desired condition~\eqref{eqn:lambda-sparsity-pattern}. As a key insight to the first step in this process, we find that although the $\lambda_{k,j}^\star$ are mostly quadratic in Q-functions, the system of equations $\lambda_{\ell,j}^\star = 0$ for $\ell \ne k(j)$ reduces to a linear system with respect to $Q(\cdot,\cdot)$ when recursively solved backward, i.e., in the order $j=N-1,N-2,\dots,1$. After finding the values of $Q(\cdot,\cdot)$, we recover $H$ from them: the following lemma, which was essentially proved in \citep{YoonRyuGrimmerInvariance2025}, shows that this can be done via Meta Algorithm~\ref{meta-alg:find-H-from-Q}.

\begin{lemma}
There exists a one-to-one correspondence between lower triangular $H \in \reals^{(N-1) \times (N-1)}$ with nonzero main diagonal entries $h_{1,1}, h_{2,2}, \dots, h_{N-1,N-1}$ and $\{Q(k,j)\}_{\substack{k=1,\dots,N-1\\j=1,\dots,N-k}}$ with nonzero $Q(1,N-1), Q(2,N-2), \dots, Q(N-1,1)$.
The $H$ corresponding to a given Q-function profile is determined by Meta Algorithm~\ref{meta-alg:find-H-from-Q}.
\end{lemma}

In a nutshell, given a sparsity pattern of $\lambda^\star$, the associated optimal vertex $H$ can be computed through an explicit numerical procedure, involving nothing more complicated than solving linear systems.

\subsection{Arc diagram representation of vertex algorithms}
\label{subsection:arc-diagram-representation}

Given a vertex algorithm $H \in \cH_\star(N-1)$, we have $\lambda_{k(j),j}^\star (H) > 0$ for exactly one $k(j) \in \{j+1,\dots,N\}$ for each $j=1,\dots,N-1$.
Therefore, we can represent $H$ by the graph $\cG(H)$ with the set of nodes $\cV = \{1,\dots,N\}$ and the set of edges $\cE = \{(k(j), j): j=1,\dots,N-1\}$.
This is an \textit{arc diagram}\footnote{Interestingly, the same type of object appeared in \cite{black2024linear} in the context of memoryless pivot rules for linear programs on simplices, where it was referred to as an \textit{arborescence}, following graph-theoretic terminology. In this work, we call it an arc diagram for simplicity, and keep the graph-theoretic discussion minimal.} where each $j\in \cV$ with $j<N$ has a unique edge---or arc---connecting it to a larger node.
Note that graph theoretically, $\cG(H)$ is a tree, as all nodes $j < N$ are connected to $N$ along the unique increasing path $(j, k(j), k(k(j)), \dots, N)$ which we denote by $\cC(j)$.
Absence of cycles follows from the connected graph having $N$ nodes and $N-1$ edges.
Furthermore, any such arc diagram $\cG$ on $N$ nodes conversely determines a unique vertex $H \in \cH_\star(N-1)$ satisfying $\cG = \cG(H)$ by \cref{theorem:all-sparsity-patterns-give-optimal-vertex}.

In \cref{table:algorithms_arc_diagrams}, we display three examples of arc diagrams with $N=5$: for OHM we have $\lambda_{j+1,j}^\star(H) > 0 \iff k(j) = j+1$ for $j=1,\dots,N-1$, so every node is connected to its immediate successor. 
For Dual-OHM we have $\lambda_{N,j}^\star(H) > 0 \iff k(j) = N$ for $j=1,\dots,N-1$, so every node is directly connected to $N$.
The third example is a member of the new algorithm class named \textit{Repeated Dual-OHM (RDO)}, which we introduce later in Section~\ref{section:repeated-dual-ohm}.
It is the case with period $p=2$, and is characterized by its novel sparsity pattern where $\lambda_{j+2,j}^\star (H) > 0$ if $j$ is odd and $\lambda_{j+1,j}^\star(H) > 0$ if $j$ is even, which is visually evident from its arc diagram.

For reasons that will be clarified in later sections, we assign weights $\lambda_{k(j),j}^\star (H)$ to each edge $(k(j), j) \in \cE$ and view $\cG(H)$ as a weighted graph.
With this viewpoint, $\cG(H)$ can also be regarded as a weighted (complete) graph with any edge not in the form $(k(j),j)$ having weight $0$; in fact, in this way, we can more generally match any $H \in \cH_\star(N-1)$ with a weighted graph with edge weights $\lambda_{k,j}^\star (H)$. 
We summarize the discussion above into a formal definition below.

\begin{definition}
\label{definition:diagram-of-algorithm}
For $H \in \cH_\star(N-1)$, \textit{the diagram} $\cG(H)$ \textit{of $H$} is the weighted undirected graph with nodes $\cV = \{1,\dots,N\}$ and edges $\cE = \{(k,j): 1\le j < k \le N\}$, where each $(k,j) \in \cE$ has the weight $\lambda_{k,j}^\star (H)$. 
In the case where $H$ is a vertex of $\cH_\star(N-1)$, all edges other than $(k(j),j)$ ($j=1,\dots,N-1$) have weight $0$, so $\cG(H)$ (identified with its support graph) is a tree, called an \textit{arc diagram of $H$}.
\end{definition}

\subsection{Gluing operation for composing algorithms and diagrams}
\label{subsection:gluing-operation}

Arc diagrams are not merely concise visual representations of algorithms; in fact, they serve as convenient handles for composing two distinct optimal algorithms. 
We first define the operation for composing arc diagrams.
\begin{definition}
\label{definition:gluing-of-diagrams}
Let $1\le N' \le N-1$, and let $\cG_1$ and $\cG_2$ be arc diagrams on $N'$ and $N-N'$ nodes, respectively. 
Define their \textit{gluing} $\cG_1 \glue \cG_2$ as the arc diagram on $N$ nodes, obtained by
\begin{enumerate}
    \item Placing $\cG_2$ to the right of $\cG_1$, after relabeling the nodes of $\cG_2$ from $1,\dots,N-N'$ to $N'+1, \dots, N$, and then
    \item Inserting an additional arc joining the rightmost nodes of $\cG_1$ and $\cG_2$.
\end{enumerate}
\end{definition}

\begin{figure}[t]
\centering

\begin{tikzpicture}[
    scale=0.68,
    dot/.style={circle,fill,inner sep=1.8pt},
    every node/.style={font=\large},
    >=stealth
]

\def\xcA{0}
\def\xcB{6.3}
\def\xcC{16.4}

\def\yGraph{0}
\def\yGLabel{-1.55}
\def\yHLabel{-3.85}
\def\yMat{-8}

\begin{scope}[shift={(\xcA-3,\yGraph)}]
    \foreach \i in {1,...,5}{
        \coordinate (A\i) at (\i,0);
        \fill (A\i) circle (2.2pt);
        \node[below=2pt] at (A\i) {\i};
    }
    \draw[line width=0.8pt, color=gray, densely dashed] (A1) -- (A5);
    \foreach \i/\r in {1/1,2/0.5,3/1,4/0.5}{
        \draw[line width=0.8pt] (A\i) arc[start angle=180,end angle=0,radius=\r];
    }
\end{scope}

\node (G1label) at (\xcA,\yGLabel) {$\cG_1$};
\node (H1label) at (\xcA,\yHLabel) {$H_1$};
\node (H1mat) at (\xcA,\yMat) {$\begin{bmatrix}
\frac{2}{3}  \\
-\frac{1}{6} & \frac{1}{2} \\
-\frac{1}{5} & -\frac{1}{5} & \frac{6}{5} \\
0 & 0 & -\frac{3}{10} & \frac{1}{2}
\end{bmatrix}$};

\draw[<->, thick] (\xcA,\yGLabel-0.5) -- (\xcA,\yHLabel+0.5);

\begin{scope}[shift={(\xcB-3,\yGraph)}]
    \foreach \i in {1,...,5}{
        \coordinate (B\i) at (\i,0);
        \fill (B\i) circle (2.2pt);
        \node[below=2pt] at (B\i) {\i};
    }
    \draw[line width=0.8pt, color=gray, densely dashed] (B1) -- (B5);
    \foreach \i/\j in {1/2,2/3,3/4,4/5}{
        \draw[line width=0.8pt] (B\i) arc[start angle=180,end angle=0,radius=0.5];
    }
\end{scope}

\node (G2label) at (\xcB,\yGLabel) {$\cG_2$};
\node (H2label) at (\xcB,\yHLabel) {$H_2$};
\node (H2mat) at (\xcB,\yMat) {$\begin{bmatrix}
\frac12 \\
-\frac16 & \frac23 \\
-\frac1{12} & -\frac16 & \frac34 \\
-\frac1{20} & -\frac1{10} & -\frac{3}{20} & \frac45
\end{bmatrix}$};

\draw[<->, thick] (\xcB,\yGLabel-0.5) -- (\xcB,\yHLabel+0.5);

\begin{scope}[shift={(\xcC-5.5,\yGraph)}]
    \foreach \i in {1,...,10}{
        \coordinate (C\i) at (\i,0);
        \fill (C\i) circle (2.2pt);
        \node[below=2pt] at (C\i) {\i};
    }

    \draw[line width=0.8pt, color=gray, densely dashed] (C1) -- (C10);

    \foreach \i/\r in {1/1,2/0.5,3/1,4/0.5}{
        \draw[line width=0.8pt] (C\i) arc[start angle=180,end angle=0,radius=\r];
    }

    \draw[line width=0.8pt,red] (C5) arc[start angle=180,end angle=0,radius=2.5];

    \foreach \i in {6,7,8,9}{
        \draw[line width=0.8pt] (C\i) arc[start angle=180,end angle=0,radius=0.5];
    }
\end{scope}

\node (G12label) at (\xcC,\yGLabel) {$\cG_1 \glue \cG_2$};
\node (H12label) at (\xcC,\yHLabel) {$H_1 \glue H_2$};
\node (H12mat) at (\xcC,\yMat) {$\begin{bmatrix}
\frac{2}{3}  \\
-\frac{1}{6} & \frac{1}{2} \\
-\frac{1}{5} & -\frac{1}{5} & \frac{6}{5} \\
0 & 0 & -\frac{3}{10} & \frac{1}{2} \\
\redd{-\frac{3}{20}} & \redd{-\frac{3}{20}} & \redd{-\frac{9}{20}} & \redd{-\frac{1}{4}} & \redd{\frac{5}{2}} \\
0 & 0 & 0 & 0 & \redd{-\frac{1}{4}} & \frac{1}{2} \\
0 & 0 & 0 & 0 & \redd{-\frac{1}{4}} & -\frac16 & \frac23 \\
0 & 0 & 0 & 0 & \redd{-\frac{1}{4}} & -\frac1{12} & -\frac16 & \frac34 \\
0 & 0 & 0 & 0 & \redd{-\frac{1}{4}} & -\frac1{20} & -\frac1{10} & -\frac{3}{20} & \frac45
\end{bmatrix}$};

\draw[<->, thick] (\xcC,\yGLabel-0.5) -- (\xcC,\yHLabel+0.5);

\end{tikzpicture}
 
\caption{\textbf{Top:} Arc diagrams for RDO(2) and OHM with $N=5$, and their gluing $\cG_1 \glue \cG_2$. \textbf{Bottom:} The corresponding vertices of $\cH_\star(N-1)$ and their gluing $H_1 \glue H_2$, as defined in \cref{definition:gluing-of-H-matrices}.
By \cref{theorem:gluing}, $H_1 \glue H_2$ is the optimal vertex algorithm corresponding to $\cG_1 \glue \cG_2$.
The entries and arc added by the gluing operation are highlighted in red.
}
\label{figure:gluing-diagram-and-H-matrices}
\end{figure}
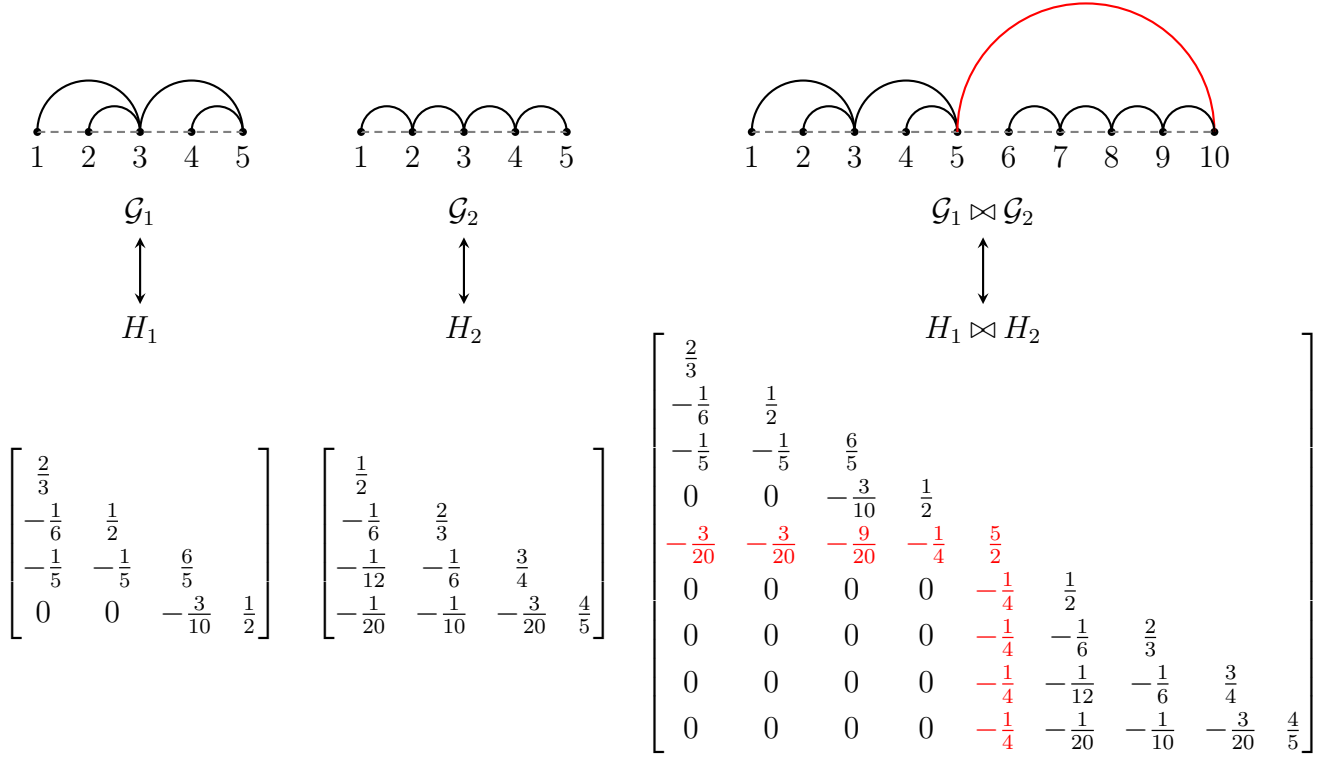

The top row of Figure~\ref{figure:gluing-diagram-and-H-matrices} provides a clear visual illustration of how the gluing of arc diagrams works.

As discussed in Section~\ref{subsection:arc-diagram-representation}, there are unique optimal vertex algorithms $H_1, H_2$ such that $\cG_1 = \cG(H_1)$ and $\cG_2 = \cG(H_2)$.
Likewise, there exists a unique optimal vertex algorithm $H$ satisfying $\cG_1 \glue \cG_2 = \cG(H)$.
Hence, through this correspondence, we can define an equivalent gluing operation for optimal vertex algorithms.
Interestingly, it turns out that the resulting ``glued matrix'' $H_1 \glue H_2$ manifests in the form that literally glues the two matrices by adding an additional row and column in between (see the bottom row of Figure~\ref{figure:gluing-diagram-and-H-matrices}).
We provide the formal definition of the gluing operation for H-matrices below, which can be performed with any two optimal (possibly non-vertex) $H_1, H_2$.
Using this notion, the aforementioned equivalence is concisely restated in \cref{theorem:gluing}.

\begin{definition}
\label{definition:gluing-of-H-matrices}
Let $1\le N' \le N-1$, and let $H_1 \in \cH_\star(N'-1)$, $H_2 \in \cH_\star(N-N'-1)$.
Define their \textit{gluing} $H_1 \glue H_2$ by
\begin{align}
\label{eqn:H-matrix-gluing}
    H_1 \glue H_2 = \begin{bmatrix}
        \begin{array}{c|cc}
             H_1 & \\ \begin{array}{ccc} h_{N',1} & \cdots & h_{N',N'-1} \end{array} & h_{N',N'} \\
            \hline 
            0_{(N-N'-1) \times (N'-1)} & \begin{array}{c} h_{N'+1,N'} \\ \vdots \\ h_{N-1,N'} \end{array} & H_2
        \end{array}
    \end{bmatrix}
\end{align}
where 
\begin{itemize}
    \item $h_{N',j} = -\frac{N-N'}{N} \sum_{k=j}^{N'-1} (H_1)_{k,j}$ for $j=1,\dots,N'-1$, and $h_{N',N'} = \frac{N'(N-N')}{N}$.

    \item $h_{N'+1,N'}, \dots, h_{N-1,N'}$ are unique values satisfying $Q(k,N'; H_1 \glue H_2) = \frac{N'(N-N'+1)}{N(k+1)} \binom{N-N'-1}{k-1}$ for $k=1,\dots,N-N'$, which can be computed by running the loop in Meta Algorithm~\ref{meta-alg:find-H-from-Q} with $k=N'$ (recovery step for column $N'$).
\end{itemize}
\end{definition}

As we state below, if $H_1, H_2$ are optimal vertex algorithms, then $\cG(H_1) \glue \cG(H_2) = \cG(H_1 \glue H_2)$.
This suggests that $\lambda^\star(H_1 \glue H_2)$ may also be related to $\lambda^\star(H_1)$ and $\lambda^\star(H_2)$ in a structured fashion.
Indeed, this relationship turns out to be delightfully simple, and we isolate it into the following definition, applicable for any optimal methods (not just vertices).

\begin{definition}
\label{definition:gluing-of-lambdas}
Let $1\le N' \le N-1$, and let $H_1 \in \cH_\star(N'-1)$, $H_2 \in \cH_\star(N-N'-1)$. We define the \textit{gluing} of their H-certificates $\lambda^\star(H_1) \glue \lambda^\star(H_2)$ as the set of $\lambda^\star_{k,j}$, defined for any $1\le j < k \le N$, satisfying
\begin{itemize}
    \item For $1\le j < k \le N'$, 
    $\lambda_{k,j}^\star = \frac{N'}{N} \lambda_{k,j}^\star (H_1)$.
    
    \item For $N'+1 \le j < k \le N$, $\lambda_{k,j}^\star = \frac{N}{N-N'} \lambda_{k-N',j-N'}^\star (H_2)$.
    
    \item $\lambda_{N,N'}^\star = \frac{N'}{N-N'}$.

    \item $\lambda_{k,j}^\star = 0$ for $1\le j < N' < k \le N$ and $\lambda_{k,N'}^\star = 0$ for $k=N'+1,\dots,N-1$.
\end{itemize}
\end{definition}

In plain words, $\lambda^\star(H_1) \glue \lambda^\star(H_2)$ is a rescaled version of $\lambda^\star(H_1)$ when constrained to the first $N'-1$ indices, and a rescaled version of $\lambda^\star(H_2)$ when constrained to the last $N-N'-1$ indices.
All the other $\lambda^\star_{k,j}$ are zero, except $\lambda^\star_{N,N'} = \frac{N'}{N-N'}$.

Now we are ready to state the main theorem of this section.

\begin{theorem}[Gluing theorem]
\label{theorem:gluing}
Let $1\le N' \le N-1$, $H_1 \in \cH_\star(N'-1)$, and $H_2 \in \cH_\star(N-N'-1)$.
Then $H_1 \glue H_2 \in \cH_\star(N-1)$ and $\lambda^\star(H_1 \glue H_2) = \lambda^\star(H_1) \glue \lambda^\star(H_2)$.
Moreover, if $H_1, H_2$ are respectively vertices of $\cH_\star(N'-1)$ and $\cH_\star(N-N'-1)$, then $H_1 \glue H_2$ is a vertex of $\cH_\star(N-1)$ and $\cG(H_1 \glue H_2) = \cG(H_1) \glue \cG(H_2)$.

\end{theorem}

An equivalent way of stating \cref{theorem:gluing} is that $H_1 \glue H_2$ admits the convergence proof
\begin{align}
\label{eqn:gluing-proof-template}
\begin{aligned}
    0 & = N \sqnorm{g_N} + \inprod{g_N}{x_N - y_0} + \sum_{1\le j < k \le N'} \frac{N'}{N} \lambda_{k,j}^\star (H_1) \inprod{x_k - x_j}{g_k - g_j} + \frac{N'}{N-N'} \inprod{x_{N} - x_{N'}}{g_N - g_{N'}} \\
    & \quad + \sum_{N'+1 \le j < k \le N} \frac{N}{N-N'} \lambda_{k-N',j-N'}^\star (H_2) \inprod{x_k - x_j}{g_k - g_j} .
\end{aligned}
\end{align}
We prove \cref{theorem:gluing} in Appendix~\ref{section:proof-of-gluing-theorem} by directly verifying that \eqref{eqn:gluing-proof-template} holds for $H_1 \glue H_2$.

\paragraph{Description of the algorithm gluing and its implications.}
The algorithm defined by $H_1 \glue H_2$ can be understood as running the algorithm $H_1$ for the first $N'-1$ steps, then running the intermediate step 
\begin{align*}
    y_{N'} & = y_{N'-1} - 2h_{N',N'} g_{N'} + \frac{N-N'}{N} \sum_{j=1}^{N'-1} \sum_{k=j}^{N'-1} 2\left( H_1 \right)_{k,j} g_j = y_{N'-1} - \frac{2N'(N-N')}{N} g_{N'} + \frac{N-N'}{N} (y_0 - y_{N'-1})
\end{align*}
and finally running the algorithm $H_2$ with modification corresponding to $h_{N'+1,N'}, \dots, h_{N-1,N'}$.
In particular, because the first $N'-1$ steps of $H$ are identical to running $H_1$, which is an optimal $(N'-1)$-iteration algorithm, we obtain the following result.

\begin{corollary}
\label{corollary:intermediate-iterate-guarantee-for-glued-algorithm}
Let $N', H_1, H_2$ be as in \cref{theorem:gluing}. Then the algorithm defined by $H_1 \glue H_2$ satisfies
\[
    \sqnorm{y_{N'-1} - \opT y_{N'-1}} \le \frac{4\sqnorm{y_0 - y_\star}}{(N')^2} .
\]
\end{corollary}

\paragraph{Algorithms and H-certificates of size zero.}
Note that the definitions above and \cref{theorem:gluing} allow for the cases $N'=1$ or $N-N'=1$, which yields one of the following: the arc diagram that has a single node and no edge, the $0\times 0$ matrix $[\,\,]$, or $\lambda^\star\left([\,\,]\right) = \emptyset$.
All three notions of gluing (Definitions~\ref{definition:gluing-of-diagrams}--\ref{definition:gluing-of-lambdas}) are well-defined in these cases; in particular, if $H_1 = [\,\,]$ (\textit{resp}.\ $H_2 = [\,\,]$), the columns below $H_1$ (\textit{resp}.\ rows to the left of $H_2$) in the block structure~\eqref{eqn:H-matrix-gluing} will simply vanish.
In the smallest extreme, we will have $N'=N-N'=1$, and in this case,
\begin{align*}
    [\,\,] \glue [\,\,] = \begin{bmatrix} \frac{1}{2} \end{bmatrix} .    
\end{align*}

\subsection{Decomposable and basic algorithms and diagrams}
\label{subsection:decomposable-and-basic}

Conversely to the previous section where we discussed how to compose algorithms and diagrams, we can also \textit{decompose} algorithms and diagrams. 

\begin{figure}[t]
\centering

\begin{subfigure}[b]{0.47\textwidth}
    \centering
    \begin{tikzpicture}[scale=0.9,
        dot/.style={circle,fill,inner sep=1.8pt},
        every node/.style={font=\large}
    ]
        \foreach \i in {1,...,8}{
            \coordinate (L\i) at (\i,0);
            \fill (L\i) circle (2.2pt);
            \node[below=3pt] at (L\i) {\i};
        }

        \draw[line width=0.8pt, color=gray, densely dashed] (L1) -- (L8);

        \foreach \a/\b in {1/4,2/3,3/4,5/7,6/7,7/8}{
            \pgfmathsetmacro{\rad}{(\b-\a)/2}
            \draw[line width=0.8pt]
                (L\a) arc[start angle=180,end angle=0,radius=\rad];
        }

        \draw[line width=0.8pt, color=red]
            (L4) arc[start angle=180,end angle=0,radius=(8-4)/2];
        
    \end{tikzpicture}
    \caption{A non-crossing arc diagram on $8$ nodes.}
    \label{fig:noncrossing-8}
\end{subfigure}
\hfill
\begin{subfigure}[b]{0.47\textwidth}
    \centering
    \begin{tikzpicture}[scale=0.9,
        dot/.style={circle,fill,inner sep=1.8pt},
        every node/.style={font=\large}
    ]
        \foreach \i in {1,...,8}{
            \coordinate (R\i) at (\i,0);
            \fill (R\i) circle (2.2pt);
            \node[below=3pt] at (R\i) {\i};
        }

        \draw[line width=0.8pt, color=gray, densely dashed] (R1) -- (R8);

        \foreach \a/\b in {3/4,5/7,6/7,7/8}{
            \pgfmathsetmacro{\rad}{(\b-\a)/2}
            \draw[line width=0.8pt]
                (R\a) arc[start angle=180,end angle=0,radius=\rad];
        }

        \foreach \a/\b in {1/4,2/5,4/8}{
            \pgfmathsetmacro{\rad}{(\b-\a)/2}
            \draw[line width=0.8pt, color=blue]
                (R\a) arc[start angle=180,end angle=0,radius=\rad];
        }

    \end{tikzpicture}
    \caption{A crossing arc diagram on $8$ nodes.}
    \label{fig:crossing-8}
\end{subfigure}

\caption{Examples of basic (non-crossing) and non-basic (crossing) arc diagrams on $N=8$ nodes. \textbf{Left:} The diagram is decomposable at $N'=4$. By removing the red arc $(N,N')=(8,4)$, the graph is separated into two subdiagrams, both of which are again decomposable.
\textbf{Right:} Each arc colored in blue crosses another arc, and the diagram is not decomposable.
}
\label{figure:basic-and-crossing-arc-diagrams}
\end{figure}
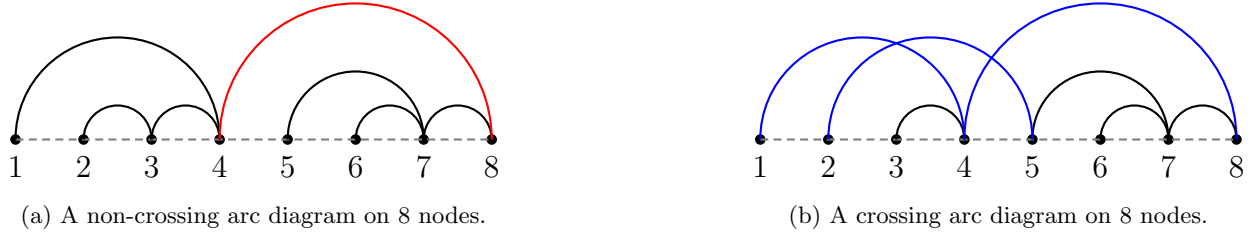

\begin{definition}
\label{definition:decomposable-diagram}
Let $\cG$ be an arc diagram on nodes $\cV = \{1,\dots,N\}$.
If there exists $1\le N' \le N-1$ such that $k(N') = N$ and $k(j) \le N'$ for all $j=1,\dots,N'-1$, then we say $\cG$ is \textit{decomposable at} $N'$.
\end{definition}

\cref{definition:decomposable-diagram} means that removing the arc $(N,N')$ from $\cG$ divides it into two connected components $\cG_{1,N'}$ with nodes $\{1,\dots,N'\}$, and $\cG_{N'+1,N}$ with nodes $\{N'+1, \dots, N\}$ (which can be relabeled to $\{1,\dots,N-N'\}$), so $\cG = \cG_{1,N'} \glue \cG_{N'+1,N}$.

The corresponding notion of decomposability for H-matrices is defined as follows.

\begin{definition}
\label{definition:decomposable-algorithm}
$H \in \cH_\star(N-1)$ is \textit{decomposable at} $N' \in \{1,\dots,N-1\}$ if
\begin{itemize}
    \item $\lambda_{N,N'}^\star (H) > 0$ and $\lambda_{k,N'}^\star (H) = 0$ for $k=N'+1,\dots,N-1$, and
    \item $\lambda_{k,j}^\star (H) = 0$ for all $j=1,\dots,N'-1$ and $k=N'+1,\dots,N$.
\end{itemize}
\end{definition}

From the above definitions, it is immediate that $H$ is an optimal vertex algorithm decomposable at $N'$ if and only if its arc diagram $\cG(H)$ is decomposable at $N'$.
This is not a merely combinatorial property, but it has the following important implication on the algorithm's convergence property.

\begin{corollary}
\label{corollary:intermediate-iterate-guarantee-for-decomposable-algorithm}
If $H$ is a vertex of $\cH_\star(N-1)$ and is decomposable at $N' \in \{1,\dots,N-1\}$, then it has an optimal guarantee on $y_{N'-1}$, i.e., $\sqnorm{y_{N'-1} - \opT y_{N'-1}} \le \frac{4\sqnorm{y_0 - y_\star}}{(N')^2}$.
\end{corollary}

\begin{proof}
Because $H$ is decomposable at $N'$, its arc diagram $\cG=\cG(H)$ is decomposable at $N'$ and $\cG = \cG_{1,N'} \glue \cG_{N'+1,N}$.
Let $H_1$ and $H_2$ be the optimal vertex algorithms corresponding to $\cG_{1,N'}$ and $\cG_{N'+1,N}$.
Then
\[
    \cG(H_1 \glue H_2) = \cG(H_1) \glue \cG(H_2) = \cG_{1,N'} \glue \cG_{N'+1,N} = \cG ,
\]
which implies $H = H_1 \glue H_2$ by \cref{theorem:all-sparsity-patterns-give-optimal-vertex}.
Hence, the conclusion follows from \cref{corollary:intermediate-iterate-guarantee-for-glued-algorithm}.
\end{proof}

Next, we characterize the maximally decomposable vertex algorithms (\textit{resp}.\ diagrams), which can be recursively decomposed until one reaches $0\times 0$ blocks $[\,\,]$ (\textit{resp}.\ single-node diagrams).
They can be defined recursively as follows.

\begin{definition}
\label{definition:basic-matrix-and-diagram}
We define the $0\times 0$ empty matrix $[\,\,]$ to be basic.
Recursively, we define any H-matrix as \textit{basic} if it is a gluing of smaller basic H-matrices.
Analogously, we define the single node diagram to be basic, and recursively define any arc diagram as basic if it is a gluing of smaller basic diagrams.
\end{definition}

\paragraph{Connection to prior work.}
This definition is reminiscent of the recent framework for composing the step-sizes for gradient descent on smooth convex functions~\citep{grimmerComposingOptimizedStepsize, zhangAcceleratedGradientDescent2026}, which appears in the context of accelerating gradient descent via careful step-size selection \citep{grimmerProvablyFasterGradient2024, altschulerAccelerationStepsizeHedging2025a, altschulerAccelerationStepsizeHedging2025, grimmerAcceleratedObjectiveGap2025}.
In that setting, one-dimensional step-size schedule vectors can be composed using distinct join operations (the so-called $f$-, $g$- or $s$-joins) depending on the performance measure of interest.
In contrast, we glue full H-matrices in a way that preserves the entire proof structure, yielding another H-matrix that is optimal with respect to the same measure of convergence.
In this sense, our result can be viewed as a structurally richer and more symmetric analogue of step-size composition in \citep{grimmerComposingOptimizedStepsize, zhangAcceleratedGradientDescent2026}, and exploring this connection further seems to be an interesting future direction.\medskip

Definition~\ref{definition:basic-matrix-and-diagram} of basic H-matrices and arc diagrams, as written, is not explicit. However, there is the following simple and explicit combinatorial characterization through arc diagrams (also see Figure~\ref{figure:basic-and-crossing-arc-diagrams}).

\begin{proposition}
\label{proposition:basic-iff-noncrossing}
$H \in \reals^{(N-1)\times (N-1)}$ is basic if and only if it is a vertex of $\cH_\star(N-1)$ and its arc diagram $\cG(H)$ is non-crossing, i.e., there is no pair of indices $(j_1, j_2)$ such that $j_1 < j_2 < k(j_1) < k(j_2)$.
\end{proposition}

To prove this result, we first introduce a handy lemma characterizing the non-crossing property.
\begin{lemma}
\label{lemma:leftmost-arc-to-N-gives-decomposition}
If $\cG$ is a non-crossing arc diagram on $\{1,\dots,N\}$ and
\[
    N' = \min\{j\in\{1,\dots,N-1\}: k(j)=N\} 
\]
is the leftmost vertex joined directly to $N$, then $\cG$ is decomposable at $N'$.
Furthermore, this is the unique index $i \in \{1,\dots,N-1\}$ at which $\cG$ is decomposable.
\end{lemma}

\begin{proof}
By definition of $N'$, we have $k(N')=N$.
Now we claim that $k(j)\le N'$ for $j=1,\dots,N'-1$, which will immediately show that $\cG$ is decomposable at $N'$. 
Assume to the contrary that $k(j)>N'$ for some $1\le j < N'$. 
By minimality of $N'$ among the nodes that are directly connected to $N$, we must have $k(j)<N$. Therefore
\[
    j < N' < k(j) < N = k(N'),
\]
which contradicts the non-crossing assumption.

Conversely, if $\cG$ is decomposable at $i$, then we must have $k(i) = N$ and $k(j) \le i < N$ for $j < i$, which implies $i = \min\{j\in\{1,\dots,N-1\}: k(j)=N\} = N'$.
\end{proof}

\begin{proof}[Proof of \cref{proposition:basic-iff-noncrossing}]

We argue by induction on $N$.
The base case is $N=1$, when $H = [\,\,]$ is the $0\times 0$ empty matrix and $\cG$ is a singleton graph with no edge. In this case, there is nothing to prove.
Now let $N\ge 2$ and suppose that the equivalence holds for all H-matrices of size strictly less than $N-1$.
Let $H \in \reals^{(N-1)\times (N-1)}$.

\smallskip
\noindent $[\Longrightarrow]$ Assuming that $H$ is basic, we show that its arc diagram $\cG$ is non-crossing. By definition, $H = H_1 \glue H_2$ for some smaller basic H-matrices $H_1 \in \reals^{(N_1-1)\times(N_1-1)}$ and $H_2 \in \reals^{(N_2-1)\times(N_2-1)}$, where $N_1+N_2=N$.
By the induction hypothesis, $H_1$ and $H_2$ are optimal vertex algorithms whose arc diagrams $\cG_1 = \cG(H_1)$ and $\cG_2 = \cG(H_2)$ are non-crossing.
Then by \cref{theorem:gluing}, $H = H_1\glue H_2$ is an optimal vertex algorithm corresponding to the arc diagram $\cG = \cG_1 \glue \cG_2$. 
Now all edges of $\cG$ are one of the following three types:
\begin{itemize}
    \item edges within the left block $\cG_1$
    \item edges within the right block $\cG_2$ 
    \item the gluing arc joining the rightmost vertex of $\cG_1$ to the rightmost vertex of $\cG$.
\end{itemize}
Two edges inside the same block do not cross because $\cG_1$ and $\cG_2$ are non-crossing.
Now observe that an edge in the left block has both endpoints in $\{1, \dots, N_1\}$, whereas an edge in the right block has both endpoints in $\{N_1+1, \dots, N_1+N_2 = N\}$, which shows that a left-block edge, a right-block edge and the gluing arc $(N,N_1)$ cannot cross each other.
This proves that $\cG$ is non-crossing.

\smallskip
\noindent$[\Longleftarrow]$
Conversely, now we assume that $H$ is an optimal vertex algorithm with a non-crossing arc diagram $\cG$, and show that $H$ is basic. Let $N'$ be as in \cref{lemma:leftmost-arc-to-N-gives-decomposition}, so that $\cG$ is decomposable at $N'$.
Then, removing the arc $(N,N')$ separates $\cG$ into two subdiagrams $\cG_{1,N'}$ and $\cG_{N'+1,N}$, and they are both non-crossing because $\cG$ is.
By \cref{theorem:all-sparsity-patterns-give-optimal-vertex}, $\cG_{1,N'}$ and $\cG_{N'+1,N}$ uniquely correspond to optimal vertex algorithms $H_1 \in \cH_\star(N'-1)$ and $H_2 \in \cH_\star(N-N'-1)$.
Then, $H_1$ and $H_2$ are basic by the induction hypothesis, and because $H$ is the unique vertex algorithm corresponding to $\cG = \cG_{1,N'} \glue \cG_{N'+1,N}$, we have $H = H_1 \glue H_2$ by \cref{theorem:gluing}.
Therefore $H$ is basic.

This completes the induction, and hence the proof.
\end{proof}

OHM is a basic algorithm, obtained by recursively gluing $[\,\,]$ to the right end. More precisely, we have
\begin{align*}
    H_\text{OHM}(k+1) = H_\text{OHM}(k) \glue [\,\,] \text{ for } k \ge 0 \implies H_\text{OHM}(N-1) = (( \cdots (([\,\,] \glue [\,\,]) \glue [\,\,]) \glue \dots )\glue [\,\,]) \glue [\,\,]
\end{align*}
where we adopt the convention $H_\text{OHM}(0) = [\,\,]$, and $\glue$ is repeated $N-1$ times in the last expression.
On the other hand, Dual-OHM is also basic, and it is obtained by recursively gluing $[\,\,]$ to the left end, i.e.,
\begin{align*}
    H_\text{Dual-OHM}(k+1) = [\,\,] \glue H_\text{Dual-OHM}(k) \text{ for } k \ge 0 \implies H_\text{Dual-OHM}(N-1) = [\,\,] \glue ( [\,\,] \glue ( \dots \glue ([\,\,] \glue ([\,\,] \glue [\,\,])) \cdots )) 
\end{align*}
where we similarly let $H_\text{Dual-OHM}(0) = [\,\,]$.

So far, we have defined and studied basic algorithms from a purely combinatorial motivation.
However, in \cref{section:duality-via-diagrams}, we show that the non-crossing property of arc diagrams is in fact inherently related to the concept of H-duality \citep{KimOzdaglarParkRyu2023_timereversed, YoonKimSuhRyu2024_optimal, YoonRyuGrimmerInvariance2025}, which establishes a connection between the proofs of algorithms that are in an H-dual relationship.
At the moment, we continue the combinatorial exploration and add another result on basic diagrams, which will be revisited later.

\begin{proposition}
\label{proposition:catalan-many-basic-arc-diagrams}
The numbers of basic (non-crossing) arc diagrams on $N$ nodes and of the corresponding vertex algorithms are both the Catalan number $C_{N-1} = \frac{1}{N}\binom{2N-2}{N-1}$.
\end{proposition}

\begin{proof}
\cref{lemma:leftmost-arc-to-N-gives-decomposition} shows that a non-crossing arc diagram $\cG$ has a unique splitting point $N'$, and then we have $\cG=\cG_1\glue \cG_2$ where $\cG_1$ and $\cG_2$ are basic arc diagrams on $N'$ and $N-N'$ nodes, respectively.
Conversely, for each $N'\in\{1,\dots,N-1\}$ and any pair $(\cG_1,\cG_2)$ of basic arc diagrams on $N'$ and $N-N'$ nodes, the gluing $\cG_1\glue \cG_2$ injectively yields a basic arc diagram on $N$ nodes. 
Therefore, if $a_N$ denotes the number of basic arc diagrams with $N$ nodes, then $a_1=1$ and 
\[
a_N=\sum_{N'=1}^{N-1} a_{N'}a_{N-N'}
\]
for $N=2,3,\dots$, which is exactly the Catalan recurrence with index shifted by 1. Thus $a_N = C_{N-1} = \frac{1}{N}\binom{2N-2}{N-1}$.
\end{proof}

\paragraph{Intermediate-iterate guarantees for basic algorithms.}
If $H$ is a basic algorithm, then we can recursively apply \cref{corollary:intermediate-iterate-guarantee-for-decomposable-algorithm} and obtain $\sqnorm{y_{j-1} - \opT y_{j-1}} \le \frac{4\sqnorm{y_0 - y_\star}}{j^2}$ for all $j \in \cC(1) = \{1, k(1), k(k(1)), \dots, N\}$. 
That is, we have optimal guarantees for intermediate iterates along the unique increasing path from $1$ to $N$, which are the ``overarching bumps'' in the non-crossing arc diagram $\cG(H)$.
In the case of OHM, we have $\cC(1) = \{1, 2, \dots, N\}$ so we have optimal guarantees on all intermediate iterates, which aligns with the fact that OHM is (the only) anytime optimal algorithm \citep{YoonRyuGrimmerInvariance2025}.
On the other hand, in the case of Dual-OHM, we have $\cC(1) = \{1, N\}$, and the guarantee $\sqnorm{y_0 - \opT y_0} \le 4\sqnorm{y_0 - y_\star}$ for $N'=1$ is vacuous, so it guarantees optimality only at the last iterate $y_{N-1}$.
Interestingly, however, stronger guarantees on intermediate iterates appear to come at the expense of the algorithm's robustness, in the sense described in Section~\ref{subsection:robustness}.

\section{Characterization of H-duality via arc diagrams}
\label{section:duality-via-diagrams}

H-dual \citep{KimFessler2021_optimizing, KimOzdaglarParkRyu2023_timereversed, KimParkOzdaglarDiakonikolasRyu2023_mirror} is an operation that takes the anti-diagonal transpose of H-matrices, and it naturally preserves the H-invariance conditions~\eqref{eqn:H-invariance} \citep{YoonKimSuhRyu2024_optimal}.
On the other hand, there exist optimal H-matrices whose H-duals are not optimal, as H-certificates may become negative when $H \in \cH_\star(N-1)$ is dualized.
In particular, when $N=4$ we have $(N-1)! = 6$ optimal vertex algorithms, and among them, there are two H-dual pairs, one is self-dual, and the remaining one becomes suboptimal under H-dual (i.e., $H^\at \notin \cH_\star(N-1)$) \citep{YoonRyuGrimmerInvariance2025}.

In this section, we thoroughly study the effect of H-dual on optimal vertex algorithms.
Based on the explicit characterization of the mapping between the H-certificates $\lambda^\star (H)$ and $\lambda^\star (H^\at)$ for $H \in \cH_\star(N-1)$, we prove that an optimal vertex algorithm stays optimal if and only if its arc diagram $\cG$ is basic, or equivalently non-crossing.
When this is the case, $H^\at$ is also a vertex of $\cH_\star(N-1)$, its arc diagram $\cG\left(H^\at\right)$ can be directly computed from $\cG$ via a simple recursive procedure, and nonzero $\lambda^\star_{k,j}(H^\at)$ are given by inverses of nonzero H-certificates $\lambda^\star(H)$ of $H$ under a combinatorial index correspondence.

\subsection{Transformation of H-certificates under H-dual}

Let $H \in \cH_\star(N-1)$ be an optimal vertex algorithm, and let $\cG(H) = (\cV, \cE, \lambda^\star(H))$ be its arc diagram, weighted by $\lambda^\star(H)$.
Within this subsection, we denote $\lambda_{k,j}^\star (H) = \lambda_{k,j}$ and $\cG = \cG(H)$ for simplicity.
Let $L(\cG)$ be the graph Laplacian of $\cG$, i.e., 
\begin{align*}
    L(\cG) = D(\cG) - A(\cG)
\end{align*}
where $D(\cG)$ is the diagonal degree matrix
\begin{align*}
    D(\cG) = \mathrm{diag}\left( \lambda_{2,1} + \lambda_{3,1} + \dots + \lambda_{N,1} , \lambda_{2,1} + \lambda_{3,2} + \dots + \lambda_{N,2} , \dots , \lambda_{N,1} + \lambda_{N,2} + \cdots + \lambda_{N,N-1} \right)
\end{align*}
which is a diagonal matrix whose $(j,j)$-entry is $\sum_{k=1}^{j-1} \lambda_{j,k} + \sum_{k=j+1}^N \lambda_{k,j}$ and $A(\cG)$ is the adjacency matrix
\begin{align*}
    A(\cG) = \begin{bmatrix}
        0 & \lambda_{2,1} & \lambda_{3,1} & \cdots & \lambda_{N,1} \\
        \lambda_{2,1} & 0 & \lambda_{3,2} & \cdots & \lambda_{N,2} \\
        \lambda_{3,1} & \lambda_{3,2} & 0 & \cdots & \lambda_{N,3} \\
        \vdots & \vdots & \vdots & \ddots & \vdots \\ 
        \lambda_{N,1} & \lambda_{N,2} & \lambda_{N,3} & \cdots & 0
    \end{bmatrix} .
\end{align*}
Then, let $\Lambda(H) = -E_{N,N} - L(\cG) \in \reals^{N \times N}$, where $e_N = \begin{bmatrix} 0 & \cdots & 0 & 1 \end{bmatrix}^\transpose \in \reals^N$ and $E_{N,N} = e_N e_N^\transpose$.
In other words, $\Lambda(H)$ is the negative graph Laplacian with $1$ additionally subtracted from the bottom-rightmost entry:
\begin{align*}
    \Lambda(H) = \begin{bmatrix}
        -\lambda_{2,1} - \lambda_{3,1} - \dots -\lambda_{N,1} & \lambda_{2,1} & \cdots & \lambda_{N,1} \\
        \lambda_{2,1} & -\lambda_{2,1} - \lambda_{3,2} - \dots -\lambda_{N,2} & \cdots & \lambda_{N,2} \\
        \vdots & \vdots & \ddots & \vdots \\
        \lambda_{N,1} & \lambda_{N,2} & \cdots & -\lambda_{N,1} - \lambda_{N,2} - \cdots - \lambda_{N,N-1} - 1
    \end{bmatrix} .
\end{align*}
In fact, $\Lambda(H)$ can be defined algebraically for any lower-triangular $H \in \reals^{(N-1)\times (N-1)}$ satisfying H-invariance~\eqref{eqn:H-invariance}, with possibly negative H-certificates as off-diagonal entries.
With that understanding, we state:

\begin{theorem}[\citep{shuUnifiedTreatmentDuality2026}]
\label{theorem:lambda-of-H-dual}
For any $H \in \cH_\star(N-1)$, $\Lambda(H)$ is invertible and $\Lambda(H^\at) = \Delta \Lambda(H)^{-1} \Delta$ where
\[
    \Delta = \begin{bmatrix} 
        {} & {} & -1 & 1 \\
        & \iddots & 1 \\
        -1 & \iddots \\
        1 &
    \end{bmatrix} \in \reals^{N \times N}
\]
is the matrix of discrete derivatives with order reversing. 
\end{theorem}

This is a particular case of \citep{shuUnifiedTreatmentDuality2026} restricted to $H \in \cH_\star(N-1)$.
While their result is for generic H-matrices, it is generally not guaranteed that $\Lambda(H)$ is invertible, so they state it as an assumption.
In our case, invertibility of $\Lambda(H)$ follows from the fact that $L(\cG)$ is the Laplacian of the weighted connected graph $\cG$, so it is positive semidefinite and $\ker L(\cG) = \mathrm{Span}\{\mathbf{1}\}$, which implies that $-\Lambda(H) = L(\cG) + E_{N,N}$ is positive definite.
For completeness, we present the self-contained proof of \cref{theorem:lambda-of-H-dual} in Appendix~\ref{section:proof-of-lambda-inversion-theorem}.

A relevant line of prior work is the so-called H-duality theorems \citep{KimOzdaglarParkRyu2023_timereversed, KimParkOzdaglarDiakonikolasRyu2023_mirror, YoonKimSuhRyu2024_optimal}, which establish an explicit correspondence between convergence proofs for an algorithm and for its H-dual.
However, these results were limited to specific proof structures that use only a restricted subset of inequalities, and thus are not directly applicable to proving \cref{theorem:lambda-of-H-dual}.
A more closely related work in spirit is \citep{kimProofExactConvergence2024}, which develops a general correspondence between proof certificates of H-dual algorithms, albeit in the setting of smooth convex minimization.
From this perspective, \cref{theorem:lambda-of-H-dual} can be viewed as a completion of H-duality theory for exact optimal fixed-point algorithms, as it shows that $H$ has an optimal H-dual if and only if $\Delta\Lambda(H)^{-1}\Delta$ is nonnegative off its diagonal, and those values are the proof certificates for $H^\at$.
On the other hand, the fully general H-duality, incorporating suboptimal algorithms such as Krasnosel'ski\u{\i}--Mann iteration \citep{Krasnoselskii1955_two, Mann1953_mean}, remains open.
Toward this end, the result of \citep{shuUnifiedTreatmentDuality2026} characterizes a correspondence between proof certificates of generic $H$ and $H^\at$ under a broad class of proof structures that often captures the tightest (best rate) proof associated with $H$.

For optimal vertex algorithms, the mapping from $\Lambda(H)$ to $\Lambda(H^\at)$ in \cref{theorem:lambda-of-H-dual} becomes more interpretable.

\subsubsection{H-duality for optimal vertex algorithms}
\label{section:H-duality-for-vertices}

To state \cref{theorem:dual-optimality-and-non-crossing-arc-diagram}, the main result of this section, first we briefly introduce the following handy graph-theoretic notion: given an arc diagram $\cG$ and its node $v$, define $\ell(v) = \min \{j\,|\,v\in \cC(j)\}$. 
In other words, $\ell(v)$ is the smallest descendant of $v$, where descendant is defined as a node that is joined with $v$ via the increasing path.

\begin{theorem}
\label{theorem:dual-optimality-and-non-crossing-arc-diagram}
Let $H \in \cH_\star(N-1)$ be an optimal vertex algorithm.
Then $H^\at \in \cH_\star(N-1)$ if and only if the arc diagram $\cG(H)$ is non-crossing (basic).
Additionally, when this is the case, $H^\at$ is also a vertex algorithm with, for $1\le j < i \le N$,
\begin{align*}
    \lambda_{N-j+1,N-i+1}^\star (H^\at) = \begin{cases}
        \frac{1}{\lambda_{k(i-1),i-1}^\star (H)} & \text{if } j = \ell(i-1) \\
        0 & \text{otherwise.}
    \end{cases}
\end{align*}
\end{theorem}

By \cref{proposition:catalan-many-basic-arc-diagrams}, this immediately implies:
\begin{corollary}
\label{corollary:catalan-many-dual-optimal-algorithms}
The number of vertex algorithms in $\cH_\star(N-1)$ that remain optimal under H-dual is $C_{N-1} = \frac{1}{N} \binom{2N-2}{N-1}$. 
\end{corollary}

While the full proof of \cref{theorem:dual-optimality-and-non-crossing-arc-diagram} will be presented in Section~\ref{section:proof-of-dual-optimality-and-basic-diagram-theorem}, here we briefly sketch the key ideas.
By \cref{theorem:lambda-of-H-dual}, for $H\in \cH_\star(N-1)$, if we let
\[
    \Lambda(H)^{-1} = \left( \nu_{i,j} \right)_{\substack{1\le i \le N \\ 1\le j \le N}}
\]
then we have
\begin{align}
\label{eqn:lambda-H-dual-explicit}
    \Lambda(H^\at) & = \Delta \Lambda(H)^{-1} \Delta \nonumber \\
    & = 
    \begin{bmatrix}
        (\nu_{N,N} - \nu_{N-1,N}) - (\nu_{N,N-1} - \nu_{N-1,N-1}) & \cdots & (\nu_{N,2} - \nu_{N-1,2}) - (\nu_{N,1} - \nu_{N-1,1}) & \nu_{N,1} - \nu_{N-1,1} \\
        \vdots & \ddots & \vdots & \vdots \\
        (\nu_{2,N} - \nu_{1,N}) - (\nu_{2,N-1} - \nu_{1,N-1}) & \cdots & (\nu_{2,2} - \nu_{1,2}) - (\nu_{2,1} - \nu_{1,1}) & \nu_{2,1} - \nu_{1,1} \\
        \nu_{1,N} - \nu_{1,N-1} & \cdots & \nu_{1,2} - \nu_{1,1} & \nu_{1,1}
    \end{bmatrix} .
\end{align}
Here $\nu_{i,j} = \nu_{j,i}$ comes from the inverse of the graph Laplacian with a single entry offset by $1$.
Hence, by borrowing standard tools from graph theory, we can express these values using certain subgraphs of $\cG$. 
If $H$ is a vertex of $\cH_\star(N-1)$, then $\Lambda(H^\at)$ admits a particularly clear graph-theoretic interpretation, and this enables us to characterize which arcs should have positive weights in the arc diagram of $H^\at$, and to compute those weights.
In fact, the characterization of $\cG^\at = \cG(H^\at)$ from $\cG = \cG(H)$ is given as follows.

\begin{corollary}
\label{corollary:dual-of-decomposable-basic-diagram}
Let $\cG$ be a basic arc diagram on $\{1,\dots,N\}$, and suppose that $\cG$ is decomposable at $N' \in \{1,\dots,N-1\}$.
Then $\cG^\at = \cG_{N'+1,N}^\at \glue \cG_{1,N'}^\at$.
In other words, H-dual acts on a basic diagram as respectively dualizing the two components obtained by removing the arc $(N,N')$, reversing their order, and gluing them back together.
\end{corollary}

\cref{corollary:dual-of-decomposable-basic-diagram} yields a concise implementation for computing the dual arc diagram $\cG^\at$ via recursion (Meta Algorithm~\ref{meta-alg:basic-diagram-dualization}).
Note that in Meta Algorithm~\ref{meta-alg:basic-diagram-dualization}, the recursion terminates and returns $\cG^\at$ in the base cases for which the dual is already known. 
In addition to the singleton case $N=1$, one may conveniently take the OHM and Dual-OHM families as base cases, as their duals are simply Dual-OHM and OHM, respectively (see Figure~\ref{figure:recursive-dualization-example}).
We provide the proof of \cref{corollary:dual-of-decomposable-basic-diagram} in Section~\ref{section:proof-of-diagram-dualization-corollary}.

\begin{algorithm}[t]
\caption{Meta algorithm for dualizing a basic arc diagram}
\label{meta-alg:basic-diagram-dualization}
\begin{algorithmic}[1]
\Function{Dualize}{$\cG$}
    \State $N \gets |\cV(\cG)|$
    \If{$\cG$ belongs to a chosen base family with known dual (e.g., for $N=1$, $\cG^\at = \cG$)}
        \State \Return $\cG^\at$
    \EndIf
    \State Find the unique $N' \in \{1,\dots,N-1\}$ such that
    \State \hspace{1.5em} $k(N')=N$ and $k(j)\le N'$ for all $j=1,\dots,N'-1$
    \State \Return $\Call{Dualize}{\cG_{N'+1,N}} \glue \Call{Dualize}{\cG_{1,N'}}$
\EndFunction
\end{algorithmic}
\end{algorithm}

\newcommand{\ArcBlock}[6]{%
\begin{scope}[shift={({#2},{#3})}]
  \foreach \i in {1,...,#4}{
    \coordinate (#1n\i) at (0.9*\i,0);
    \fill (#1n\i) circle (1.6pt);
    \node[below=2pt,font=\scriptsize] at (#1n\i) {\i};
  }
  \draw[line width=0.6pt, color=gray, densely dashed] (#1n1) -- (#1n#4);
  \foreach \a/\b/\clr in {#5}{
    \pgfmathsetmacro{\rad}{0.45*(\b-\a)}
    \draw[line width=0.8pt,\clr] (#1n\a)
      arc[start angle=180,end angle=0,radius=\rad];
  }
  \coordinate (#1mid) at ($(#1n1)!0.5!(#1n#4)$);
  \coordinate (#1top) at ($(#1mid)+(0,1.00)$);
  \coordinate (#1bot) at ($(#1mid)+(0,-0.58)$);
  \node[align=center,font=\scriptsize] at ($(#1mid)+(0,-1.23)$) {#6};
\end{scope}
}

\begin{figure}[!p]
\thispagestyle{empty}
\centering

\begin{minipage}[c][\textheight][c]{\textwidth}
\centering

\begin{tikzpicture}[
    every node/.style={font=\small},
    splitarrow/.style={
      gray!70, line width=0.55pt,
      -{Latex[length=2.2mm,width=1.7mm]},
      shorten >=1pt, shorten <=1pt
    },
    dualarrow/.style={
      gray!85, line width=0.75pt, densely dashed,
      -{Triangle[open,length=2.6mm,width=2.1mm]},
      shorten >=1pt, shorten <=1pt
    },
    gluearrow/.style={
      gray!70, line width=0.55pt,
      -{Stealth[length=2.0mm,width=1.6mm]},
      shorten >=1pt, shorten <=1pt
    }
]

\def\yA{0.0}
\def\yB{-3.6}
\def\yC{-7.2}
\def\yD{-10.8}
\def\yE{-14.4}
\def\yF{-18.0}

\node[anchor=east] at (-0.5,\yA) {Input $\cG = (\cG_a) \glue (\cG_b \glue \cG_c)$};
\ArcBlock{Top}{1.2}{\yA}{10}{
1/2/black,2/3/black,3/4/black,4/10/red,
5/7/black,6/7/black,7/10/black,8/9/black,9/10/black}

\node[anchor=east] at (-0.5,\yB) {Split $\cG$ into $\cG_b \glue \cG_c$ and $\cG_a$};
\ArcBlock{GR}{0.5}{\yB}{6}{
1/3/black,2/3/black,3/6/red,4/5/black,5/6/black}{}
\coordinate (GRtop) at ($(GRmid)+(0,1.40)$);

\ArcBlock{GL}{8.1}{\yB}{4}{
1/2/black,2/3/black,3/4/black}{}
\coordinate (GLtop) at ($(GLmid)+(0,1.4)$);

\node[anchor=east] at (-0.5,\yC) {Split $\cG_b \glue \cG_c$ into $\cG_c$ and $\cG_b$};
\ArcBlock{H3a}{0.6}{\yC}{3}{
1/2/black,2/3/black}{}
\coordinate (H3atop) at ($(H3amid)+(0,1.1)$);

\ArcBlock{H3b}{4.3}{\yC}{3}{
1/3/black,2/3/black}{}
\coordinate (H3btop) at ($(H3bmid)+(0,1.1)$);

\ArcBlock{H4}{8.1}{\yC}{4}{
1/2/black,2/3/black,3/4/black}{}
\coordinate (H4top) at ($(H4mid)+(0,0.7)$);

\node[anchor=east] at (-0.5,\yD) {Dualize base cases $\cG_c,\cG_b,\cG_a$};
\ArcBlock{D3a}{0.6}{\yD}{3}{
1/3/black,2/3/black}{}
\coordinate (D3atop) at ($(D3amid)+(0,1.1)$);

\ArcBlock{D3b}{4.3}{\yD}{3}{
1/2/black,2/3/black}{}
\coordinate (D3btop) at ($(D3bmid)+(0,0.7)$);

\ArcBlock{D4}{8.1}{\yD}{4}{
1/4/black,2/4/black,3/4/black}{}
\coordinate (D4top) at ($(D4mid)+(0,1.6)$);

\node[anchor=east] at (-0.5,\yE) {Glue to obtain $\cG_c^\at \glue \cG_b^\at$};
\ArcBlock{GRA}{0.5}{\yE}{6}{
1/3/black,2/3/black,3/6/red,4/5/black,5/6/black}{}
\coordinate (GRAtop) at ($(D3amid)-(0,2.5)$);

\ArcBlock{GLA}{8.1}{\yE}{4}{
1/4/black,2/4/black,3/4/black}{}
\coordinate (GLAtop) at ($(GLAmid)+(0,1.6)$);

\node[anchor=east] at (-0.5,\yF) {Glue to obtain $\cG^\at = (\cG_c^\at \glue \cG_b^\at) \glue \cG_a^\at$};
\ArcBlock{Bot}{1.2}{\yF}{10}{
1/3/black,2/3/black,3/6/black,4/5/black,5/6/black,
6/10/red,7/10/black,8/10/black,9/10/black}


\coordinate (TopLsrc) at ($(Topn1)!0.5!(Topn4) + (0,-0.5)$);
\coordinate (TopRsrc) at ($(Topn5)!0.5!(Topn10) + (0,-0.5)$);

\coordinate (GRLsrc) at ($(GRn1)!0.5!(GRn3) + (0,-0.5)$);
\coordinate (GRRsrc) at ($(GRn4)!0.5!(GRn6) + (0,-0.5)$);

\coordinate (GRAinL) at ($(GRAn1)!0.5!(GRAn3) + (0,1.4)$);
\coordinate (GRAinR) at ($(GRAn4)!0.5!(GRAn6) + (0,0.5)$);

\coordinate (BotInL) at ($(Botn1)!0.5!(Botn6) + (0,1.5)$);
\coordinate (BotInR) at ($(Botn7)!0.5!(Botn10) + (0,1.5)$);

\draw[splitarrow]
  (TopLsrc) .. controls +(0,-1.15) and +(0,1.10) .. (GLtop);
\draw[splitarrow]
  (TopRsrc) .. controls +(0,-1.15) and +(0,1.10) .. (GRtop);

\draw[splitarrow]
  (GRLsrc) .. controls +(0,-1.05) and +(0,1.00) .. (H3btop);
\draw[splitarrow]
  (GRRsrc) .. controls +(0,-1.05) and +(0,1.00) .. (H3atop);

\draw[splitarrow] (GLbot) -- (H4top);

\draw[dualarrow] (H3abot) -- (D3atop);
\draw[dualarrow] (H3bbot) -- (D3btop);
\draw[dualarrow] (H4bot)  -- (D4top);

\draw[gluearrow]
  (D3abot) -- (GRAtop);
\draw[gluearrow]
  (D3bbot) .. controls +(0,-0.90) and +(0,1.22) .. (GRAinR);
\draw[gluearrow]
  (D4bot) -- (GLAtop);

\draw[gluearrow]
  (GRAbot) .. controls +(0,-0.50) and +(-0.70,1.00) .. (BotInL);
\draw[gluearrow]
  (GLAbot) .. controls +(0,-1.00) and +(0.70,1.00) .. (BotInR);

\end{tikzpicture}

\caption{An example of recursive dualization of a basic arc diagram with $N=10$, using a version of Meta Algorithm~\ref{meta-alg:basic-diagram-dualization} where OHM and Dual-OHM are treated as base cases. 
Here $\cG_a, \cG_b, \cG_c$ each denote the subdiagrams $\cG_{1,4}$, $\cG_{5,7}$ and $\cG_{8,10}$, which respectively correspond to OHM, Dual-OHM and OHM.
The crossing arrows indicate decomposition steps in which the order of the two components is reversed, dashed arrows indicate dualization, and solid non-crossing arrows indicate backward gluing steps. 
Whenever present, the red arc marks the arc used for the current decomposition or gluing step.}
\label{figure:recursive-dualization-example}

\end{minipage}
\end{figure}
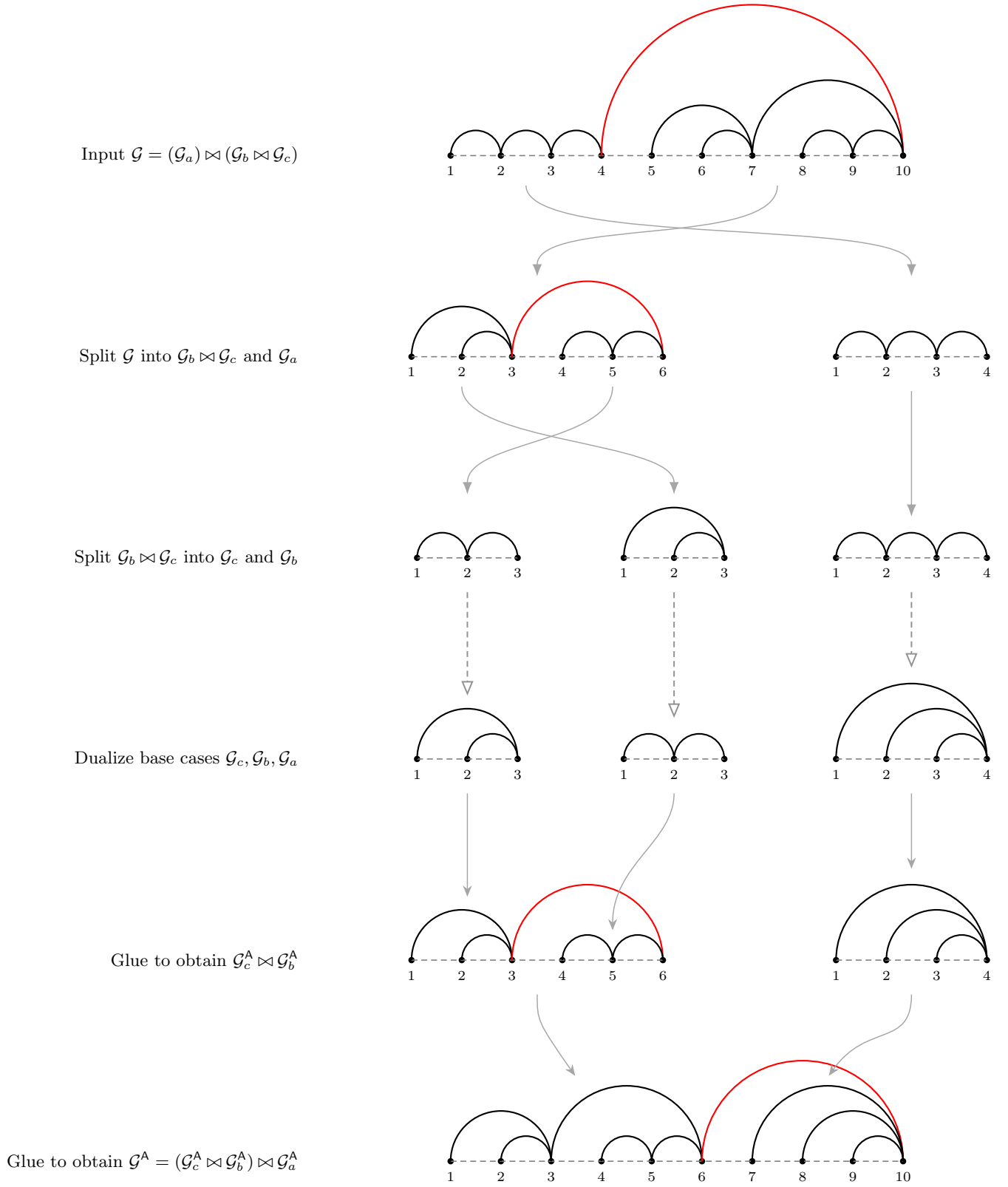

\subsection{Proof of H-duality for optimal vertex algorithms}

\subsubsection{A brief review of graph theory}
\label{subsubsection:graph-theory}

For $H\in \cH_\star(N-1)$, consider its weighted diagram $\cG(H) = (\cV,\cE,\lambda^\star(H))$ as defined in \cref{definition:diagram-of-algorithm}.
Recall that a subgraph $\cT$ is a \textit{tree} if it is connected and acyclic, and a subgraph $\cF$ is a \textit{forest} if it is a disjoint union of trees (possibly including singletons).
If $\cT$ (\textit{resp.}\ $\cF$) contains every node in $\cV$ then $\cT$ (\textit{resp.}\ $\cF$) is a \textit{spanning tree} (\textit{resp.\ spanning forest}) of $\cG(H)$.
For convenience, we will simply denote $\cG(H)$ by $\cG$ and $\lambda_{k,j}^\star(H)$ by $\lambda_{k,j}$ in the following, if it is clear from the context.

We will use a classical but fundamental result characterizing the minors of the Laplacian $L(\cG)$, and eventually, the entries of $\Lambda(H)^{-1}$.
Let $\cU, \cW \subset \cV$ be subsets of vertices of the same size $r$ (i.e.\ $|\cU| = |\cW| = r$), for some $1\le r\le N-1$. 
Let $\mathfrak{F}_{\cU,\cW}$ be the set of all spanning forests of $\cG$ satisfying
\begin{itemize}
    \item Each $\cF \in \mathfrak{F}_{\cU,\cW}$ has $r$ disjoint connected components (trees).
    \item Each component $\cT$ of $\cF$ contains exactly one vertex in $\cU$ and exactly one vertex in $\cW$.
\end{itemize}
Each $\cF \in \mathfrak{F}_{\cU,\cW}$ defines a bijection $\pi_\cF: \cU \to \cW$ by 
\[
    \pi_\cF(u) = \left( \text{the unique }w \in \cW \text{ in the component } \cT \text{ of }\cF \text{ containing } u \right)
    .
\]
We define $\mathrm{sign}(\cF) = (-1)^{n(\pi_\cF)}$ where $n(\pi_\cF)$ is the number of inversions in $\pi_\cF$, i.e., pairs $(u_1,u_2) \in \cU \times \cU$ such that $u_1 < u_2$ and $\pi_\cF(u_1) > \pi_\cF(u_2)$.

\begin{lemma}[All minors matrix tree theorem {\cite{ChaikenCombinatorial1982}}]
\label{lemma:all-minors-matrix-tree-theorem}
Let $\cU,\cW\subset \cV$ with $|\cU|=|\cW|=r$, and write
\[
\Sigma(\cU):=\sum_{u\in\cU}u,
\qquad
\Sigma(\cW):=\sum_{w\in\cW}w.
\]
Denote by $L(\cG)_{\cV\setminus\cU, \cV\setminus\cW}$ the $(N-r)\times (N-r)$ submatrix (minor matrix) of $L(\cG)$ obtained by eliminating the rows indexed by $\cU$ and columns indexed by $\cW$.
Then
\[
\det L(\cG)_{\cV\setminus\cU,\cV\setminus\cW}
=
(-1)^{\Sigma(\cU)+\Sigma(\cW)}
\sum_{\cF\in\mathfrak F_{\cU,\cW}}
\mathrm{sign}(\cF)\,\lambda(\cF),
\]
where $\mathrm{sign}(\cF)=(-1)^{n(\pi_\cF)}$ and $\lambda(\cF) = \prod_{(i,j) \in \cE_\cF} \lambda_{i,j}$.
\end{lemma}

Applying \cref{lemma:all-minors-matrix-tree-theorem} to our context, we can characterize the determinant and inverse of $\Lambda(H)$.

\begin{lemma}
\label{lemma:Lambda-inverse-expression}
We have
\begin{align*}
    \det\left( -\Lambda(H) \right) = \det\left( L(\cG) + E_{N,N} \right) = \det L(\cG)_{\cV\setminus \{N\}, \cV\setminus \{N\}} = \sum_{\substack{\cT: \text{spanning}\\ \text{tree of }\cG}} \lambda(\cT) := \tau > 0 ,
\end{align*} 
$\nu_{N,i} = \nu_{i,N} = -1$ for $i=1,\dots,N$, and
\begin{align*}
    \nu_{i,j} = \nu_{j,i} = -1 - \frac{1}{\tau} \sum_{\cF \in \mathfrak{F}_{\{i, N\}, \{j, N\}}} \lambda(\cF) 
\end{align*}
for $i,j=1,\dots,N-1$.
\end{lemma}

\begin{proof}

We have $L(\cG)\mathbf 1 = \mathbf{0}$, so the first $N-1$ rows of $-\Lambda(H) = L(\cG) + E_{N,N}$ have row sum equal to $0$, while the last row sums up to $1$.
Hence, if we add the first $N-1$ columns of $-\Lambda(H)$ to its last column, it becomes
\[
    \begin{bmatrix} 0 & \cdots & 0 & 1 \end{bmatrix}^\transpose .
\]
This column operation does not change the determinant, and the expansion along the last column gives
\[
    \det (-\Lambda(H)) = \det \widehat L 
\]
where $\widehat L = L(\cG)_{\cV\setminus\{N\},\,\cV\setminus\{N\}}$ is the minor obtained by eliminating the last row and last column from $L(\cG)$.
Then by the matrix-tree theorem,
\[
\det \widehat L
=
\sum_{\substack{\cT:\text{spanning}\\\text{tree of }\cG}}
\lambda(\cT) = \tau 
\]
as desired.
Note that for $H \in \cH_\star (N-1)$, the support graph of $\cG$ is connected and all its positive-edge weights are nonnegative, so $\tau>0$, i.e., $\widehat L$ is invertible.

Next, $\Lambda(H) \mathbf{1} = -e_N \implies \Lambda(H)^{-1} e_N = -\mathbf{1}$ and $\Lambda(H)^{-1}$ is symmetric, so $\nu_{N,i} = \nu_{i,N} = -1$ for $i=1,\dots,N$.
For the remaining $\nu_{i,j}$ with $1\le i,j \le N-1$, we can write $-\Lambda(H)^{-1}$ in the block structure
\[
-\Lambda(H)^{-1} =
\begin{bmatrix}
R & \mathbf{1} \\
\mathbf{1}^\transpose & 1
\end{bmatrix}
\]
whose Schur complement for the upper left block is $R - \mathbf{1}\mathbf{1}^\transpose$, so we have
\begin{align*}
    \begin{bmatrix}
    R & \mathbf{1} \\
    \mathbf{1}^\transpose & 1
    \end{bmatrix}^{-1} = 
    \begin{bmatrix}
    (R - \mathbf{1}\mathbf{1}^\transpose)^{-1} & * \\
    * & *
    \end{bmatrix} =
    -\Lambda(H) = 
    \begin{bmatrix}
    \widehat L & * \\
    * & *
    \end{bmatrix} ,
\end{align*}
i.e., $R = \widehat L^{-1} + \mathbf{1}\mathbf{1}^\transpose$.
This implies $\nu_{i,j} = -1-(\widehat L^{-1})_{i,j}$ for $1\le i,j \le N-1$.
Next, using the adjugate formula for matrix inverses, we have
\[
(\widehat L^{-1})_{i,j}
=
\frac{(-1)^{i+j}\det \widehat L_{\{1,\dots,N-1\}\setminus\{j\}, \{1,\dots,N-1\}\setminus\{i\}}}{\det \widehat L}
= \frac{(-1)^{i+j}}{\tau} \det \widehat L_{\{1,\dots,N-1\}\setminus\{j\}, \{1,\dots,N-1\}\setminus\{i\}}
.
\]
Because $\widehat L=L(\cG)_{\cV\setminus\{N\}, \cV\setminus\{N\}}$, we can rewrite the minor of $\widehat L$ as
\[
\widehat L_{\{1,\dots,N-1\}\setminus\{j\}, \{1,\dots,N-1\}\setminus\{i\}}
=
L(\cG)_{\cV\setminus\{j,N\}, \cV\setminus\{i,N\}}.
\]
Hence, we can apply \cref{lemma:all-minors-matrix-tree-theorem} with $\cU=\{j,N\}$ and $\cW=\{i,N\}$ to obtain
\[
\det L(\cG)_{\cV\setminus\{j,N\}, \cV\setminus\{i,N\}}
=
(-1)^{i+j}
\sum_{\cF\in \mathfrak F_{\{j,N\},\{i,N\}}}
\mathrm{sign}(\cF)\,\lambda(\cF).
\]
For any $\cF\in \mathfrak F_{\{j,N\},\{i,N\}}$, the component containing $N$ must pair $N$ with itself, and the other component must pair $j$ with $i$, so the induced bijection has no inversions, i.e., \(\mathrm{sign}(\cF)=1\).
Also, because $\widehat L$ is symmetric, we can rewrite the forest family as \(\mathfrak F_{\{i,N\},\{j,N\}}\). Together, this yields
\[
(\widehat L^{-1})_{i,j}
=
\frac{1}{\tau}
\sum_{\cF\in \mathfrak F_{\{i,N\},\{j,N\}}} \lambda(\cF) 
\implies \nu_{i,j} = -1 - \frac{1}{\tau} \sum_{\cF\in \mathfrak F_{\{i,N\},\{j,N\}}} \lambda(\cF)
\]
as desired.
\end{proof}

Now, if we specialize to the case of vertex algorithms, $\cG$ itself is a tree (neglecting the edges of weight $0$) and thus, \cref{lemma:Lambda-inverse-expression} implies $\tau = \det\left( L(\cG) + E_{N,N} \right) = \lambda(\cG) = \prod_{j=1}^{N-1} \lambda_{k(j),j}$.
Additionally, for any $i,j < N$, a spanning forest $\cF \in \mathfrak{F}_{\{i, N\}, \{j, N\}}$ should have $N-2$ edges and two connected components, so it will be obtained by eliminating exactly one edge from $\cG$.
Recall that $\cC(j)$ denotes the unique increasing path to $N$, starting at $j$.
Then we have the following characterization.

\begin{corollary}
\label{corollary:nu-explicit-form}
Let $H \in \cH_\star(N-1)$ be a vertex algorithm and consider its arc diagram $\cG = \cG(H)$.
Then for $1\le i,j\le N$,
\begin{align*}
    \nu_{i,j} = -1 - \sum_{v \in \cC(i)\cap \cC(j)\setminus\{N\}} \frac{1}{\lambda_{k(v),v}} = -1 - \sum_{v=1}^{N-1} \frac{1_{v\in \cC(i)} 1_{v\in \cC(j)}}{\lambda_{k(v),v}} .
\end{align*}
\end{corollary}

\begin{proof}
If either $i=N$ or $j=N$, then $\cC(i)\cap \cC(j)\setminus\{N\} = \emptyset$ and the summation becomes vacuous. This aligns with the formula $\nu_{N,j} = \nu_{i,N} = -1$ for $i,j=1,\dots,N$, proved in \cref{lemma:Lambda-inverse-expression}.

Note that because $H$ is a vertex, $\cG$ is a tree, and thus it is the only spanning tree of itself.
That is, $\tau = \lambda(\cG) = \prod_{u=1}^{N-1} \lambda_{k(u),u}$.
Now for $1\le i,j\le N-1$, we have $\cF \in \mathfrak{F}_{\{i, N\}, \{j, N\}}$ if and only if $\cF$ is obtained by removing an arc $(k(v), v)$ from $\cG$, and it has both $i$ and $j$ in its connected component not containing $N$.
When an arc $(k(v),v)$ is removed, it separates a node $u$ from $N$ if and only if $v \in \cC(u) \setminus \{N\}$, i.e., when $v$ lies on the simple path $\cC(u)$ from $u$ to $N$ (note that $v$ cannot be $N$ because $N$ is the terminal node).
Hence, removing $(k(v), v)$ places both $i,j$ in the component separated from $N$ if and only if $v \in \cC(i) \cap \cC(j) \setminus \{N\}$.
For the resulting forest $\cF$, we have $\lambda(\cF) = \prod_{\substack{1\le u\le N-1 \\ u\ne v}} \lambda_{k(u),u} = \frac{\tau}{\lambda_{k(v),v}}$.
Therefore, by \cref{lemma:Lambda-inverse-expression},
\begin{align*}
    \nu_{i,j} = -1 - \sum_{\cF \in \mathfrak{F}_{\{i, N\}, \{j, N\}}} \frac{\lambda(\cF)}{\tau} = -1 - \sum_{v \in \cC(i)\cap \cC(j)\setminus\{N\}} \frac{1}{\lambda_{k(v),v}} .
\end{align*}
\end{proof}

\begin{lemma}
\label{lemma:non-crossing-component}
Let $\cG$ be a basic (i.e., non-crossing) arc diagram.
Suppose $\cF$ is a spanning forest of $\cG$ with two disjoint connected components: $\cT_1$ containing $i,j$ such that $j\le i<N$, and $\cT_2$ containing $N$. Then we have $\{j,j+1,\dots,i\} \subseteq \cT_1$.
\end{lemma}

\begin{proof}
Suppose that $v \notin \cT_1$ for some $j < v < i$. Then we must have $v \in \cT_2$.
That is, $\cT_2$ is the connected component containing both $v$ and $N$, so it must contain the unique simple path $\cC(v)$ joining them.
This implies, $\cC(v) \cap \cT_1 = \emptyset$.
On the other hand, $\cT_1$ contains the unique simple path between $i$ and $j$, and thus contains the segment of $\cC(j)$ joining $j$ to $m = \min \left(\cC(j)\cap \cC(i)\right)$.
Note that we must have $i\le m < N$; if $m=N$ then $j \in \cT_1$ is joined to $N \in \cT_2$ within $\cF$, a contradiction.

Let $v_0 \in \cC(v)$ be the node satisfying $v_0 < m < k(v_0)$. Such a vertex exists because the increasing path $v < k(v) < k(k(v)) < \cdots$ should terminate at $N$, and it does not contain $m$ because $m\in \cT_1$ while $\cC(v) \subseteq \cT_2$. 
Similarly, as the increasing path $j < k(j) < k(k(j)) < \cdots$ should terminate at $m$ without intersecting $\cC(v)$, there exists $j_0 \in \cC(j)$ satisfying $j_0 < v_0 < k(j_0)$.
This contradicts the assumption that $\cG$ is non-crossing, as $j_0 < v_0 < k(j_0) \le m < k(v_0)$.
\end{proof}

\begin{lemma}
\label{lemma:descendant-interval}
Let $\cG$ be a basic arc diagram.
For each $v = 1,\dots,N-1$, the set of its descendants
\[
\cI(v) := \{\, j\in\{1,\dots,N-1\} : v\in \cC(j)\,\} 
\]
is an interval, i.e., $\cI(v) = \{\ell(v),\ell(v)+1,\dots,v\}$.
\end{lemma}

\begin{proof}
Since $\cG$ is a tree, removing a single arc from $\cG$ yields a spanning forest with two connected components.
Let $\cF$ be the forest obtained by removing the arc $(k(v), v)$, so that $v$ lies in the component $\cT$ of $\cF$ not containing $N$.
Because $\cF$ contains all edges with starting node (smaller node) less than $v$, we see that any $j\in \cI(v)$ is joined to $v$ by $\cC(j)$ within $\cF$, so $\cI(v) \subseteq \cT$.
Conversely, if $j \notin \cI(v)$, then $v\notin \cC(j)$, and $\cF$ contains all edges within the path $\cC(j)$.
Therefore, $j$ is joined to $N$ in $\cF$, so $j$ belongs to a different connected component than $v$, i.e., $j\notin \cT$.
This shows that $\cT \subseteq \cI(v)$, and eventually, $\cT = \cI(v)$.

Now, as defined in Section~\ref{section:H-duality-for-vertices}, $\ell(v)$ is the smallest descendant of $v$, so clearly $\ell(v) \in \cI(v) = \cT$.
So $\cT$ is the connected component of $\cF$ containing $\ell(v)$ and $v$ but not $N$; hence, by \cref{lemma:non-crossing-component}, we have $\{\ell(v),\ell(v)+1,\dots,v\} \subseteq \cT = \cI(v)$.
However, this inclusion is in fact an equality because $\ell(v)$ is the smallest element of $\cI(v)$, and by definition $\cI(v)$ cannot contain any node larger than $v$.
This completes the proof.
\end{proof}

\subsubsection{Proof of \texorpdfstring{\cref{theorem:dual-optimality-and-non-crossing-arc-diagram}}{Theorem 4}}
\label{section:proof-of-dual-optimality-and-basic-diagram-theorem}

First, we assume $j_1 < j_2 < k(j_1) < k(j_2)$ for some $j_1, j_2$ and show that $H^\at$ cannot be optimal.
Suppose to the contrary, so that $H^\at$ has nonnegative H-certificates, i.e., all off-diagonal entries of $\Lambda(H^\at)$ in~\eqref{eqn:lambda-H-dual-explicit} are nonnegative.
Then
\begin{align*}
    \nu_{i,j} - \nu_{i-1,j} - \nu_{i,j-1} + \nu_{i-1,j-1} \ge 0
\end{align*}
for all $1\le j < i \le N$, where we let $\nu_{\cdot,0} = \nu_{0,\cdot} = -1$ for convenience.
Summing this inequality for all $j = j_1 + 1 , \dots, j_2$ and $i = k(j_1) + 1, \dots , k(j_2)$, we obtain
\begin{align}
\label{eqn:nonnegative-flux-contradiction}
    \nu_{k(j_2), j_2} - \nu_{k(j_2), j_1} - \nu_{k(j_1), j_2} + \nu_{k(j_1), j_1} \ge 0 .
\end{align}
Rewrite $k(j_1) = i_1$ and $k(j_2) = i_2$ for simplicity.
Then we have $\cC(i_1) \subset \cC(j_1) = \{j_1\} \cup \cC(i_1)$ and $\cC(i_2) \subset \cC(j_2) = \{j_2\} \cup \cC(i_2)$. 
On the other hand, because $j_2 \notin \cC(i_1)$ and $j_1 \notin \cC(i_2)$, we have $\cC(i_1) \cap \cC(j_2) = \cC(i_1) \cap \cC(i_2) = \cC(j_1) \cap \cC(i_2)$. 
By \cref{corollary:nu-explicit-form}, this implies
\begin{align*}
    \nu_{i_2,j_2} - \nu_{i_1,j_2} = -\sum_{v \in \cC(i_2)\setminus \{N\}} \frac{1}{\lambda_{k(v),v}} + \sum_{v \in (\cC(i_1) \cap \cC(i_2)) \setminus \{N\}} \frac{1}{\lambda_{k(v),v}} = -\sum_{v \in \cC(i_2)\setminus \cC(i_1)} \frac{1}{\lambda_{k(v),v}}
\end{align*}
and similarly
\begin{align*}
    \nu_{i_1,j_1} - \nu_{i_2,j_1} = -\sum_{v \in \cC(i_1)\setminus \{N\}} \frac{1}{\lambda_{k(v),v}} + \sum_{v \in (\cC(i_1) \cap \cC(i_2)) \setminus \{N\}} \frac{1}{\lambda_{k(v),v}} = -\sum_{v \in \cC(i_1)\setminus \cC(i_2)} \frac{1}{\lambda_{k(v),v}} .
\end{align*}
In particular, we have $\nu_{i_2,j_2} - \nu_{i_1,j_2} \le 0$ and $\nu_{i_1,j_1} - \nu_{i_2,j_1} < 0$ because $i_1 \in \cC(i_1) \setminus \cC(i_2)$. This contradicts~\eqref{eqn:nonnegative-flux-contradiction}.

Conversely, suppose the arc diagram $\cG$ of $H$ is non-crossing.
Select $1\le j<i \le N$ arbitrarily.
Using \eqref{eqn:lambda-H-dual-explicit} and \cref{corollary:nu-explicit-form}, we can write
\begin{align*}
    \lambda_{N-j+1,N-i+1}^\star (H^\at) & = \nu_{i,j} - \nu_{i-1,j} - \nu_{i,j-1} + \nu_{i-1,j-1} \\
    & = -\sum_{v=1}^{N-1} \frac{(1_{v\in \cC(i)} - 1_{v\in \cC(i-1)}) (1_{v\in \cC(j)} - 1_{v\in \cC(j-1)})}{\lambda_{k(v),v}} \\
    & = \sum_{v=1}^{N-1} \frac{(1_{v\in \cC(i-1)} - 1_{v\in \cC(i)}) (1_{v\in \cC(j)} - 1_{v\in \cC(j-1)})}{\lambda_{k(v),v}}
\end{align*}
where we set $\cC(N) := \{N\}$ so that $1_{v\in \cC(N)} = 0$ for $v=1,\dots,N-1$, and $\cC(0):=\emptyset$.
Because we have 
\[
    \cI(v) = \{u \in \{1,\dots,N-1\}: v \in \cC(u)\} = \{\ell(v), \ell(v)+1, \dots, v\}
\]
by \cref{lemma:descendant-interval}, the product $(1_{v\in \cC(i-1)} - 1_{v\in \cC(i)}) (1_{v\in \cC(j)} - 1_{v\in \cC(j-1)})$ is nonzero only when $i$ and $j$ both lie at the boundary of $\cI(v)$. More precisely, we have
\begin{align*}
    (1_{v\in \cC(i-1)} - 1_{v\in \cC(i)}) (1_{v\in \cC(j)} - 1_{v\in \cC(j-1)}) = 1 \iff i = v+1, j = \ell(v)
\end{align*}
and otherwise the product is $0$ and the corresponding summand does not count.
Therefore,
\begin{align*}
    \lambda_{N-j+1,N-i+1}^\star (H^\at) = \begin{cases}
        \frac{1}{\lambda_{k(i-1),i-1}} & \text{if } j = \ell(i-1) \\
        0 & \text{otherwise.}
    \end{cases}
\end{align*}
This shows that $H^\at$ is optimal (as its H-certificates are all nonnegative), and it is a vertex algorithm because for each $N-i+1$ ($i=2,\dots,N$), we have a unique $N-j+1$ ($1\le j < i$) such that $\lambda_{N-j+1,N-i+1}^\star (H^\at)$ is positive. 

\subsubsection{Proof of \texorpdfstring{\cref{corollary:dual-of-decomposable-basic-diagram}}{Corollary 4}}
\label{section:proof-of-diagram-dualization-corollary}

\begin{proof}
Let
\[
    \cG_L := \cG_{1,N'}, \qquad
    \cG_R := \cG_{N'+1,N},
\]
where the nodes of $\cG_R$ are relabeled to $\{1,\dots,N_R\}$ with $N_R:=N-N'$.
Denote the sets $\cI(\cdot)$ associated with $\cG, \cG_L$ and $\cG_R$ (defined as in \cref{lemma:descendant-interval}) by $\cI_\cG (\cdot), \cI_{\cG_L} (\cdot)$ and $\cI_{\cG_R} (\cdot)$, and their minimum elements by $\ell_{\cG}(\cdot), \ell_{\cG_L}(\cdot)$ and $\ell_{\cG_R}(\cdot)$, respectively.
We show that 
\[
    \cI_\cG(v)=
    \begin{cases}
        \cI_{\cG_L}(v) & \text{if } 1\le v\le N'-1 \\[3pt]
        \{1,\dots,N'\} & \text{if } v=N' \\[3pt]
        N' + \cI_{\cG_R}(v-N') & \text{if } N'+1\le v\le N-1 .
    \end{cases}
\]
First, the case $1\le v \le N'-1$ is clear because $\cI_\cG (v)$ depends only on the structure of $\cG_L$.
Next, because $\cG$ is decomposable at $N'$, we have $v < k(v) \le N'$ for any $v \in \{1,\dots,N'-1\}$, which implies $N' \in \cC(v)$. Equivalently, $\cI_\cG(N') = \{1,\dots,N'\}$.
Finally, for $N'+1 \le v \le N-1$, we have $v \notin \cC(j)$ for any $j=1,\dots,N'$ so we can exclude the nodes $1,\dots,N'$ from consideration when determining $\cI_\cG(v)$. 
Thus $\cI_\cG(v)$ is determined solely by the structure of $\cG_R$ up to node relabeling, which shows the identity above.
Therefore, we obtain
\begin{equation}
\label{eq:ell-decomposable}
    \ell_\cG(v)=
    \begin{cases}
        \ell_{\cG_L}(v), & 1\le v\le N'-1,\\[3pt]
        1, & v=N',\\[3pt]
        N' + \ell_{\cG_R}(v-N'), & N'+1\le v\le N-1 .
    \end{cases}
\end{equation}
Now if we denote by $\cE^\at$ the edge set of $\cG^\at$, then as shown in the proof of \cref{theorem:dual-optimality-and-non-crossing-arc-diagram},
\[
    (N-j+1, N-i+1) \in \cE^\at 
    \iff j = \ell(i-1) .
\]
Applying this with $i=N'+1$ and $j=1$ we have $(N, N-N') = (N, N_R) \in \cE^\at$, which is the gluing arc.

For $1\le v < w \le N_R$,
\[
    (w,v) \in \cE^\at \iff N-w+1 = \ell(N-v) = N' + \ell_{\cG_R} (N-N'-v) \iff N_R - w + 1 = \ell_{\cG_R}(N_R - v)
\]
and the last condition is precisely the condition for $(w,v) \in \cE(\cG_R^\at)$, where $\cE(\cG_R^\at)$ is the edge set of $\cG_R^\at$.

Finally, for $N_R + 1 \le v < w \le N$, let $v'=v-N_R$ and $w'=w-N_R$ so that $N-v=N'-v'$ and $N-w=N'-w'$. Then
\begin{align*}
    (w,v) \in \cE^\at \iff N-w+1 = \ell(N-v) = \ell_{\cG_L}(N-v) \iff N'-w'+1 = \ell_{\cG_L} (N'-v')
\end{align*}
and the last condition is equivalent to $(w',v') \in \cE(\cG_L^\at)$, where $\cE(\cG_L^\at)$ is the edge set of $\cG_L^\at$.
Hence the right block of $\cG^\at$ has the same structure as $\cG_L^\at$, with node labels shifted to the right by $N_R$.

Put altogether, by definition of gluing, the above characterization implies
\[
    \cG^\at = \cG_R^\at \glue \cG_L^\at .
\]
\end{proof}

\subsection{Self-dual vertex algorithms}
\label{section:self-dual-vertex-algorithms}

We say that an optimal vertex algorithm \(H\) is \textit{self-dual} if \(H^\at = H\); equivalently, if its arc diagram satisfies \(\cG^\at = \cG\).
The recursive dualization rule in \cref{corollary:dual-of-decomposable-basic-diagram} suggests that a self-dual basic diagram must split into two blocks of equal size, with one block being the dual of the other.
This necessary condition is in fact also sufficient, and yields the exact count of self-dual vertex algorithms---this is again a Catalan number.

\begin{theorem}
\label{theorem:self-dual-vertex-algorithms}
Let \(N\ge 2\), let \(H \in \cH_\star (N-1)\) be an optimal vertex algorithm, and let \(\cG = \cG(H)\) be its arc diagram.
Then \(H\) is self-dual if and only if \(N\) is even and there exists a basic arc diagram \(\cG_0\) on \(\{1,\dots,N/2\}\) such that
\begin{align}
\label{eqn:self-dual-vertex-structure}
    \cG = \cG_0 \glue \cG_0^\at ,
\end{align}
i.e., \(\cG\) is decomposable at \(N/2\) and removing the arc \((N,N/2)\) separates it into subdiagrams that are dual to each other.
Consequently, there is no self-dual vertex algorithm when \(N\) is odd, while for even \(N\) the number of self-dual vertex algorithms is
\[
    C_{N/2-1} = \frac{1}{N/2}\binom{N-2}{N/2-1}.
\]
\end{theorem}

\begin{proof}
Assume first that \(H^\at = H \in \cH_\star(N-1) \). By \cref{theorem:dual-optimality-and-non-crossing-arc-diagram}, its arc diagram \(\cG\) is basic, i.e., non-crossing.
By \cref{lemma:leftmost-arc-to-N-gives-decomposition}, $\cG$ is decomposable at $N' = \min\{j\in\{1,\dots,N-1\}: k(j)=N\}$, and this decomposition index is unique.
Denote the two arc diagrams obtained by removing $(N,N')$ from $\cG$ by $\cG_L$ (the left component) and $\cG_R$ (the right component).
Then 
\begin{align*}
    \cG_L \glue \cG_R = \cG = \cG^\at = \cG_R^\at \glue \cG_L^\at
\end{align*}
by \cref{corollary:dual-of-decomposable-basic-diagram}.
This implies that $\cG$ is decomposable at both \(N'\) and \(N-N'\), so we must have $N' = N-N'$ by uniqueness of the decomposition index.
Therefore, \(N\) is even and \(N'=N/2\).
Furthermore, comparing the left and right blocks of the two decompositions, we obtain $\cG_L = \cG_R^\at$ and $\cG_R = \cG_L^\at$.
Therefore
\[
    \cG = \cG_L \glue \cG_L^\at ,
\]
which is exactly \eqref{eqn:self-dual-vertex-structure} with \(\cG_0=\cG_L\).

Conversely, suppose \(N\) is even and \eqref{eqn:self-dual-vertex-structure} holds for some basic arc diagram \(\cG_0\) on \(N/2\) nodes.
Because \(\cG_0\) is basic, its dual \(\cG_0^\at\) is also basic by \cref{theorem:dual-optimality-and-non-crossing-arc-diagram}, and hence \(\cG\) is basic as well.
Additionally, by \cref{corollary:dual-of-decomposable-basic-diagram},
\[
    \cG^\at = (\cG_0^\at)^\at \glue \cG_0^\at
    = \cG_0 \glue \cG_0^\at
    = \cG .
\]
Therefore, the optimal vertex algorithm $H$ corresponding to $\cG$ is self-dual.

Finally, the above characterization implies that no self-dual vertex algorithm exists when \(N\) is odd, and when \(N\) is even, the map
\[
    \cG_0 \mapsto \cG_0 \glue \cG_0^\at
\]
is a bijection from the set of basic arc diagrams on \(N/2\) nodes to the set of self-dual vertex algorithms on \(N\) nodes.
Hence the counting argument directly follows from \cref{corollary:catalan-many-dual-optimal-algorithms}, and the proof is complete.
\end{proof}

\section{Application to novel algorithm design and principled analyses}
\label{section:deducing-algorithm-design-and-analysis}

\subsection{Two classes of new optimal algorithms}

Here, we design two new fixed-point methods with natural arc diagram forms. By virtue of these forms, their analysis is immediate by our composing theory, ensuring they are minimax optimal without additional effort. Although both are optimal (as are OHM and Dual-OHM), our theory allows us to contrast these methods in terms of their level of guarantees at intermediate iterates and their robustness to violation of nonexpansiveness.

\subsubsection{Repeated Dual-OHM}
\label{section:repeated-dual-ohm}

We consider periodic arc diagrams $\cG$ as in Figure~\ref{figure:periodic-dual-ohm-p=5}—this corresponds to the H-matrices defined by taking a fixed period $p \in \mathbb N$, and letting $H_{\text{RDO}(p)}^{(0)} = [\,\,]$ and
\begin{align*}
    H_{\text{RDO}(p)}^{(n)} = H_{\text{RDO}(p)}^{(n-1)} \glue H_{\text{Dual-OHM}} (p-1) \in \cH_\star(np) \subset \reals^{np \times np}
\end{align*}
for $n=1,2,\dots$.
We define the algorithm whose H-matrix for the first $np$ iterations is $H_{\text{RDO}(p)}^{(n)}$ for all $n=1,2,\dots$ as \textit{Repeated Dual-OHM (RDO) with period $p$}.
Note that this is well-defined without specifying a terminal iteration $N$ because each $H_{\text{RDO}(p)}^{(n)}$ is an upper left submatrix of $H_{\text{RDO}(p)}^{(m)}$ for any $m>n$.

\begin{figure}[t]
\centering
\begin{tikzpicture}[baseline=(current bounding box.center), scale=0.68,
    dot/.style={circle,fill,inner sep=1.5pt},
    every node/.style={font=\small}
]
    \foreach \i in {1,...,16}{
        \coordinate (P\i) at (\i,0);
        \fill (P\i) circle (2pt);
        \node[below=2pt] at (P\i) {\i};
    }
    \coordinate (P16) at (16,0);
    \coordinate (P17) at (17,0);

    \draw[line width=0.8pt, color=gray, densely dashed] (P1) -- (P17);

    \begin{scope}
        \foreach \b in {6,11,16}{
            \foreach \d/\r in {5/2.5,4/2,3/1.5,2/1,1/0.5}{
                \pgfmathtruncatemacro{\a}{\b-\d}
                \draw[line width=0.8pt]
                    (P\a) arc[start angle=180,end angle=0,radius=\r];
            }
        }
    \end{scope}

    \node[draw=none, fill=white, inner sep=1.5pt] at (16.7,1.25) {$\cdots$};
\end{tikzpicture}
\caption{An example arc diagram for Repeated Dual-OHM with period $p=5$. It periodically repeats the pattern from Dual-OHM, and thereby has guarantees on the iterates $y_5, y_{10}, y_{15}, \dots$.}
\label{figure:periodic-dual-ohm-p=5}
\end{figure}
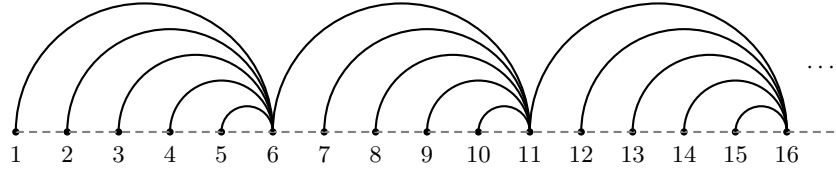

By applying \cref{theorem:gluing} recursively, we see that for $n\ge 3$,
\begin{align}
\label{eq:periodic-dual-ohm-block-structure}
H_{\text{RDO}(p)}^{(n)} =
\begin{bmatrix}
\begin{array}{c|c|c|c|c|c}
H_{\text{Dual-OHM}} (p) & & & & \\[0.2em]
\hline
a_1^\transpose & \frac{p(p+1)}{2p+1} & & & \\
\hline
0_{(p-1)\times p} & b_1 & H_{\text{Dual-OHM}} (p-1) & & \\[0.2em]
\hline
\multicolumn{3}{c|}{a_2^\transpose} & \frac{p(2p+1)}{3p+1} & \\
\hline
\multicolumn{3}{c|}{0_{(p-1)\times (2p)}} & b_2 & H_{\text{Dual-OHM}} (p-1) \\
\hline
\multicolumn{3}{c|}{\vdots} & \vdots & \vdots & \ddots
\end{array}
\end{bmatrix},
\end{align}
where $a_1 \in \reals^p, a_2 \in \reals^{2p}, \dots$ are the H-matrix rows corresponding to the update rule
\begin{align*}
    y_{kp+1} = y_{kp} - \frac{p(kp+1)}{(k+1)p+1} (y_{kp} - \opT y_{kp}) + \frac{p}{(k+1)p+1} (y_0 - y_{kp})
\end{align*}
for each $k=1,2,\dots$ and
\begin{align*}
    b_k = \left( \frac{1}{(k+1)p + 1} - \frac{1}{p} \right)
    \begin{bmatrix}
        p-1 \\ p-2 \\ \vdots \\ 1
    \end{bmatrix} \in \reals^{p-1} .
\end{align*}
The computation of $b_k$ follows from the fact that it is independent of the left block diagram $\cG_{1,kp+1}$, and the $b_k$ produced by gluing with the arc diagram of $H_\text{Dual-OHM}(p-1)$ was provided in \citep[Proposition~7]{YoonRyuGrimmerInvariance2025}.
The update rule of RDO can be expressed as an efficient recurrence that requires keeping only two past vectors, one of which is $y_0$, in memory.

\begin{proposition}
The recurrence defining the algorithm associated with the H-matrix~\eqref{eq:periodic-dual-ohm-block-structure} is given by 
\begin{align}
\label{eqn:RDO-p-recurrence}
y_{kp+\ell+1} = \begin{cases}
    y_{kp} - \frac{p(kp+1)}{(k+1)p+1} (y_{kp} - \opT y_{kp}) + \frac{p}{(k+1)p+1}\,(y_0-y_{kp}) & \text{if } \ell = 0 \\
    y_{kp+1} + \frac{p-1}{p} \left( \opT y_{kp+1} - y_{kp+1} + \frac{kp+1}{(k+1)p+1} (y_{kp} - \opT y_{kp}) \right) & \text{if } \ell = 1 \text{ and } p \ge 2 \\
    y_{kp+\ell} + \frac{p-\ell}{p-\ell+1} \left(\opT y_{kp+\ell} - \opT y_{kp+\ell-1} \right) & \text{if } 2\le \ell \le p-1
\end{cases}
\end{align}
for $k=0,1,\dots$.
\end{proposition}

\begin{proof}
The first line in \eqref{eqn:RDO-p-recurrence} follows from how $a_k$ are chosen.
Next, because we have
\[
    \bigl(H_{\text{Dual-OHM}}(p-1)\bigr)_{\ell,j}
=
\begin{cases}
-\dfrac{p-\ell}{(p-j)(p-j+1)} & \text{if } j<\ell\\[0.8ex]
\dfrac{p-\ell}{p-\ell+1} & \text{if } j=\ell,
\end{cases}
\]
from the H-matrix representation~\eqref{eq:periodic-dual-ohm-block-structure}, the update formula becomes
\begin{align}
    y_{kp+\ell+1} - y_{kp+\ell} & = -2\left( \frac{1}{(k+1)p+1} - \frac{1}{p} \right) (p-\ell) g_{kp+1} + \sum_{j=1}^{\ell-1}\frac{2(p-\ell)}{(p-j)(p-j+1)} g_{kp+j+1} - \frac{2(p-\ell)}{p-\ell+1} g_{kp+\ell+1} \nonumber \\    
    & = \frac{2(p-\ell)(kp+1)}{p((k+1)p+1)} g_{kp+1} + \sum_{j=1}^{\ell-1}\frac{2(p-\ell)}{(p-j)(p-j+1)} g_{kp+j+1} - \frac{2(p-\ell)}{p-\ell+1} g_{kp+\ell+1} .
\label{eqn:RDO-delta}
\end{align}
For \(\ell=1\), \eqref{eqn:RDO-delta} becomes
\begin{align*}
    y_{kp+2} = y_{kp+1} - \frac{p-1}{p} (2g_{kp+2}) + \frac{(p-1)(kp+1)}{p((k+1)p+1)} (2g_{kp+1})
\end{align*}
which agrees with \eqref{eqn:RDO-p-recurrence} once we substitute $2g_{kp+1} = y_{kp} - \opT y_{kp}$ and $2g_{kp+2} = y_{kp+1} - \opT y_{kp+1}$.

For \(2\le \ell \le p-1\), rewriting 
\[
    \opT y_{kp+\ell}-\opT y_{kp+\ell-1} = y_{kp+\ell} - y_{kp+\ell-1} - 2g_{kp+\ell+1} + 2g_{kp+\ell}
\]
using \eqref{eqn:RDO-delta} with the substitution \(\ell \leftarrow \ell-1\) yields
\[
    \opT y_{kp+\ell}-\opT y_{kp+\ell-1} = \frac{2(p-\ell+1)(kp+1)}{p((k+1)p+1)} g_{kp+1} + \sum_{j=1}^{\ell-2}\frac{2(p-\ell+1)}{(p-j)(p-j+1)} g_{kp+j+1} + \frac{2}{p-\ell+2} g_{kp+\ell} - 2g_{kp+\ell+1} .
\]
This precisely agrees with $\frac{p-\ell+1}{p-\ell}$ times \eqref{eqn:RDO-delta}; the only term that is not immediate is the coefficient of $g_{kp+\ell}$, for which we have
\[
    \frac{p-\ell+1}{p-\ell} \frac{2(p-\ell)}{(p-(\ell-1))(p-(\ell-1)+1)} = \frac{2}{p-\ell+2} 
\]
as desired. Thus,
\[
    \opT y_{kp+\ell}-\opT y_{kp+\ell-1} = \frac{p-\ell+1}{p-\ell} \left( y_{kp+\ell+1} - y_{kp+\ell} \right) ,
\]
which is equivalent to the recurrence in~\eqref{eqn:RDO-p-recurrence}.
\end{proof}

By the discussion in the last paragraph of \cref{section:general-theory-of-optimal-vertex-algorithms} (or by recursively applying \cref{corollary:intermediate-iterate-guarantee-for-decomposable-algorithm}), we see that RDO has the following convergence property:
\begin{proposition}
For any $p\in \mathbb N$, RDO with period $p$ exhibits the rate $\sqnorm{y_{kp} - \opT y_{kp}} \le \frac{4\sqnorm{y_0 - y_\star}}{(kp + 1)^2}$ for $k=1,2,\dots$.
\end{proposition}

\paragraph{OHM and Dual-OHM are special cases of RDO.}
RDO with period $p=1$ corresponds to OHM---in this case only the first line of~\eqref{eqn:RDO-p-recurrence} is active, which becomes the OHM update rule $y_{k+1} = \frac{1}{k+2} y_0 + \frac{k+1}{k+2} \opT y_k$.
On the other hand, restricting~\eqref{eqn:RDO-p-recurrence} to the first $p$ iterations (so $k=0$) and letting $p=N-1$ yields the Dual-OHM update rule $y_{\ell+1} = y_\ell + \frac{N-\ell-1}{N-\ell} (\opT y_\ell - \opT y_{\ell-1})$ with the identification $\opT y_{-1} = y_0$, which agrees with the first line $y_1 = y_0 - \frac{N-1}{N} (y_0 - \opT y_0)$.

\paragraph{RDO with generalized schedules.}
In fact, one can run RDO with any sequence $\{p_k\}_{k=1,2,\dots}$ of natural numbers indicating the length of Dual-OHM's arc diagram pattern.
Formally, we define $H_{\text{RDO}(p_1)} = H_{\text{Dual-OHM}} (p_1)$ and
\begin{align*}
    H_{\text{RDO}(p_1,\dots,p_n)} = H_{\text{RDO}(p_1,\dots,p_{n-1})} \glue H_{\text{Dual-OHM}} (p_n - 1) \in \reals^{(p_1 + \dots + p_n) \times (p_1 + \dots + p_n)}
\end{align*}
for $n\ge 2$, and refer to the associated algorithm as \textit{RDO with schedule} $\{p_k\}$.
Letting $S_0 = 0$ and $S_n = \sum_{i=1}^n p_i$ for $n\ge 1$, the recurrence will be given by
\begin{align*}
y_{S_{k}+\ell+1} =
\begin{cases}
    y_{S_{k}} - \frac{p_{k+1} (S_k+1)}{S_{k+1}+1} (y_{S_k} - \opT y_{S_k}) + \frac{p_{k+1}}{S_{k+1}+1}\,(y_0-y_{S_k}) & \text{if } \ell = 0 \\
    y_{S_k+1} + \frac{p_{k+1}-1}{p_{k+1}} \left( \opT y_{S_k+1} - y_{S_k+1} + \frac{S_k+1}{S_{k+1}+1} (y_{S_k} - \opT y_{S_k}) \right) & \text{if } \ell = 1 \text{ and } p_{k+1} \ge 2 \\
    y_{S_k+\ell} + \frac{p_{k+1}-\ell}{p_{k+1}-\ell+1} \left(\opT y_{S_k+\ell} - \opT y_{S_k+\ell-1} \right) & \text{if } 2\le \ell \le p_{k+1} - 1 
\end{cases}
\end{align*}
for $k=0,1,\dots$, and the algorithm will satisfy $\sqnorm{y_{S_k} - \opT y_{S_k}} \le \frac{4\sqnorm{y_0 - y_\star}}{(S_k + 1)^2}$.

\subsubsection{Fractal Self-Dual Method}

Based on the characterization of self-dual algorithms from Section~\ref{section:self-dual-vertex-algorithms}, in this section, we explore the algorithms that are \textit{maximally} self-dual.
This can be designed inductively; starting from $H_{\text{FSDM}}(0) = [\,\,]$, we can define
\begin{align*}
    H_{\text{FSDM}}(N-1) = H_{\text{FSDM}}\left( \frac{N}{2}-1 \right) \glue H_{\text{FSDM}}^\at \left( \frac{N}{2}-1 \right)
\end{align*}
if $N$ is even\footnote{A similar rule can be used for odd $N$, e.g., $H_{\text{FSDM}}(N-1) = H_{\text{FSDM}}\left( \frac{N-1}{2} \right) \glue H_{\text{FSDM}}^\at \left( \frac{N-3}{2} \right)$ so that $H_{\text{FSDM}}(N-1)$ is structurally similar to its own dual, but it cannot be truly self-dual due to \cref{theorem:self-dual-vertex-algorithms}.}, so that it is self-dual.
If $\frac{N}{2}$ is again even, then we will have $H_{\text{FSDM}}\left( \frac{N}{2}-1 \right) = H_{\text{FSDM}}\left( \frac{N}{4}-1 \right) \glue H_{\text{FSDM}}^\at \left( \frac{N}{4}-1 \right)$, which is also self-dual.
This indicates that for $n=1,2,\dots$, $H_\text{FSDM}(2^n - 1)$ is self-dual with its dyadic sub-blocks (either in the H-matrix or in the arc diagram) being again self-dual, allowing for the fractal-like repeated decomposition as in Figure~\ref{figure:fractal-self-dual-N=16}.
We define the algorithm whose H-matrix for the first $2^n - 1$ iterations is $H_\text{FSDM}(2^n - 1)$ for all $n=1,2,\dots$ as \textit{Fractal Self-Dual Method (FSDM)}.

\begin{figure}[t]
\centering
\begin{tikzpicture}[baseline=(current bounding box.center), scale=0.68,
    dot/.style={circle,fill,inner sep=1.5pt},
    every node/.style={font=\small}
]
    \foreach \i in {1,...,16}{
        \coordinate (P\i) at (\i,0);
        \fill (P\i) circle (2pt);
        \node[below=2pt] at (P\i) {\i};
    }

    \coordinate (P17) at (17,0);

    \draw[line width=0.8pt, color=gray, densely dashed] (P1) -- (P17);

    \draw[line width=0.8pt]
        (P8) arc[start angle=180,end angle=0,radius=4];

    \foreach \b in {8,16}{
        \pgfmathtruncatemacro{\a}{\b-4}
        \draw[line width=0.8pt]
            (P\a) arc[start angle=180,end angle=0,radius=2];
    }

    \foreach \b in {4,8,12,16}{
        \pgfmathtruncatemacro{\a}{\b-2}
        \draw[line width=0.8pt]
            (P\a) arc[start angle=180,end angle=0,radius=1];
    }

    \foreach \b in {2,4,6,8,10,12,14,16}{
        \pgfmathtruncatemacro{\a}{\b-1}
        \draw[line width=0.8pt]
            (P\a) arc[start angle=180,end angle=0,radius=0.5];
    }
    
    \node[draw=none, fill=white, inner sep=1.5pt] at (16.7, 1.8) {$\cdots$};
\end{tikzpicture}
\caption{The arc diagram for the fractal self-dual algorithm $H_{\text{FSDM}}(15)$ on $16$ nodes. Each dyadic sub-block is again self-dual.}
\label{figure:fractal-self-dual-N=16}
\end{figure}
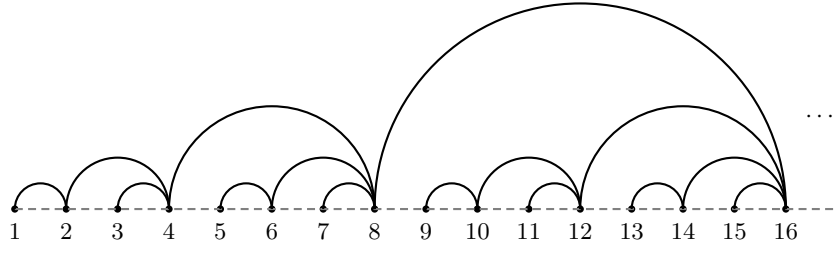

For $j \in \mathbb N$, denote by $\nu_2 (j)$ its 2-adic valuation, i.e., the largest exponent $i$ such that $2^i$ divides $j$.
Then we have $k(j) = j + 2^{\nu_2 (j)}$; by construction $k(2^n) = 2^{n+1}$, and for $2^n < j < 2^{n+1}$ we have $\nu_2 (j) = \nu_2 (j - 2^n)$ so we can apply induction with the fact that $\cG_{1, 2^n} = \cG_{2^n + 1, 2^{n+1}}$.
Define $m(j) = j - 2^{\nu_2 (j)}$, which is a node that makes $j$ the midpoint of $m(j)$ and $k(j)$.
These nodes can be understood as endpoints of the smallest dyadic block containing $j$; for example, for $j=12$, it is the midpoint of the block $\cG_{8,16}$ enclosed by the arc $(16, 8)$, and for $j=14$, it is the midpoint of $\cG_{12,16}$ enclosed by the arc $(16, 12)$ (see Figure~\ref{figure:fractal-self-dual-N=16}). 

We can characterize the recurrence for FSDM as follows:
\begin{proposition}
\label{proposition:FSDM-recurrence}
For $j=1,2,\dots$, take the binary expansion of $m(j)$ as
\[
    m(j) = 2^{i_1} + 2^{i_2} + \dots + 2^{i_p} , \quad i_1 > i_2 > \dots > i_p 
\]
and let $t_r = \sum_{q=1}^r 2^{i_q}$ for $r=1,\dots,p$ so that $2^{i_1} = t_1 < t_2 < \cdots < t_p = m(j)$.
Then the $y_j$-update for FSDM is given by
\begin{align}
\label{eqn:FSDM-recurrence}
    y_j = y_{j-1} - 2^{\nu_2 (j)} g_{j} + \frac{1}{2} \left( y_{m(j)} - y_{j-1} \right) + \sum_{r=1}^p 2^{\nu_2 (j) - r} g_{t_{p-r+1}} 
\end{align}
where $g_t = \frac{1}{2}\left( y_{t-1} - \opT y_{t-1} \right)$ for $t\ge 1$.
In particular, if $j = 2^n$ so that $m(j) = 0$, the binary expansion becomes a vacuous summation and \eqref{eqn:FSDM-recurrence} simply becomes 
\[
    y_{2^n} = y_{2^n - 1} - 2^n g_{2^n} + \frac{1}{2} \left( y_0 - y_{2^n - 1}\right) .
\]
\end{proposition}

Hence, during the runtime of FSDM, at the $j$-th iteration, one has to keep within the memory the values of $y_0$ and $y_t, g_t$ from active dyadic sub-blocks, whose corresponding iteration numbers have not completely passed.
For example, during iterations $12$ to $16$, one must consistently keep $y_8, y_{12}$ and $g_8, g_{12}$ within the memory stack, but they can be freed after iteration $16$.
This means that the memory requirement for the recurrence~\eqref{eqn:FSDM-recurrence} is proportional to the fractal depth, and does not exceed $2n$ during the first $2^n$ iterations.
In other words, $\cO(\log N)$ memory is required for running $N$ iterations.

The proof of \cref{proposition:FSDM-recurrence}, presented in Appendix~\ref{section:proof-of-fsdm-recurrence}, relies on the H-matrix characterization of FSDM.
For convenience, denote $H^{(0)} = H_\text{FSDM}(0) = [\,\,]$ and
\[
    H^{(n)} = H_{\text{FSDM}}(2^n-1)\in \mathbb R^{(2^n-1)\times(2^n-1)}
\]
for $n=1,2,\dots$.
The structure of $H^{(n)}$ is a fractal---its precise recursive formula is given in \cref{lemma:FSDM-H-matrix}, and it is visually demonstrated in Figure~\ref{figure:fsdm-H-matrix-heatmap}, which is a heatmap of nonzero entries with gray regions indicating zeros.

\noindent
\begin{minipage}[t]{0.65\linewidth}
\vspace{0pt}
\begin{lemma}
\label{lemma:FSDM-H-matrix}
For \(n\ge 1\), the matrices \(H^{(n)}\) take the form
\[
    H^{(n)}
    =
    \left[
    \begin{array}{c|c|c}
    H^{(n-1)} & \\
    \hline
    -\frac12 \big( a^{(n-1)} \big) & 2^{n-2} \\
    \hline
    0_{(2^{n-1}-1) \times (2^{n-1}-1)} & -\frac12 b^{(n-1)} & H^{(n-1)}
    \end{array}
    \right] 
\]
where the row vectors $a^{(n)}$ are determined by the recurrence
\[
    a^{(1)} = \begin{bmatrix} \frac12 \end{bmatrix},
    \qquad
    a^{(n)} = \begin{bmatrix} \frac12 a^{(n-1)} & \frac{2^{n-1}+1}{4} & a^{(n-1)} \end{bmatrix} \in \reals^{2^n - 1}
\]
and $b^{(n)}$ is the column vector obtained by reversing the order of entries in $a^{(n)}$.
In particular, the diagonal entries are $h^{(n)}_{i,i}=2^{\nu_2(i)-1}$ for $i = 1,\dots,2^n-1$.
\end{lemma}
\end{minipage}
\hfill
\begin{minipage}[t]{0.32\linewidth}
\vspace{0pt}
\centering
\includegraphics[width=\linewidth]{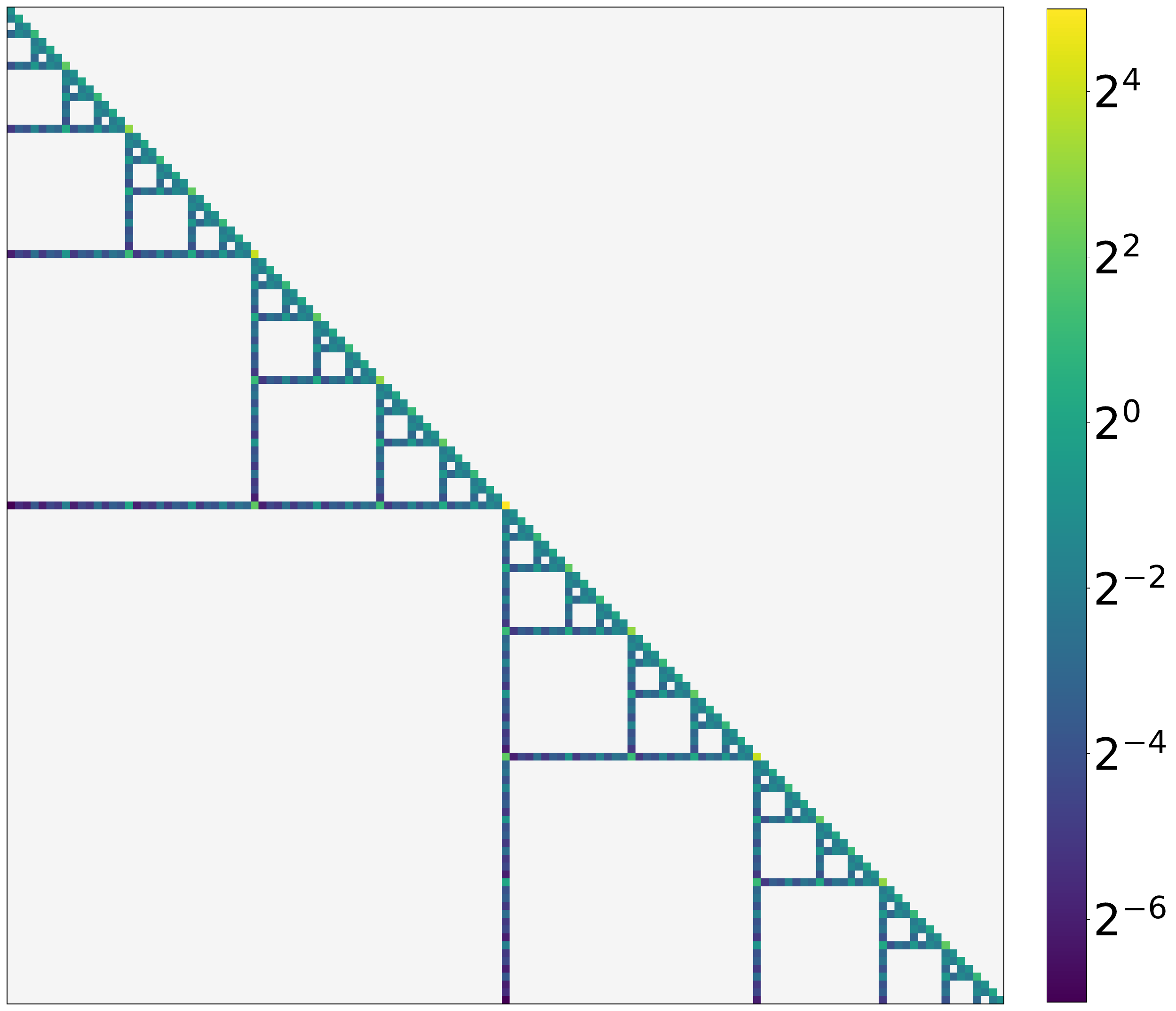}
\captionof{figure}{Heatmap of $H^{(n)}, n=7$.}
\label{figure:fsdm-H-matrix-heatmap}
\end{minipage}

\begin{proof}
We argue by induction on \(n\).
There is nothing to prove in the base case \(n=1\), when we have $H^{(1)} = \begin{bmatrix}\frac12\end{bmatrix}$.

Next, we see that by \cref{theorem:gluing}, applied with $N=2^n$ and $N'=2^{n-1}$, we must have the claimed block matrix structure with
\begin{align*}
    a^{(n-1)}_j = \sum_{i=j}^{2^{n-1}-1} h_{i,j}
\end{align*}
for $j=1,\dots,2^{n-1}-1$, i.e., $a^{(n-1)}$ is the vector obtained by taking the sum of each column from $H^{(n-1)}$.
Because the column sum of $H^{(n)}$ is equal to that of $H^{(n-1)}$ in the last $2^{n-1}-1$ columns, $a^{(n-1)}$ is embedded into the last part of $a^{(n)}$.
The first $2^{n-1}-1$ entries of $a^{(n)}$ get halved from $a^{(n-1)}$ due to the effect of the middle row $-\frac{1}{2} a^{(n-1)}$ in $H^{(n)}$. 

It remains to show that $b^{(n)}$ is the order-reversed version of $a^{(n)}$ and the middle entry of $a^{(n)}$ is $\frac{2^{n-1}+1}{4}$.
The first part follows from the fact that $H^{(n)}$ is self-dual because its arc diagram is.
Then the second claim is equivalent to
\begin{align}
\label{eqn:FSDM-a_nm1-summation}
    2^{n-2} - \frac{1}{2} \sum_{j=1}^{2^{n-1}-1} a_j^{(n-1)} = \frac{2^{n-1} + 1}{4} \iff \sum_{j=1}^{2^{n-1}-1} a_j^{(n-1)} = \frac{2^{n-1} - 1}{2} ,
\end{align}
which can be checked by induction.
For $n=2$ we have $a^{(1)}_1 = \frac{1}{2} = \frac{2^1-1}{2}$.
For the induction step, assuming that~\eqref{eqn:FSDM-a_nm1-summation} holds, we have
\begin{align*}
    \sum_{j=1}^{2^{n}-1} a_j^{(n)} = \frac{3}{2} \sum_{j=1}^{2^{n-1}-1} a_j^{(n-1)} + \frac{2^{n-1} + 1}{4} = \frac{3(2^{n-1} - 1) + 2^{n-1} + 1}{4} = \frac{2^{n+1} - 2}{4} = \frac{2^n - 1}{2} 
\end{align*}
which completes the induction.
\end{proof}

Finally, recursive application of \cref{corollary:intermediate-iterate-guarantee-for-decomposable-algorithm} shows that FSDM converges with the following property.

\begin{proposition}
FSDM exhibits the rate $\sqnorm{y_{2^n-1} - \opT y_{2^n-1}} \le \frac{4\sqnorm{y_0 - y_\star}}{2^{2n}}$ for $n=1,2,\dots$.
\end{proposition}

\subsection{Sum of H-certificates as a measure of robustness}
\label{subsection:robustness}

Consider the setting where the fixed-point operator $\opT\colon \cX \to \cX$ is nonexpansive only up to a bounded violation, i.e.,
\begin{align}
\label{eqn:nonexpansive-with-violation}
    \sqnorm{\opT x - \opT y} \le \sqnorm{x - y} + \epsilon  
\end{align}
for all $x,y \in \cX$ for some fixed $\epsilon > 0$.
For such $\opT$, the proof template \eqref{eqn:optimal-family-proof-core-identity} for $H \in \cH_\star(N-1)$ yields
\begin{align*}
    0 & = N \sqnorm{g_N} + \inprod{g_N}{x_N - y_0} + \sum_{k=1}^N \sum_{j=1}^{k-1} \frac{\lambda_{k,j}^\star (H)}{4} \left( \sqnorm{y_{k-1} - y_{j-1}} - \sqnorm{\opT y_{k-1} - \opT y_{j-1}} \right) \\
    & \ge N \sqnorm{g_N} + \inprod{g_N}{x_N - y_0} - \frac{\epsilon}{4} \sum_{j=1}^{N-1} \sum_{k=j+1}^N \lambda_{k,j}^\star(H) \\
    & \ge N \sqnorm{g_N} + \inprod{g_N}{y_\star - y_0} - \frac{\epsilon}{4} - \frac{\epsilon}{4} \sum_{j=1}^{N-1} \sum_{k=j+1}^N \lambda_{k,j}^\star(H) \\
    & \ge \frac{N}{2} \sqnorm{g_N} - \frac{1}{2N} \sqnorm{y_0 - y_\star} - \frac{\epsilon}{4} \left( 1 + \sum_{j=1}^{N-1} \sum_{k=j+1}^N \lambda_{k,j}^\star(H) \right)
\end{align*}
where the third line follows from \eqref{eqn:nonexpansive-with-violation} applied with $x=y_\star$ and $y=y_{N-1}$ which implies
\begin{align*}
    \inprod{g_N}{x_N - y_\star} & = \inprod{\frac{y_{N-1} - \opT y_{N-1}}{2}}{\frac{y_{N-1} + \opT y_{N-1}}{2} - y_\star } \\
    & = \inprod{\frac{y_{N-1} - y_\star - (\opT y_{N-1} - \opT y_\star)}{2}}{\frac{y_{N-1} - y_\star + \opT y_{N-1} - \opT y_\star}{2}} \\
    & = \frac{1}{4} \left( \sqnorm{y_{N-1} - y_\star} - \sqnorm{\opT y_{N-1} - \opT y_\star} \right) \ge -\frac{\epsilon}{4} ,
\end{align*}
and the last inequality follows from Young's inequality $\inprod{g_N}{y_\star - y_0} \ge -\frac{N}{2} \sqnorm{g_N} - \frac{1}{2N} \sqnorm{y_0 - y_\star}$.
Then we obtain the guarantee
\begin{align*}
    \sqnorm{g_N} \le \frac{\sqnorm{y_0 - y_\star}}{N^2} + \frac{\epsilon}{2N} \left( 1 + \sum_{j=1}^{N-1} \sum_{k=j+1}^N \lambda_{k,j}^\star(H) \right) .
\end{align*}
Through this analysis, we can interpret the sum of H-certificates as the algorithm's sensitivity to the oracle's violation of the nonexpansiveness assumption, so the smaller it is, the more robust the algorithm is to a perturbation of the problem class.
Define
\begin{align}
\label{eqn:sum-of-H-certificates-as-robustness}
    \rho(H) = \sum_{j=1}^{N-1} \sum_{k=j+1}^N \lambda_{k,j}^\star(H) .
\end{align}

\paragraph{Minimum of $\rho(H)$.}
For any optimal $H\in \cH_\star(N-1)$, we have $\sum_{j=1}^{N-1} \lambda_{N,j}^\star(H) = N-1$ by \cref{lemma:Q-lambda-identities}(b) (see also \citep[Theorem~5]{YoonRyuGrimmerInvariance2025}). 
Therefore, $\rho(H) \ge N-1$ and the equality is achieved if and only if $\lambda_{k,j}^\star(H) = 0$ for all $1\le j < k\le N-1$, i.e., iff $H = H_\text{Dual-OHM}(N-1)$.
On the other hand, for basic vertex algorithms, having a guarantee on an intermediate iterate $y_{N'-1}$ (i.e., being decomposable at $N'$) incurs an increase in $\rho(H)$.
More precisely, we have:

\begin{proposition}
\label{proposition:intermediate-guarantee-costs-robustness}
Let $H \in \cH_\star(N-1)$ be an optimal vertex algorithm and suppose that it is decomposable at $1\le N' \le N-1$. 
Then
\[
    \rho(H) \ge N-1 + \frac{N'(N'-1)}{N} .
\]
\end{proposition}

\begin{proof}
Because $H$ is decomposable, we can write $H = H_1 \glue H_2$, where $H_1, H_2$ are optimal vertex algorithms of sizes $N'-1$ and $N-N'-1$, respectively (see the proof of \cref{corollary:intermediate-iterate-guarantee-for-decomposable-algorithm}).
Then, by \cref{theorem:gluing}, we have
\begin{align*}
    \rho(H) = \sum_{j=1}^{N-1} \lambda_{k(j),j}^\star (H) & = \sum_{j=1}^{N'-1} \lambda_{k(j),j}^\star (H) + \lambda_{N,N'}^\star (H) + \sum_{j=N'+1}^{N-1} \lambda_{k(j),j}^\star (H) \\
    & = \frac{N'}{N} \sum_{j=1}^{N'-1} \lambda_{k(j),j}^\star (H_1) + \frac{N'}{N-N'} + \frac{N}{N-N'} \sum_{j=1}^{N-N'-1} \lambda_{k(j),j}^\star (H_2) \\
    & \ge \frac{N'(N'-1)}{N} + \frac{N'}{N-N'} + \frac{N(N-N'-1)}{N-N'} \\
    & = \frac{N'(N'-1)}{N} + N - 1 
\end{align*}
where the inequality uses $\sum_{j=1}^{N'-1} \lambda_{k(j),j}^\star (H_1) = \rho(H_1) \ge N'-1$ and $\sum_{j=1}^{N-N'-1} \lambda_{k(j),j}^\star (H_2) = \rho(H_2) \ge N-N'-1$.
\end{proof}

\paragraph{Maximum of $\rho(H)$.}
The above result suggests that having more intermediate guarantees will increase $\rho(H)$, and OHM is at the extreme in this direction, as it is (the unique) anytime optimal algorithm \citep{YoonRyuGrimmerInvariance2025}.
Naturally, we have
\begin{proposition}
\label{proposition:OHM-is-least-robust}
For $N\ge 2$, OHM satisfies 
\begin{align*}
    \rho (H_\text{OHM}(N-1)) 
    = \sum_{j=1}^{N-1} \lambda_{j+1,j}^\star (H_\text{OHM}(N-1)) = \frac{N^2 - 1}{3} ,
\end{align*}
and it is the unique maximum for $\rho(H)$ over basic optimal vertex algorithms in $\cH_\star(N-1)$.
\end{proposition}

\begin{proof}
For $H = H_\text{OHM}(N-1)$, we have $k(j)=j+1$ and $\lambda_{j+1,j}^\star(H) = \frac{j(j+1)}{N}$ for $j=1,\dots,N-1$, so their sum is
\begin{equation*}
    \rho (H_\text{OHM}(N-1)) = \sum_{j=1}^{N-1} \lambda_{j+1,j}^\star (H_\text{OHM}(N-1)) = \frac{1}{N} \left( \sum_{j=1}^{N-1} (j^2 + j) \right) = \frac{1}{N} \left( \frac{N(N-1)(2N-1)}{6} + \frac{N(N-1)}{2} \right) = \frac{N^2 - 1}{3} .
\end{equation*}
Next, we use induction on $N$ to show that OHM is the unique maximum over basic optimal algorithms.
In the base case $N=2$, $H_\text{OHM}(1) = H_\text{Dual-OHM}(1) = \begin{bmatrix} \frac{1}{2} \end{bmatrix}$ is the only basic optimal algorithm and it has $\lambda_{2,1}^\star (H) = 1$, and we have nothing to prove.

For $N\ge 3$, assume that $\sum_{j=1}^{\ell-1} \lambda_{k(j),j}^\star (H') \le \frac{\ell^2 - 1}{3}$ for all basic optimal $H' \in \cH_\star(\ell-1)$ with equality holding if and only if $H' = H_\text{OHM}(\ell-1)$ for $\ell=2,\dots,N-1$.
Now let $H \in \cH_\star(N-1)$ be a basic optimal algorithm, so that it is decomposable at a unique $N' \in \{1,\dots,N-1\}$ by \cref{lemma:leftmost-arc-to-N-gives-decomposition} and we can write $H = H_1 \glue H_2$.
Then, using the induction hypothesis and \cref{theorem:gluing}, we obtain
\begin{align*}
    \rho(H) = \sum_{j=1}^{N-1} \lambda_{k(j),j}^\star (H) & 
    = \frac{N'}{N} \rho (H_1) + \frac{N'}{N-N'} + \frac{N}{N-N'} \rho (H_2) \\
    & \le \frac{N'}{N} \frac{(N')^2 - 1}{3} + \frac{N'}{N-N'} + \frac{N}{N-N'} \frac{(N-N')^2 - 1}{3} := \sigma(N') .
\end{align*}
Observe that we have $\sigma(N-1) = \frac{N^2 - 1}{3}$, and for $N' < N-1$, we have 
\begin{align*}
    \sigma(N-1) - \sigma(N') = \frac{N'(N+N')((N-N')^2 - 1)}{3N(N-N')} > 0
\end{align*}
from direct simplification. Therefore, we have $\rho(H) \le \frac{N^2 - 1}{3}$ and the equality holds if and only if $N'=N-1$ and $\rho(H_1) = \frac{(N-1)^2 - 1}{3}$, which is equivalent to
\[
    H_1 = H_\text{OHM} (N-2) \iff H = H_\text{OHM} (N-1)
\]
by induction hypothesis and the fact that $H_\text{OHM} (N-1) = H_\text{OHM} (N-2) \glue [\,\,]$. This completes the induction, and hence the proof.
\end{proof}

\paragraph{The trade-off between anytime guarantees and robustness.}
The previous discussion shows that OHM and Dual-OHM lie at opposite ends of the anytime--robustness spectrum. 
Dual-OHM minimizes $\rho(H)$ but provides no guarantees for intermediate iterates, whereas OHM guarantees every intermediate iterate at the cost of robustness, resulting in the largest $\rho(H)$ among optimal vertex algorithms.
The algorithms from the previous section, RDO and FSDM, lie between these two extremes, each offering a different compromise between anytime guarantees and robustness.
Both algorithms are \textit{quasi-anytime}, which we define as having a sequence of iterates $y_{m_k}$ with convergence guarantees, where $\{m_k\}_{k=0,1,\dots}$ is an increasing sequence of natural numbers.
On the other hand, their robustness can be analyzed as follows:
\begin{itemize}
    \item RDO (with period $p$) provides guarantees at linearly spaced iterates, and applying \cref{theorem:gluing} with $N=np+1$ and $N'=(n-1)p+1$ it satisfies
    \[
        \rho\left( H_{\text{RDO}(p)}^{(n)} \right) = \sum_{k=(n-1)p+1}^{np} \lambda_{N,k}^\star \left( H_{\text{RDO}(p)}^{(n)} \right) + \frac{N'}{N} \rho\left( H_{\text{RDO}(p)}^{(n-1)} \right) = np + \frac{(n-1)p+1}{np+1} \rho\left( H_{\text{RDO}(p)}^{(n-1)} \right) 
    \]
    for $n=1,2,\dots$ (with $\rho\left( H_{\text{RDO}(p)}^{(0)} \right) = 0$).
    Therefore, letting $a_n = (np+1) \rho\left( H_{\text{RDO}(p)}^{(n)} \right)$, we obtain the recurrence $a_n = a_{n-1} + np(np+1)$ and $a_0 = 0$, so we obtain
    \begin{align*}
        \rho\left( H_{\text{RDO}(p)}^{(n)} \right) = \frac{a_n}{np+1} = \frac{1}{np+1} \sum_{k=1}^n kp(kp+1) = \frac{pn(n+1)(2np+p+3)}{6(np+1)} = \frac{n^2 p}{3} + \cO(n) \approx \frac{N^2}{3p}.
    \end{align*}
    So the robustness is approximately improved by a factor of $p$ over OHM.

    \item For FSDM, applying \cref{theorem:gluing} with $N=2^n$ and $N'=2^{n-1}$, we observe that for $n=1,2,\dots$
    \begin{align*}
        \rho \left( H_{\text{FSDM}}(2^n-1) \right) & = \frac{N'}{N} \rho \left( H_{\text{FSDM}}(2^{n-1}-1) \right) + \lambda_{N,N'}^\star \left( H_{\text{FSDM}}(2^n-1) \right) + \frac{N}{N-N'} \rho\left( H_{\text{FSDM}}(2^{n-1}-1) \right) \\
        & = \frac{5}{2} \rho \left( H_{\text{FSDM}}(2^{n-1}-1) \right) + 1 
    \end{align*}
    where we use $\lambda_{N,N'}^\star \left( H_{\text{FSDM}}(2^n-1) \right) = \frac{N'}{N-N'} = 1$.
    Therefore, $b_n = \rho \left( H_{\text{FSDM}}(2^{n}-1) \right) + \frac{2}{3}$ is a geometric sequence satisfying $b_n = \frac{5}{2} b_{n-1}$ with $b_0 = \frac{2}{3}$. This implies
    \begin{align*}
        \rho \left( H_{\text{FSDM}}(2^{n}-1) \right) = \frac{2}{3} \left(\frac{5}{2}\right)^n - \frac{2}{3} \approx \frac{2}{3} N^{\log_2 \frac{5}{2}} < \frac{2}{3} N^{1.322} .
    \end{align*}
\end{itemize}

\section{Experiments}
\label{section:experiments}

In this section, we provide numerical demonstrations that the theory we developed is predictive.
In particular, we compare the four algorithms: OHM, Dual-OHM, Repeated Dual-OHM (RDO) and Fractal Self-Dual Method (FSDM).

First, we compare their behavior on the worst-case nonexpansive linear operator $\opT\colon \reals^N \to \reals^N$ from \citep{ParkRyu2022_exact}, defined by 
\begin{align*}
    \opT(x_1, \dots, x_N) = \begin{bmatrix} 
    0 & 0 & \cdots & -1 \\ 
    1 & 0 & \cdots & 0  \\
    \vdots & \ddots & \ddots & \vdots \\
    0 & \cdots & 1 & 0
    \end{bmatrix} \begin{bmatrix}
        x_1 \\ x_2 \\ \vdots \\ x_N
    \end{bmatrix}
    = (-x_N, x_1, \dots, x_{N-1})
\end{align*}
which has the unique fixed point $y_\star = 0$.
We set $N=1024$ and use the period $p=33$ (a divisor of $N-1$) for RDO, and plot $\sqnorm{y_k - \opT y_k}$.
For the initial point, we set $y_0 = -\frac{1}{\sqrt{N}} (1,1,\dots,1)$.
For this problem, OHM converges most steadily---as it is an anytime algorithm---while the other algorithms match its performance at the iterations where they have convergence guarantees (Figure~\ref{fig:worst-case}).

Next, we replace $\opT$ with a non-worst-case operator, namely the contractive operator $\gamma\opT$ with $\gamma=0.975$ (Figure~\ref{fig:non-worst-case}), while keeping all the other parameters fixed.
For this setting, we observe that Dual-OHM achieves the terminal accuracy fastest (within the first quarter of iterations).
Notably, FSDM displays a staircase behavior, having sharp progress at iteration numbers of the form $2^k - 1$.
RDO behaves similarly, but periodically and with less pronounced improvements.

\begin{figure}[t]
    \centering
    \begin{subfigure}[t]{0.4\linewidth}
        \centering
        \includegraphics[width=\linewidth]{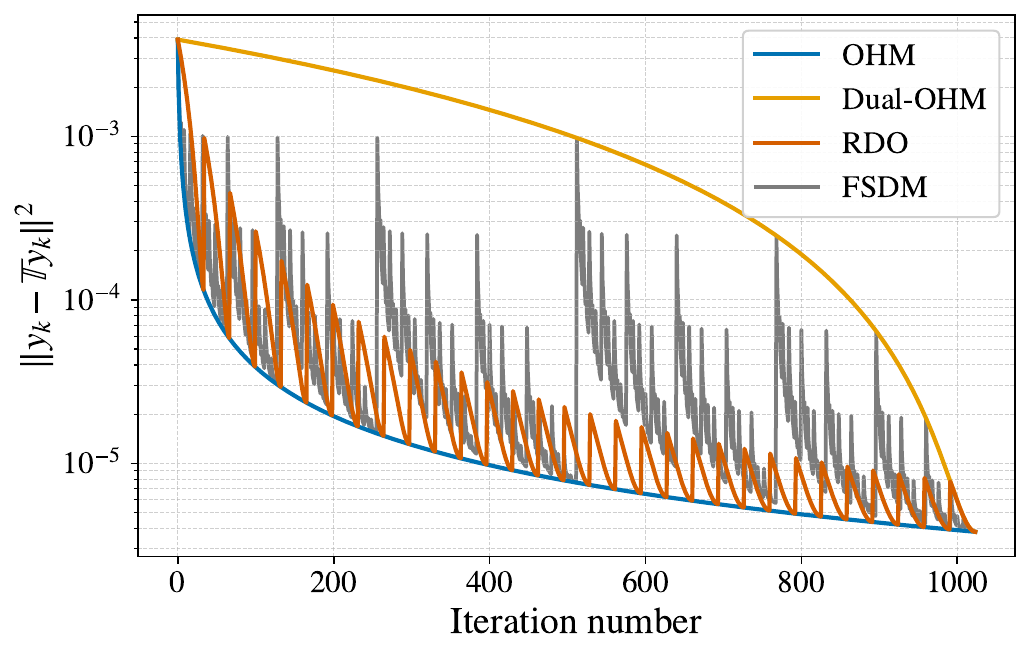}
        \caption{Worst-case $\opT$ from \citep{ParkRyu2022_exact}}
        \label{fig:worst-case}
    \end{subfigure}
    \hspace{0.04\linewidth}
    \begin{subfigure}[t]{0.4\linewidth}
        \centering
        \includegraphics[width=\linewidth]{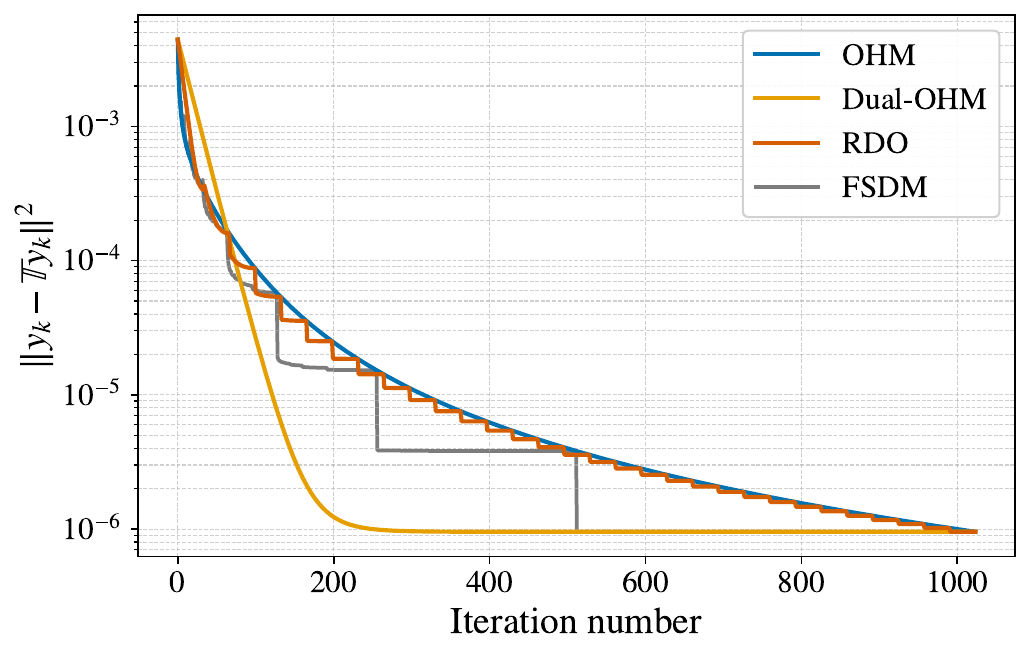}
        \caption{Non-worst-case (contractive) $\gamma\opT$}
        \label{fig:non-worst-case}
    \end{subfigure}

    \vspace{0.8em}

    \begin{subfigure}[t]{0.4\linewidth}
        \centering
        \includegraphics[width=\linewidth]{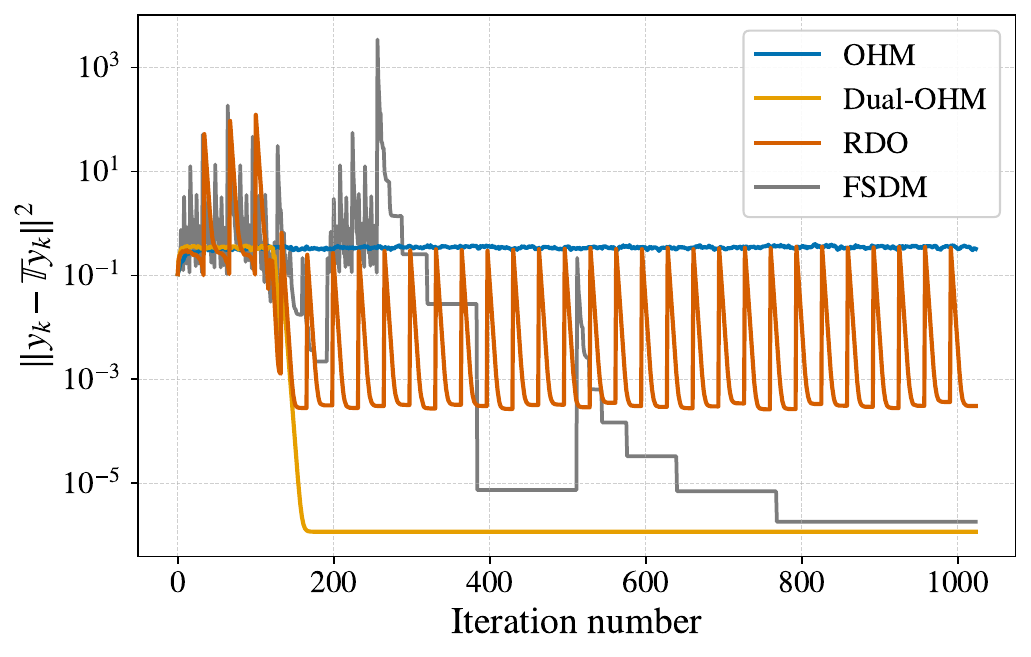}
        \caption{Operator $\opT_\delta$ with bounded violation}
        \label{fig:bounded-violation}
    \end{subfigure}
    \hspace{0.04\linewidth}
    \begin{subfigure}[t]{0.4\linewidth}
        \centering
        \includegraphics[width=\linewidth]{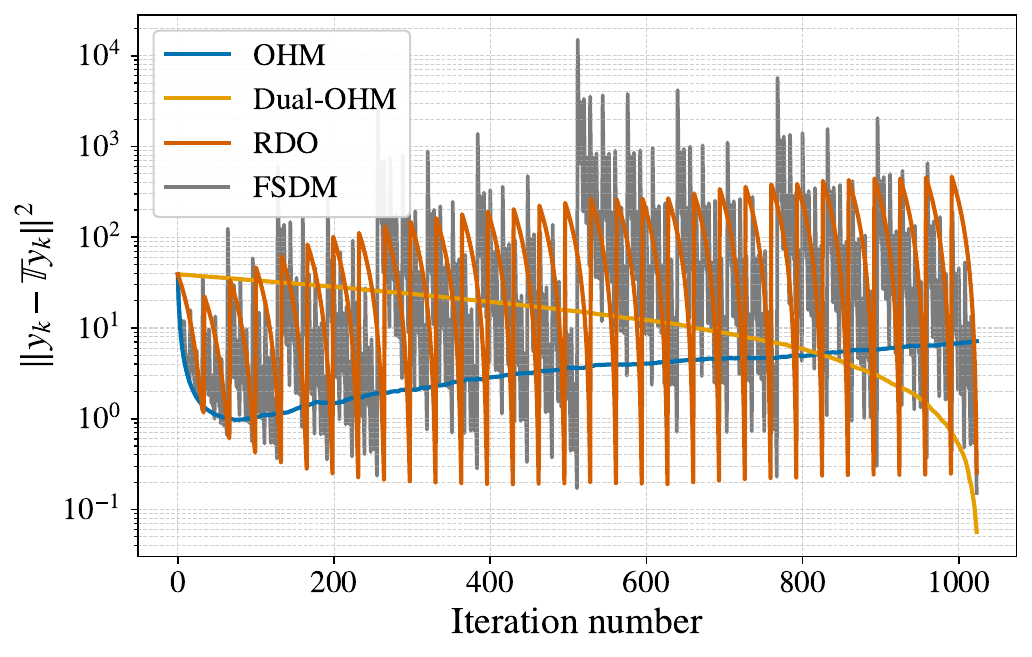}
        \caption{Operator with unbounded violation}
        \label{fig:unbounded-violation}
    \end{subfigure}

    \caption{Performance comparison of OHM, Dual-OHM, RDO (periodic), and FSDM on four distinct fixed-point operators $\opT\colon \reals^d \to \reals^d$. We take $N=1024$, $d=N$ for the first two experiments, and $d=2N$ for the last two experiments. For RDO, we run the periodic version with $p=33$.}
    \label{figure:performance-plots}
\end{figure}

Finally, we construct $\opT_\delta \colon \reals^{2N} \to \reals^{2N}$ satisfying \eqref{eqn:nonexpansive-with-violation} by using a sign hash function $\mathsf{Hash}(\cdot)$ which, given a vector $w \ne \mathbf{0}$ of arbitrary length, outputs a pseudo-random vector of the same length whose entries are $\pm1$. For $\delta > 0$, define
\[
    \opT_\delta(x_1, x_2) = \left( \gamma\opT x_1 , \gamma\opT x_2 + \frac{\delta}{2\sqrt{N}} \mathsf{Hash}(x_1) \right)
\]
where $\opT\colon \reals^N \to \reals^N$ is the worst-case linear operator defined above and $0 < \gamma < 1$ is a contraction factor. Then, we have
\begin{align*}
    \sqnorm{\opT_\delta (x_1, x_2) - \opT_\delta (y_1, y_2)} & \le \frac{1}{\gamma^2} \sqnorm{\big( \gamma (\opT x_1 - \opT y_1) , \gamma (\opT x_2 - \opT y_2) \big) } + \frac{1}{1-\gamma^2} \sqnorm{\delta v(x_1,y_1)} \\
    & \le \sqnorm{(x_1, x_2) - (y_1, y_2)} + \frac{\delta^2}{1-\gamma^2}
\end{align*}
where $v(x_1, y_1) = \frac{1}{2\sqrt{N}} (\mathsf{Hash}(x_1) - \mathsf{Hash}(y_1)) \in \reals^N$ has all its entries within $\left[-\frac{1}{\sqrt{N}}, \frac{1}{\sqrt{N}} \right]$, from which the second inequality follows.
This shows that $\opT_\delta$ satisfies \eqref{eqn:nonexpansive-with-violation} with $\epsilon = \frac{\delta^2}{1-\gamma^2}$.
We set $\mathsf{Hash}(\mathbf{0}) = \mathbf{0}$ so that $\mathbf{0}_{2N}$ is a fixed point of $\opT_\delta$.

In Figure~\ref{fig:bounded-violation}, we set $y_0 = -\frac{1}{\sqrt{N}} (\mathbf{1}_N, \mathbf{0}_N)$, $\gamma=0.8$ and $\delta=0.5$, and run all algorithms with $N=1024$ and $p=33$.
As the theory predicts, we observe that Dual-OHM exhibits the most stable behavior against violation of nonexpansiveness, OHM performs worst, and RDO and FSDM display in-between performance.
In Figure~\ref{fig:unbounded-violation}, we consider the case $\gamma=1$: the squared-norm violation bound~\eqref{eqn:nonexpansive-with-violation} becomes unbounded as $\gamma\to 1$, and we instead have $\norm{\opT_\delta (x_1,x_2) - \opT_\delta (y_1,y_2)} \le \norm{(x_1, x_2) - (y_1, y_2)} + \delta$.
For this setting, we take the initial point $y_0 = -\frac{R}{\sqrt{N}} (\mathbf{1}_N, \mathbf{0}_N)$ with $R=10^2$.
Interestingly, even in this case, we observe the same hierarchy of increasing robustness $\text{OHM} < \text{RDO}(p) < \text{FSDM} < \text{Dual-OHM}$.

\section{Conclusion}

Together with \citep{YoonRyuGrimmerInvariance2025}, this work marks another step toward a more comprehensive understanding of optimal acceleration for fixed-point algorithms and the closely related settings of monotone inclusions and minimax optimization \citep{MokhtariOzdaglarPattathil2020_unified, yoonAcceleratedMinimaxAlgorithms2025a}.
The story of acceleration for these problems began with Halpern iterations, or so-called anchor-type algorithms \citep{Halpern1967_fixed, Lieder2021_convergence, Kim2021_accelerated, Diakonikolas2020_halpern, YoonRyu2021_accelerated, LeeKim2021_fast, Tran-DinhLuo2021_halperntype}.
These methods have been studied extensively and further developed through connections to ODE models and Nesterov acceleration \citep{botFastOptimisticGradient2025, SuhParkRyu2023_continuoustime, tran-dinhHalpernsFixedpointIterations2024}.
Although the algorithms proposed in the aforementioned works are formally distinct, they can still be broadly categorized as anchor-type variants, since they all focus on deriving anytime or asymptotic convergence guarantees.
H-duals of anchor-type algorithms (dual-anchor type) \citep{YoonKimSuhRyu2024_optimal} were then proposed as a materially different acceleration mechanism that does not provide intermediate-iterate guarantees.
While this initially appeared to be a drawback, our results suggest that it is better understood as a benign trade-off for robustness.
Together with the new algorithms we propose, which attain distinct midpoints along this trade-off curve, this may enable a formal study of Pareto optimality and the extension of this perspective to stochastic fixed-point problems and similar settings, offering a distinct approach from existing work on the topic \citep{CaiSongGuzmanDiakonikolas2022_stochastic, bravoStochasticFixedpointIterations2024, caiVarianceReducedHalpern2024, 
chenNearoptimalAlgorithmsMaking2024, bravoStochasticHalpernIteration2026, pischkeAsymptoticRegularityGeneralised2026}.
One could even develop additional tie-breaking criteria beyond intermediate-iterate guarantees or robustness to noise, such as average-case behavior.
Our composing theory may provide a tractable handle and insights for addressing these research directions in future work.

\section*{Acknowledgments}

Benjamin Grimmer's work was supported as a fellow of the Alfred P.\ Sloan Foundation.

\bibliographystyle{plainnat}
\bibliography{ref}

\appendix

\section{Proof of \texorpdfstring{\cref{lemma:nonzero-lambda-exists}}{Lemma 2}}
\label{section:proof-of-sparsity-lemma}

We first state and prove some useful combinatorial lemmas.

\begin{lemma}[Chu--Vandermonde identity]
\label{lemma:chu-vandermonde}
For any $\alpha,\beta\in\reals$ and any integer $n\ge 0$,
\[
    \sum_{m=0}^{n} \binom{\alpha}{m}\binom{\beta}{n-m}
    = \binom{\alpha+\beta}{n} .
\]
\end{lemma}

\begin{proof}
Compare the coefficient of $t^n$ in the following identity using the formal binomial series expansion:
\[
    \sum_{k=0}^\infty \binom{\alpha+\beta}{k} t^k = (1+t)^{\alpha+\beta} = (1+t)^\alpha(1+t)^\beta = \left( \sum_{i=0}^\infty \binom{\alpha}{i} t^i \right) \left( \sum_{j=0}^\infty \binom{\beta}{j} t^j \right) .
\]
\end{proof}

\begin{lemma}
\label{lemma:pascal-matrix-properties}
Let $r$ be a positive integer. Define $\vS_r,\vP_r\in\reals^{r\times r}$ by
\begin{align*}
    \vS_r = \begin{bmatrix}
        \binom{1}{0} & -\binom{2}{0} & \cdots & (-1)^{r-1} \binom{r}{0} \\
        -\binom{2}{1} & \binom{3}{1} & \cdots & (-1)^{r} \binom{r+1}{1} \\
        \vdots & \vdots & & \vdots \\
        (-1)^{r-1} \binom{r}{r-1} & (-1)^{r} \binom{r+1}{r-1} & \cdots & \binom{2r-1}{r-1}
    \end{bmatrix} , \quad & \text{i.e.,} \quad (\vS_r)_{a,b} = (-1)^{a+b}\binom{a+b-1}{a-1}, \\
    \vP_r = \begin{bmatrix}
         \binom{2}{1} & -\binom{3}{1} & \cdots & (-1)^{r+1} \binom{r+1}{1} \\
        -\binom{3}{2} & \binom{4}{2} & \cdots & (-1)^{r+2} \binom{r+2}{2} \\ 
        \vdots & \vdots & \cdots & \vdots \\
        (-1)^{r+1} \binom{r+1}{r} & (-1)^{r+2} \binom{r+2}{r} & \cdots & \binom{2r}{r} 
    \end{bmatrix} , \quad & \text{i.e.,} \quad  (\vP_r)_{a,b} = (-1)^{a+b}\binom{a+b}{a} .
\end{align*}
Then the following properties hold:
\begin{enumerate}[label=\normalfont(\roman*)]
    \item $\det \vS_r = 1$; therefore, every leading principal minor of $\vS_r$ is also $1$.
    \item The last row and last column of $\vS_r^{-1}$ are given by
    \[
        \ve_r^{\transpose} \vS_r^{-1} = 
    \begin{bmatrix}
        \binom{r}{1} & \binom{r}{2} & \cdots & \binom{r}{r}
    \end{bmatrix} := \vb_r^{\transpose} , \qquad 
        \vS_r^{-1} \ve_r =
        \begin{bmatrix}
            \binom{r-1}{0} & \binom{r-1}{1} & \cdots & \binom{r-1}{r-1}
        \end{bmatrix}^{\transpose} := \vc_r .
    \]
    \item $\vP_r = \vR_r^\transpose \vR_r$, where 
    \begin{align*}
        \vR_r = \begin{bmatrix}
            \binom{1}{0} & -\binom{2}{0} & \cdots & (-1)^{r-1} \binom{r}{0} \\
            \binom{1}{1} & -\binom{2}{1} & \cdots & (-1)^{r-1} \binom{r}{1} \\
            & -\binom{2}{2} & \cdots & (-1)^{r-1} \binom{r}{2} \\
            & & \ddots & \vdots \\
            & & & (-1)^{r-1} \binom{r}{r}
        \end{bmatrix} \in \reals^{(r+1) \times r}
        , \quad \text{i.e.,} \quad
        (\vR_r)_{a,b} = (-1)^{b-1}\binom{b}{a-1} .
    \end{align*}
    In particular, $\vP_r$ is symmetric positive definite.
    \item Let $\vd_r = \left( \vc_{r+1} \right)_{1:r} = \begin{bmatrix}
        \binom{r}{0} & \binom{r}{1} & \cdots & \binom{r}{r-1}
    \end{bmatrix}^{\transpose} \in \reals^r$. Then
    \[
        \vR_r \vd_r 
        = (-1)^{r+1}\left(\vd_{r+1} - \ve_{r+1}\right).
    \]
\end{enumerate}
\end{lemma}

\begin{proof}
We first prove that $\vS_r \vc_r = \ve_r$ and $\vS_r^\transpose \vb_r = \ve_r$. Then, once we show (i), this will immediately imply (ii).
For $a=1,\dots,r$, using $\binom{-a}{u}=(-1)^u\binom{a+u-1}{u}$ for $u=0,1,\dots$ and \cref{lemma:chu-vandermonde},
\begin{align*}
    (\vS_r \vc_r)_a
    & = \sum_{b=1}^{r} (-1)^{a+b}\binom{a+b-1}{a-1}\binom{r-1}{b-1} = (-1)^{a}\sum_{b=1}^{r} (-1)^{b} \binom{a+b-1}{b}\binom{r-1}{b-1} \\
    & = (-1)^{a} \sum_{b=1}^{r} \binom{-a}{b}\binom{r-1}{r-b} = (-1)^a \binom{r-a-1}{r} .
\end{align*}
If $a<r$, then $\binom{r-1-a}{r}=0$; if $a=r$ then $\binom{-1}{r}=(-1)^r$ so we see that the above expression becomes $1$. This shows that $\vS_r \vc_r = \ve_r$.
Similarly, for $a=1,\dots,r$, we have $\binom{-a-1}{u} = (-1)^u \binom{a+u}{u}$ for $u=0,1,\dots$, so
\begin{align*}
    \left(\vS_r^\transpose \vb_r \right)_a = \sum_{b=1}^r (-1)^{a+b} \binom{a+b-1}{b-1} \binom{r}{b} = (-1)^{a-1} \sum_{b=1}^r \binom{-a-1}{b-1} \binom{r}{r-b} = (-1)^{a-1} \binom{r-a-1}{r-1} =
    \begin{cases}
        0 & \text{if } a \le r-1 \\
        1 & \text{if } a = r
    \end{cases}
\end{align*}
which shows that $\vS_r^\transpose \vb_r = \ve_r$.

\smallskip

Now we prove (i). Let $\vU_r = \begin{bmatrix} \ve_1 & \cdots & \ve_{r-1} & \vc_r \end{bmatrix}$. Because the last entry of $\vc_r$ is $1$, $\vU_r$ is upper triangular with unit diagonal, so $\det \vU_r = 1$. Since $\vS_r \vc_r = \ve_r$, we have
\[
    \vS_r \vU_r =
    \begin{bmatrix}
        \vS_{r-1} & 0 \\
        * & 1
    \end{bmatrix}.
\]
Therefore, $\det \vS_r = \det(\vS_r \vU_r)=\det \vS_{r-1}$. 
Then we can start from the base case $\det \vS_1=1$ and apply induction to conclude that $\det \vS_r=1$ for all $r\ge 1$. 
The statement on leading principal minors follows immediately, because the leading $s\times s$ principal submatrix of $\vS_r$ is exactly $\vS_s$ for $s<r$.

For (iii), observe that
\begin{align*}
    (\vR_r^\transpose \vR_r)_{i,j}
    & = \sum_{a=1}^{r+1} (-1)^{i-1}\binom{i}{a-1}(-1)^{j-1}\binom{j}{a-1} = (-1)^{i+j}\sum_{t=0}^{\min\{i,j\}} \binom{i}{t}\binom{j}{t} \\
    & = (-1)^{i+j}\sum_{t=0}^{\min\{i,j\}} \binom{i}{t}\binom{j}{j-t} = (-1)^{i+j}\binom{i+j}{j}
    = (\vP_r)_{i,j},
\end{align*}
where the second last equality follows from \cref{lemma:chu-vandermonde}. 
This shows $\vP_r = \vR_r^\transpose \vR_r$. 
Furthermore, the bottom $r\times r$ block of $\vR_r$ is upper triangular with diagonal entries $\pm 1$, so $\vR_r$ has full column rank. 
This implies that $\vP_r = \vR_r^\transpose \vR_r$ is positive definite.

Finally, we prove (iv). 
For $a=1,\dots,r+1$, we have
\begin{align}
\label{eqn:Rrdr-entry-identity}
    (\vR_r\vd_r)_a = \sum_{b=1}^{r}(-1)^{b-1}\binom{b}{a-1}\binom{r}{b-1} = \sum_{m=0}^{r-1} (-1)^{m} \binom{r}{m} \binom{m+1}{a-1} 
\end{align}
where we make the change of index $m=b-1$.
Next, note that for any $t=0,\dots,r-1$, we have $\binom{r}{m} \binom{m}{t} = \binom{r}{t} \binom{r-t}{m-t}$ for $m=t,\dots,r-1$, so
\begin{align}
\label{eqn:binom-rm-binom-mt-identity}
    \sum_{m=0}^{r-1}(-1)^m \binom{r}{m} \binom{m}{t} = \binom{r}{t} \sum_{m=t}^{r-1} (-1)^m\binom{r-t}{m-t} = (-1)^{r+1}\binom{r}{t}
\end{align}
where in the first equality, we use $\binom{m}{t} = 0$ for $m<t$, and the second equality uses
\[
    \sum_{m=t}^{r-1} (-1)^m\binom{r-t}{m-t} = \sum_{j=0}^{r-t-1} (-1)^{j+t} \binom{r-t}{j} = (-1)^t \left( -(-1)^{r-t} \binom{r-t}{r-t} \right) = (-1)^{r+1} .
\]
Now, when $a=1$, \eqref{eqn:Rrdr-entry-identity} becomes $(\vR_r\vd_r)_1 = \sum_{m=0}^{r-1} (-1)^{m} \binom{r}{m} = (-1)^{r+1}$, which is the first entry of $(-1)^{r+1} (\vd_{r+1} - \ve_{r+1})$.
For $a\ge 2$, we use $\binom{m+1}{a-1} = \binom{m}{a-1} + \binom{m}{a-2}$ to rewrite it as
\begin{align}
    (\vR_r\vd_r)_a & = \sum_{m=0}^{r-1} (-1)^{m} \binom{r}{m} \left( \binom{m}{a-1} + \binom{m}{a-2} \right) \nonumber \\
    & = \sum_{m=0}^{r-1}(-1)^m \binom{r}{m} \binom{m}{a-1} + \sum_{m=0}^{r-1}(-1)^m \binom{r}{m} \binom{m}{a-2} . \label{eqn:Rrdr-entry-decomposed-summation}
\end{align}
If $a\le r$, applying \eqref{eqn:binom-rm-binom-mt-identity}, the last expression yields
\begin{align*}
    (\vR_r\vd_r)_a = (-1)^{r+1} \binom{r}{a-1} + (-1)^{r+1} \binom{r}{a-2} = (-1)^{r+1} \binom{r+1}{a-1} = (-1)^{r+1} (\vd_{r+1})_a .
\end{align*}
When $a=r+1$, the first summation in \eqref{eqn:Rrdr-entry-decomposed-summation} vanishes, and applying \eqref{eqn:binom-rm-binom-mt-identity} to the second summation yields
\begin{align*}
    (\vR_r\vd_r)_a = (-1)^{r+1} \binom{r}{r-1} = (-1)^{r+1} \left( (\vd_{r+1})_{r+1} - 1 \right)
\end{align*}
which proves the desired identity.

\end{proof}

\begin{proof}[Proof of~\cref{lemma:nonzero-lambda-exists}]
Fix $j\in\{1,\dots,N-1\}$; write $r=N-j$ and let
\[
    \vq_j =
    \begin{bmatrix}
        Q(1,j) & Q(2,j) & \cdots & Q(r,j)
    \end{bmatrix}^{\transpose} \in \reals^r .
\]
Assume to the contrary that $H \in \cH_\star(N-1)$ and $\lambda_{j+1,j}^\star=\lambda_{j+2,j}^\star=\cdots=\lambda_{N,j}^\star=0$. 
We show that this forces $\vS_r \vq_j = 0$ by induction.
The starting point is clear: we have
\[
    0 = \frac{1}{N}\lambda_{N,j}^\star = \sum_{a=1}^{r} (-1)^{a-1} Q(a,j) = \left(\vS_r \right)_{1,:} \vq_j = \left(\vS_r \vq_j\right)_1 .
\]
Next, assuming that we have already established $\left(\vS_r \vq_j\right)_\ell = 0$ for all $\ell=1,\dots,s$, where $1\le s < r$, we have
\begin{align*}
    0 = \lambda_{N-s,j}^\star & = N \sum_{a=1}^{s} \sum_{b=1}^{N-j} (-1)^{a+b-1} \binom{a+b}{a} Q(a,N-s) Q(b,j) \\
    & = N \sum_{a=1}^{s} Q(a,N-s) \sum_{b=1}^{r} (-1)^{a+b+1} \binom{a+b}{b} Q(b,j) \\
    & = N \sum_{a=1}^{s} Q(a,N-s) \left( \vS_r \right)_{a+1,:} \vq_j \\
    & = N Q(s,N-s) \left( \vS_r \right)_{s+1,:} \vq_j 
\end{align*}
where the last equality uses $\left( \vS_r \right)_{\ell,:} \vq_j = 0$ for $\ell=2,\dots,s$ by the induction hypothesis.
Because $H \in \cH_\star(N-1)$, it satisfies the H-invariance conditions~\eqref{eqn:H-invariance}, and in particular $Q(N-1,1) = h_{1,1} h_{2,2} \cdots h_{N-1,N-1} = \frac{1}{N}\binom{N}{N} = \frac{1}{N} \ne 0$.
Therefore, $Q(s,N-s) = h_{N-s,N-s} h_{N-s+1,N-s+1} \dots h_{N-1,N-1} \ne 0$, and this implies $\left( \vS_r \right)_{s+1,:} \vq_j = 0$, completing the induction.

Now we have established $\vS_r \vq_j = 0$. However, by \cref{lemma:pascal-matrix-properties}(i), $\vS_r$ is nonsingular, so this implies $\vq_j=0$. In particular,
\[
    Q(r,j)=Q(N-j,j)=h_{j,j}h_{j+1,j+1} \cdots h_{N-1,N-1} = 0 ,
\]
which is a contradiction to $Q(N-1,1) = h_{1,1} h_{2,2} \cdots h_{N-1,N-1} \ne 0$.
\end{proof}

\section{Proof of \texorpdfstring{\cref{theorem:all-sparsity-patterns-give-optimal-vertex}}{Theorem 1}}
\label{section:vertex-algorithms-existence-proof}

The proof is structured as follows: we show that if $H\in \cH_\star(N-1)$ satisfying $\lambda_{\ell,j}^\star = 0$ for all $1\le j < \ell \le N$ with $\ell \ne k(j)$ exists, then this sparsity pattern of $\lambda^\star$ enforces the direction of $\begin{bmatrix} Q(1,j) & Q(2,j) & \cdots & Q(N-j,j) \end{bmatrix}^\transpose$ for each $j=1,\dots,N-1$.
That is, when we normalize it to have the last entry $1$, it must be a fixed vector $\vv_j \in \reals^{N-j}$ with positive entries that is determined solely from the indices $k(N-1),k(N-2),\dots,k(1)$.
Then, we show that the H-invariance conditions~\eqref{eqn:H-invariance} in fact determine the diagonal values $Q(N-j,j) > 0$, and hence the entire Q-profile. 
Then, by virtue of Meta Algorithm~\ref{meta-alg:find-H-from-Q}, this in turn recovers $H$.
This argument proves both uniqueness and existence of $H$, constructively.

For $j=N-1$ we have $k(N-1) = N$, and we simply have $\vv_{N-1} = \begin{bmatrix} 1 \end{bmatrix}$ and $Q(1,N-1) \ne 0$ because $H \in \cH_\star(N-1)$ (nothing to prove). 
Next, we use backward induction to recursively construct $\vv_j$ for $j=N-2,N-3,\dots,1$, and show that $\begin{bmatrix} Q(1,j) & Q(2,j) & \cdots & Q(N-j,j) \end{bmatrix}^\transpose \in \mathrm{Sp}\{\vv_j\}$.
For the base case $j=N-2$, we divide into cases—either \textbf{(i)} $k(N-2) = N-1$ or \textbf{(ii)} $k(N-2) = N$.
In the former case, we have
\begin{align*}
    \lambda^\star_{N,N-2} = N(Q(1,N-2) - Q(2,N-2)) = 0 \implies \begin{bmatrix} Q(1,N-2) \\ Q(2,N-2) \end{bmatrix} \in \spann \left\{\begin{bmatrix} 1 \\ 1 \end{bmatrix}\right\} ,
\end{align*}
and in the latter case, we have 
\begin{align*}
    \lambda^\star_{N-1,N-2} = N \cdot Q(1,N-1) \left( -2Q(1,N-2) + 3Q(2,N-2) \right) = 0 \implies \begin{bmatrix} Q(1,N-2) \\ Q(2,N-2) \end{bmatrix} \in \spann \left\{\begin{bmatrix} \frac{3}{2} \\ 1 \end{bmatrix}\right\} .
\end{align*}
In both cases, we can take $\vv_{N-2} = \begin{bmatrix} * & 1 \end{bmatrix}^\transpose$ and $\begin{bmatrix} Q(1,N-2) & Q(2,N-2) \end{bmatrix}^{\transpose}$ is a nonzero multiple of $\vv_{N-2}$ since $Q(2,N-2) \ne 0$.

Now let $j \in \{1,\dots,N-3\}$, and suppose for all $k=j+1,\dots,N-2$ we have
\begin{align}
\label{eqn:Q-vector-belongs-to-vk-span}
    \mathbf{0} \ne \begin{bmatrix} Q(1,k) & \cdots & Q(N-k,k) \end{bmatrix}^\transpose \in \spann\{\vv_k\} 
\end{align}
where $(\vv_k)_{N-k} = 1$.

Consider the equations $\lambda_{k,j}^\star = 0$ for $k=N,N-1,\dots,j+1$, $k \ne k(j)$.
We will show that each of these equations can be reduced to a linear equation in $Q(1,j), \dots, Q(N-j,j)$.
When $k=N$ (which is considered only when $k(j) < N$), this is immediate: by \eqref{eqn:H-certificate-N},
\begin{align*}
    \lambda_{N,j}^\star = 0 \iff Q(1,j) - Q(2,j) + \cdots + (-1)^{N-j-1} Q(N-j,j) = 0 .
\end{align*}
For $k=N-1$ (which is considered only when $k(j) \ne N-1$), by \eqref{eqn:H-certificate-others}, we also immediately have
\begin{align*}
    \lambda_{N-1,j}^\star = 0 & \iff Q(1,N-1) \left( -2Q(1,j) + 3Q(2,j) + \cdots + (-1)^{N-j} (N-j+1) Q(N-j,j) \right) = 0 \\
    & \iff -2Q(1,j) + 3Q(2,j) + \cdots + (-1)^{N-j} (N-j+1) Q(N-j,j) = 0 .
\end{align*}
For any other $k \ne k(j)$ that is less than $N-1$, we use \eqref{eqn:H-certificate-others} and the induction hypothesis \eqref{eqn:Q-vector-belongs-to-vk-span} to obtain 
\begin{align*}
    \lambda_{k,j}^\star & = N\sum_{\ell=1}^{N-j} \sum_{m=1}^{N-k} (-1)^{\ell+m-1} \binom{\ell+m}{m} Q(\ell,j) Q(m,k) \\
    & = N \begin{bmatrix} Q(1,k) & \cdots & Q(N-k,k) \end{bmatrix} 
    \underbrace{\begin{bmatrix}
        -\binom{2}{1} & \binom{3}{1} & \cdots & (-1)^{N-j} \binom{N-j+1}{1} \\
        \binom{3}{2} & -\binom{4}{2} & \cdots & (-1)^{N-j+1} \binom{N-j+2}{2} \\ 
        \vdots & \vdots & \cdots & \vdots \\
        (-1)^{N-k} \binom{N-k+1}{N-k} & (-1)^{N-k+1} \binom{N-k+2}{N-k} & \cdots & (-1)^{2N-k-j-1} \binom{2N-k-j}{N-k} 
    \end{bmatrix}}_{:=\vA_{k,j}}    
    \begin{bmatrix} Q(1,j) \\ \vdots \\ Q(N-j,j) \end{bmatrix} \\
    & = N \cdot Q(N-k,k) \vv_k^\transpose \vA_{k,j} \begin{bmatrix} Q(1,j) \\ \vdots \\ Q(N-j,j) \end{bmatrix} \\
    & = N \cdot Q(N-k,k) \underbrace{\begin{bmatrix} \vv_k^\transpose & 0 & \cdots & 0 \end{bmatrix}}_{\in \reals^{N-j-1}} \vA_{j+1,j} \begin{bmatrix} Q(1,j) \\ \vdots \\ Q(N-j,j) \end{bmatrix} \\
    & = N \cdot Q(N-k,k) \underbrace{\begin{bmatrix} 0 & \vv_k^\transpose & 0 & \cdots & 0 \end{bmatrix}}_{\in \reals^{N-j}} \vS_{N-j} 
    \begin{bmatrix} Q(1,j) \\ \vdots \\ Q(N-j,j) \end{bmatrix}
\end{align*}
where $\vS_{N-j}$ is the $(N-j) \times (N-j)$ Pascal-type matrix from \cref{lemma:pascal-matrix-properties}; in particular, it is nonsingular by part~(i):
\begin{align*}
    \vS_{N-j} = \begin{bmatrix}
        1 & -1 & \cdots & (-1)^{N-j-1} \\
        -\binom{2}{1} & \binom{3}{1} & \cdots & (-1)^{N-j} \binom{N-j+1}{1} \\
        \binom{3}{2} & -\binom{4}{2} & \cdots & (-1)^{N-j+1} \binom{N-j+2}{2} \\ 
        \vdots & \vdots & \cdots & \vdots \\
        (-1)^{N-j-1} \binom{N-j}{N-j-1} & (-1)^{N-j} \binom{N-j+1}{N-j-1} & \cdots & \binom{2N-2j-1}{N-j-1} 
    \end{bmatrix} .
\end{align*}
Now consider the $(N-j) \times (N-j)$ lower triangular matrix
\begin{align}
\label{eqn:matrix-Lj}
    \vL_j = \begin{bmatrix}
        1 & 0 & 0 & 0 & \multicolumn{2}{c}{\cdots} & \multicolumn{2}{c}{\cdots} & 0 \\
        0 & 1 & 0 & 0 & \multicolumn{2}{c}{\cdots} & \multicolumn{2}{c}{\cdots} & 0 \\
        0 & \multicolumn{2}{c}{\vv_{N-2}^\transpose} & 0 & \multicolumn{2}{c}{\cdots} & \multicolumn{2}{c}{\cdots} & 0 \\
        \vdots &  & \vdots &  & \ddots \\
        0 & \multicolumn{4}{c}{\text{------\,}\vv_{k(j)}^\transpose\text{\,------}} & 0 & \cdots & \cdots & 0 \\
        0 & \multicolumn{5}{c}{\text{--------\,}\vv_{k(j)-1}^\transpose\text{\,---------}} & 0 & \cdots & 0 \\
        \vdots & \multicolumn{6}{c}{\ddots} & \ddots & \vdots \\
        0 & \multicolumn{7}{c}{\text{-----------------\,}\vv_{j+2}^\transpose\text{\,-----------------}} & 0 \\
        0 & \multicolumn{8}{c}{\text{--------------------\,}\vv_{j+1}^\transpose\text{\,--------------------}}
    \end{bmatrix} .
\end{align}
This matrix is invertible as its main diagonal is filled with $1$.
Also, by how $\vv_k$ are constructed, $\lambda_{k,j}^\star = 0$ is equivalent to 
\begin{align*}
    \left(\vL_j \right)_{N-k+1,:} \vS_{N-j} \begin{bmatrix} Q(1,j) \\ \vdots \\ Q(N-j,j) \end{bmatrix} = 0 .
\end{align*}
Therefore, if we define $\widetilde{\vL}_j$ as the matrix obtained by replacing the row $N-k(j)+1$ (containing $\vv_{k(j)}^\transpose$) by $\mathbf{0}^\transpose$, then the condition $\lambda_{k,j}^\star = 0$ for all $k\ne k(j)$ is equivalent to $\widetilde{\vL}_j \vS_{N-j} \begin{bmatrix} Q(1,j) & Q(2,j) & \cdots & Q(N-j,j) \end{bmatrix}^\transpose = \mathbf{0}$.
Because $\vS_{N-j}$ is nonsingular and $\widetilde{\vL}_j$ has rank $N-j-1$, the null space of $\widetilde{\vL}_j \vS_{N-j}$ is of dimension $1$.
In fact, it is spanned by the vector $\vu_j = \left(\vL_j \vS_{N-j} \right)^{-1} \ve_{N-k(j)+1}$ where $\ve_{i}$ is the $i$-th canonical basis vector, because $\vu_j$ is nonzero and $\vL_j \vS_{N-j} \vu_j = \ve_{N-k(j)+1}$, which implies that 
\begin{align*}
    \left(\vL_j \right)_{N-k+1,:} \vS_{N-j} \vu_j = 0 , \,\, \forall k \in \{j+1,\dots,N\}, \,\, k \ne k(j) \, \iff \, \widetilde{\vL}_j \vS_{N-j} \vu_j = \mathbf{0} .
\end{align*}
This shows that $\begin{bmatrix} Q(1,j) & \cdots & Q(N-j,j) \end{bmatrix}^\transpose \in \spann \{\vu_j\}$.
For the induction to go through, we must show that $\vu_j$ constructed as above has positive entries, and so does its normalization $\vv_j$ with last entry $1$.

Consider $\vT_{N-i} = \vS_{N-i}^{-1}$ and $\vM_i = (\vL_i \vS_{N-i})^{-1} = \vT_{N-i} \vL_i^{-1}$ for each $i=j,\dots,N-2$.
Because each $\vL_i$ is lower triangular and has main diagonal filled with $1$, we can write
\begin{align*}
    \vL_{i} = \begin{bmatrix}
        \vL_{i+1} & \mathbf{0} \\
        \vl_i^\transpose & 1
    \end{bmatrix}
\end{align*}
where $\mathbf{l}_i^\transpose = \begin{bmatrix} 0 & (\vv_{i+1})_{1:N-i-2}^\transpose \end{bmatrix}$.
Also, by \cref{lemma:pascal-matrix-properties}(ii), the bottom rightmost entry of $\vT_{N-i}$ is 1, so we can write
\begin{align*}
    \vT_{N-i} = \begin{bmatrix}
        \widetilde{\vT}_{N-i-1} & \vw_{i} \\
        \vt_{i}^\transpose & 1
    \end{bmatrix} .
\end{align*}
Taking the Schur complement, we obtain
\begin{align*}
    \begin{bmatrix}
        \vS_{N-i-1} & * \\
        * & *
    \end{bmatrix} 
    = \vS_{N-i} = \vT_{N-i}^{-1} = 
    \begin{bmatrix}
        \left( \widetilde{\vT}_{N-i-1} - \vw_{i} \vt_{i}^\transpose \right)^{-1} & * \\
        * & *
    \end{bmatrix} 
\end{align*}
which implies that $\widetilde{\vT}_{N-i-1} - \vw_{i} \vt_{i}^\transpose = \vS_{N-i-1}^{-1} = \vT_{N-i-1} \implies \widetilde{\vT}_{N-i-1} = \vT_{N-i-1} + \vw_{i} \vt_{i}^\transpose$.
Hence
\begin{align*}
    \vM_i = \vT_{N-i} \vL_i^{-1} & = 
    \begin{bmatrix}
        \vT_{N-i-1} + \vw_{i} \vt_{i}^\transpose & \vw_{i} \\
        \vt_{i}^\transpose & 1
    \end{bmatrix}
    \begin{bmatrix}
        \vL_{i+1}^{-1} & \mathbf{0} \\
        -\vl_i^\transpose \vL_{i+1}^{-1} & 1
    \end{bmatrix} \\
    & = \begin{bmatrix}
        \vT_{N-i-1} \vL_{i+1}^{-1} + \vw_{i} \left(\vt_{i}^\transpose - \vl_i^\transpose \right) \vL_{i+1}^{-1} & \vw_{i} \\
        \left(\vt_{i}^\transpose - \vl_i^\transpose \right) \vL_{i+1}^{-1} & 1
    \end{bmatrix} .
\end{align*}
Letting $\vm_i^\transpose = \left(\vt_{i}^\transpose - \vl_i^\transpose \right) \vL_{i+1}^{-1}$ and using $\vM_{i+1} = \vT_{N-i-1} \vL_{i+1}^{-1}$, we can more simply rewrite this as
\begin{align*}
    \vM_i = \begin{bmatrix}
        \vM_{i+1} + \vw_i \vm_i^\transpose & \vw_i \\
        \vm_i^\transpose & 1
    \end{bmatrix} .
\end{align*}
Here, in fact, $\vw_i$ is explicitly characterized by \cref{lemma:pascal-matrix-properties}(ii):
\begin{align}
\label{eqn:vec-wi-expression}
    \vw_i = \begin{bmatrix}
        \binom{N-i-1}{0} &
        \binom{N-i-1}{1} &
        \cdots &
        \binom{N-i-1}{N-i-2}
    \end{bmatrix}^\transpose .
\end{align}
Now observe that the lower left $(N-i-1) \times (N-i-1)$ submatrix of $\vS_{N-i}$ is the symmetric matrix
\begin{align*}
    -\vP_{N-i-1} = \left(\vS_{N-i}\right)_{2:N-i, 1:N-i-1} = 
    -\begin{bmatrix}
         \binom{2}{1} & -\binom{3}{1} & \cdots & (-1)^{N-i} \binom{N-i}{1} \\
        -\binom{3}{2} & \binom{4}{2} & \cdots & (-1)^{N-i+1} \binom{N-i+1}{2} \\ 
        \vdots & \vdots & \ddots & \vdots \\
        (-1)^{N-i} \binom{N-i}{N-i-1} & (-1)^{N-i+1} \binom{N-i+1}{N-i-1} & \cdots & \binom{2N-2i-2}{N-i-1} 
    \end{bmatrix} .
\end{align*}
Next, the upper left block of $\vL_i \vS_{N-i}$ is $\vL_{i+1} \vS_{N-i-1}$ because
\begin{align*}
    \vL_{i} \vS_{N-i} =
    \begin{bmatrix}
        \vL_{i+1} & \mathbf{0} \\
        \mathbf{l}_i^\transpose & 1
    \end{bmatrix}
    \begin{bmatrix}
        \vS_{N-i-1} & * \\
        * & *
    \end{bmatrix} =
    \begin{bmatrix}
        \vL_{i+1} \vS_{N-i-1} & * \\
        * & *
    \end{bmatrix}
\end{align*}
while the last row of $\vL_i \vS_{N-i}$ (excluding the last column) is $-\vv_{i+1}^\transpose \vP_{N-i-1}$ since
\begin{align*}
    \vL_i \vS_{N-i} =
    \begin{bmatrix}
        * & * \\
        0 & \vv_{i+1}^\transpose
    \end{bmatrix}
    \begin{bmatrix}
        * & * \\
        -\vP_{N-i-1} & *
    \end{bmatrix} =
    \begin{bmatrix}
        * & * \\
        -\vv_{i+1}^\transpose \vP_{N-i-1} & *
    \end{bmatrix} .
\end{align*}
Now investigating the last row of the matrix identity $\vI_{N-i} = \vM_i \vL_i \vS_{N-i}$ we obtain
\begin{align*}
    & \vI_{N-i} = 
    \vM_i \vL_i \vS_{N-i} =
     \begin{bmatrix}
        * & * \\
        \vm_i^\transpose & 1
    \end{bmatrix}
    \begin{bmatrix}
        \vL_{i+1} \vS_{N-i-1} & * \\
        -\vv_{i+1}^\transpose \vP_{N-i-1} & *
    \end{bmatrix} =
    \begin{bmatrix}
        * & * \\
        \vm_i^\transpose \vL_{i+1} \vS_{N-i-1} - \vv_{i+1}^\transpose \vP_{N-i-1} & * 
    \end{bmatrix} \\
    & \implies \mathbf{0}^\transpose = \vm_i^\transpose \vL_{i+1} \vS_{N-i-1} - \vv_{i+1}^\transpose \vP_{N-i-1} \\
    & \implies \vm_i^\transpose = \vv_{i+1}^\transpose \vP_{N-i-1} \vM_{i+1}
\end{align*}
where the last implication is obtained by multiplying $\vM_{i+1}$ to the right.
Because $\vv_{i+1}$ is obtained by dividing $\vu_{i+1} = \vM_{i+1} \ve_{N-k(i+1)+1}$ by its last entry, we have
\begin{align}
    \vm_i^\transpose = \vv_{i+1}^\transpose \vP_{N-i-1} \vM_{i+1} & = \frac{1}{(\vM_{i+1})_{N-i-1,N-k(i+1)+1}} \left( \vM_{i+1} \ve_{N-k(i+1)+1} \right)^\transpose \vP_{N-i-1} \vM_{i+1} \nonumber \\
    & = \frac{1}{(\vM_{i+1})_{N-i-1,N-k(i+1)+1}} \ve_{N-k(i+1)+1}^\transpose \vM_{i+1}^\transpose \vP_{N-i-1} \vM_{i+1} . \label{eqn:vec-mi-expression}
\end{align}
By \cref{lemma:pascal-matrix-properties}(iii), we have
\begin{align*}
    \vP_{N-i-1} = \vR_{N-i-1}^\transpose \vR_{N-i-1}
\end{align*}
where $\vR_{N-i-1}$ is the $(N-i) \times (N-i-1)$ matrix
\begin{align*}
    \vR_{N-i-1} = 
    \begin{bmatrix}
        \binom{1}{0} & -\binom{2}{0} & \cdots & (-1)^{N-i} \binom{N-i-1}{0} \\
        \binom{1}{1} & -\binom{2}{1} & \cdots & (-1)^{N-i} \binom{N-i-1}{1} \\
                     & -\binom{2}{2} & \cdots & (-1)^{N-i} \binom{N-i-1}{2} \\
                     &               & \ddots & \vdots \\
                     &               &        & (-1)^{N-i} \binom{N-i-1}{N-i-1}
    \end{bmatrix} .
\end{align*}
Letting $\vR_{N-i-1} \vM_{i+1} = \vW_{i+1} \in \reals^{(N-i)\times (N-i-1)}$, and noting that the last row of $\vM_{i+1}$ is $\widetilde \vm_{i+1}^\transpose := \begin{bmatrix} \vm_{i+1}^\transpose & 1 \end{bmatrix}$, we can rewrite \eqref{eqn:vec-mi-expression} as
\begin{align*}
    \vm_i^\transpose = \frac{1}{(\vM_{i+1})_{N-i-1,N-k(i+1)+1}} \ve_{N-k(i+1)+1}^\transpose (\vR_{N-i-1}\vM_{i+1})^\transpose \vR_{N-i-1}\vM_{i+1} = \frac{1}{(\widetilde\vm_{i+1})_{N-k(i+1)+1}} \ve_{N-k(i+1)+1}^\transpose \vW_{i+1}^\transpose \vW_{i+1} .
\end{align*}
Now we prove by backward induction on $j=N-2,\dots,1,0$ that \textbf{(i)} $\vM_j$ has positive entries, \textbf{(ii)} $\vW_j$ is \textit{row-sign consistent}, i.e., for each fixed row of $\vW_j$, its entries are either all nonnegative or all nonpositive, and \textbf{(iii)} the last row entries of $\vW_j$ are all nonzero.
For $j=1,\dots,N-2$, this will show that $\vu_j = \vM_j \ve_{N-k(j)+1}$ has positive entries, as desired.

For the base case $j=N-2$, we have $\vv_{N-1} = \begin{bmatrix} 1 \end{bmatrix}$, and thus,
\begin{align*}
    \vL_{N-2} \vS_2 = 
    \begin{bmatrix} 1 & 0 \\ 0 & 1 \end{bmatrix}
    \begin{bmatrix} 1 & -1 \\ -2 & 3 \end{bmatrix}
    \implies \vM_{N-2} = \begin{bmatrix} 1 & -1 \\ -2 & 3 \end{bmatrix}^{-1}
    = \begin{bmatrix} 3 & 1 \\ 2 & 1 \end{bmatrix} 
\end{align*}
clearly has positive entries.
Additionally,
\begin{align*}
    \vW_{N-2} = \vR_2 \vM_{N-2} = 
    \begin{bmatrix} 1 & -1 \\ 1 & -2 \\ 0 & -1 \end{bmatrix}
    \begin{bmatrix} 3 & 1 \\ 2 & 1 \end{bmatrix} 
    = \begin{bmatrix}
        1 & 0 \\ -1 & -1 \\ -2 & -1
    \end{bmatrix}
\end{align*}
which shows that $\vW_{N-2}$ is row-sign consistent and has nonzero entries in the last row.

Now for the induction step, suppose that the three conditions hold for $\vM_{j+1}$ and $\vW_{j+1}$. 
Then
\begin{align}
\label{eqn:vec-mj-identity}
    \vm_j^\transpose = \frac{1}{(\widetilde\vm_{j+1})_{N-k(j+1)+1}} \ve_{N-k(j+1)+1}^\transpose \vW_{j+1}^\transpose \vW_{j+1}
\end{align} 
has positive entries because $(\widetilde\vm_{j+1})_{N-k(j+1)+1} > 0$ and for $a,b=1,\dots,N-j-1$,
\[
    \left( \vW_{j+1}^\transpose \vW_{j+1} \right)_{a,b} = \inprod{\left(\vW_{j+1}\right)_{:,a}}{\left(\vW_{j+1}\right)_{:,b}} \ge \left(\vW_{j+1}\right)_{N-j,a} \left(\vW_{j+1}\right)_{N-j,b} > 0
\]
because $\vW_{j+1}$ is row-sign consistent and its last row entries $\vW_{j+1}$ are nonzero. 
Together with the induction hypothesis that $\vM_{j+1}$ has positive entries and \eqref{eqn:vec-wi-expression}, this implies 
$\vM_j = \begin{bmatrix}
    \vM_{j+1} + \vw_j \vm_j^\transpose & \vw_j \\
    \vm_j^\transpose & 1
\end{bmatrix}$
has positive entries as well.
Next,
\begin{align*}
    \vW_j = \vR_{N-j} \vM_j & = 
    \begin{bmatrix}
        \vR_{N-j-1} & (-1)^{N-j+1} \vw_{j-1} \\
        \mathbf{0}^\transpose & (-1)^{N-j+1}
    \end{bmatrix}   
    \begin{bmatrix}
        \vM_{j+1} + \vw_j \vm_j^\transpose & \vw_j \\
        \vm_j^\transpose & 1
    \end{bmatrix}
    \\
    & = \begin{bmatrix}
        \vR_{N-j-1} \vM_{j+1} + \vR_{N-j-1} \vw_j \vm_j^\transpose + (-1)^{N-j+1} \vw_{j-1} \vm_j^\transpose & \vR_{N-j-1} \vw_j + (-1)^{N-j+1} \vw_{j-1} \\
        (-1)^{N-j+1} \vm_j^\transpose & (-1)^{N-j+1} 
    \end{bmatrix}
\end{align*}
and this shows that all last row entries of $\vW_j$ are nonzero.
Additionally, by \cref{lemma:pascal-matrix-properties}(iv) we have
\begin{align*}
    \vR_{N-j-1} \vw_j = (-1)^{N-j} (\vw_{j-1} - \ve_{N-j}) .
\end{align*}
Using this together with the identity $\vR_{N-j-1} \vM_{j+1} = \vW_{j+1}$, we obtain
\begin{align*}
    \vW_j = \begin{bmatrix}
        \vW_{j+1} + (-1)^{N-j+1} \ve_{N-j} \vm_j^\transpose & (-1)^{N-j+1} \ve_{N-j} \\
        (-1)^{N-j+1} \vm_j^\transpose & (-1)^{N-j+1}
    \end{bmatrix} .
\end{align*}
The last row of $\vW_j$ is $(-1)^{N-j+1} \begin{bmatrix} \vm_j^\transpose & 1 \end{bmatrix}$, which is sign-consistent.
Also note that
\begin{align*}
    \left( \vW_j \right)_{1:N-j-1, 1:N-j} =
    \begin{bmatrix}
        \left( \vW_{j+1} \right)_{1:N-j-1, 1:N-j-1} & \mathbf{0}    
    \end{bmatrix}
\end{align*}
which implies that the first $N-j-1$ rows of $\vW_j$ are sign-consistent.
It only remains to check sign consistency of the $(N-j)$-th row of $\vW_j$.
We have
\begin{align}
    \left( \vW_j \right)_{N-j, 1:N-j} & = 
    \begin{bmatrix} 
        \left( \vW_{j+1} \right)_{N-j, 1:N-j-1} + (-1)^{N-j+1} \vm_j^\transpose & (-1)^{N-j+1}
    \end{bmatrix} \nonumber \\
    & = 
    \begin{bmatrix} 
        (-1)^{N-j} \vm_{j+1}^\transpose + (-1)^{N-j+1} \left(\vm_j\right)_{1:N-j-2}^\transpose & (-1)^{N-j} + (-1)^{N-j+1} (\vm_j)_{N-j-1} & (-1)^{N-j+1}
    \end{bmatrix} \nonumber \\
    & = 
    (-1)^{N-j+1} \begin{bmatrix} 
        \left(\vm_j\right)_{1:N-j-2}^\transpose - \vm_{j+1}^\transpose & (\vm_j)_{N-j-1} - 1 & 1 
    \end{bmatrix} \label{eqn:Wj-row-N-j-expression}
\end{align}
where the second identity holds because the last row of $\vW_{j+1}$ is $(-1)^{N-j} \begin{bmatrix} \vm_{j+1}^\transpose & 1 \end{bmatrix}$. 
Now for each $\ell=1,\dots,N-j-2$, using \eqref{eqn:vec-mj-identity} we have 
\begin{align*}
    (\vm_j)_\ell & = 
    \frac{1}{(\widetilde\vm_{j+1})_{N-k(j+1)+1}} \ve_{N-k(j+1)+1}^\transpose \vW_{j+1}^\transpose \vW_{j+1} \ve_\ell \\
    & = \frac{1}{(\widetilde\vm_{j+1})_{N-k(j+1)+1}} \inprod{\left( \vW_{j+1} \right)_{1:N-j,\ell}}{\left( \vW_{j+1} \right)_{1:N-j, N-k(j+1)+1}} \\
    & \ge \frac{1}{(\widetilde\vm_{j+1})_{N-k(j+1)+1}} \left( \vW_{j+1} \right)_{N-j, \ell} \left( \vW_{j+1} \right)_{N-j, N-k(j+1)+1} \\
    & = \frac{1}{(\widetilde\vm_{j+1})_{N-k(j+1)+1}} (-1)^{N-j} (\widetilde\vm_{j+1})_\ell \cdot (-1)^{N-j} (\widetilde\vm_{j+1})_{N-k(j+1)+1} \\
    & = (\vm_{j+1})_\ell 
\end{align*}
and similarly, 
\begin{align*}
    (\vm_j)_{N-j-1} & \ge \frac{1}{(\widetilde\vm_{j+1})_{N-k(j+1)+1}} \left( \vW_{j+1} \right)_{N-j, N-j-1} \left( \vW_{j+1} \right)_{N-j, N-k(j+1)+1} \\
    & = \frac{1}{(\widetilde\vm_{j+1})_{N-k(j+1)+1}} (-1)^{N-j} \cdot (-1)^{N-j} (\widetilde\vm_{j+1})_{N-k(j+1)+1} = 1 .
\end{align*}
Hence, in \eqref{eqn:Wj-row-N-j-expression}, we see that the $(N-j)$-th row has the consistent sign $(-1)^{N-j+1}$, completing the induction.

The previous arguments show that $\vv_{N-2}, \dots, \vv_1$ can be determined recursively.
Now, using this information, we determine the values of $Q(k,j)$ satisfying \eqref{eqn:H-invariance} and \eqref{eqn:Q-vector-belongs-to-vk-span}. 
If we denote $Q(N-j,j) = q_j$, then
\begin{align*}
    \begin{bmatrix}
        1 & q_{N-1} & q_{N-2} & \cdots & q_1
    \end{bmatrix}
    \vL_0 
    & =
    \begin{bmatrix}
        1 & Q(1,N-1) & Q(2,N-2) & \cdots & Q(N-1,1)
    \end{bmatrix}
    \begin{bmatrix}
        1 & 0 & 0 & \multicolumn{2}{c}{\cdots} & 0 \\
        0 & 1 & 0 & \multicolumn{2}{c}{\cdots} & 0 \\
        0 & \multicolumn{2}{c}{\vv_{N-2}^\transpose} & \multicolumn{2}{c}{\cdots} & 0 \\
        \vdots & \multicolumn{2}{c}{\vdots} & \multicolumn{2}{c}{\ddots} & \vdots \\
        0 & \multicolumn{5}{c}{\text{--------\,}\vv_{1}^\transpose\text{\,--------}}
    \end{bmatrix} \\
    & = \begin{bmatrix}
        1 & \sum_{j=1}^{N-1} Q(j,N-j) \left(\vv_{N-j} \right)_1 & \sum_{j=2}^{N-1} Q(j,N-j) \left(\vv_{N-j} \right)_2 & \cdots & Q(N-1,1)
    \end{bmatrix} \\
    & = \begin{bmatrix}
        1 & \sum_{j=1}^{N-1} Q(1,N-j) & \sum_{j=2}^{N-1} Q(2,N-j) & \cdots & Q(N-1,1)
    \end{bmatrix} \\
    & = \begin{bmatrix}
        \frac{1}{N} \binom{N}{1} &
        \frac{1}{N} \binom{N}{2} &
        \frac{1}{N} \binom{N}{3} &
        \cdots &
        \frac{1}{N} \binom{N}{N}
    \end{bmatrix}
\end{align*}
where the third equality uses \eqref{eqn:Q-vector-belongs-to-vk-span} and the last equality uses \eqref{eqn:H-invariance}.
This shows that $q_{N-1}, \dots, q_1$ are uniquely determined as
\begin{align*}
    \begin{bmatrix}
        1 & q_{N-1} & q_{N-2} & \cdots & q_1
    \end{bmatrix} 
    & = \begin{bmatrix}
        \frac{1}{N} \binom{N}{1} &
        \frac{1}{N} \binom{N}{2} &
        \frac{1}{N} \binom{N}{3} &
        \cdots &
        \frac{1}{N} \binom{N}{N}
    \end{bmatrix} \vL_0^{-1} \\
    & = \begin{bmatrix}
        \frac{1}{N} \binom{N}{1} &
        \frac{1}{N} \binom{N}{2} &
        \frac{1}{N} \binom{N}{3} &
        \cdots &
        \frac{1}{N} \binom{N}{N}
    \end{bmatrix} \vS_N \vM_0 \\
    & = \frac{1}{N} \begin{bmatrix}
        0 & 0 & \cdots & 0 & 1 
    \end{bmatrix} \vM_0 \\
    & = \frac{1}{N} \begin{bmatrix} \vm_0^\transpose & 1 \end{bmatrix} ,
\end{align*}
where the third identity uses \cref{lemma:pascal-matrix-properties}(ii).
In particular, $q_1,\dots,q_{N-1}$ are all positive numbers.

Finally, the set of Q-functions with positive values of $Q(N-j,j)$ for $j=1,\dots,N-1$ uniquely corresponds to an H-matrix $H$ via \cref{meta-alg:find-H-from-Q}.
The $H$ associated with $Q(\cdot,\cdot)$ determined as above is indeed optimal: by construction it satisfies the H-invariance conditions~\eqref{eqn:H-invariance} and $\lambda_{k,j}^\star(H) = 0$ for all $k\ne k(j)$, $j=1,\dots,N-1$.
For $\lambda_{k(j),j}^\star (H)$, we verify that they are positive: if $k(j) = N$ then we have
\begin{align*}
    \lambda_{k(j),j}^\star = \lambda_{N,j}^\star = N \sum_{\ell=1}^{N-j} (-1)^{\ell-1} Q(\ell,j) = N \ve_1^\transpose \vL_j \vS_{N-j} 
    \begin{bmatrix}
        Q(1,j) \\
        \vdots \\
        Q(N-j,j)
    \end{bmatrix}
    = N q_j \ve_1^\transpose \vL_j \vS_{N-j} \vv_j = \frac{Nq_j}{(\vM_j)_{N-j,1}} > 0
\end{align*}
where we use the convention $\vL_{N-1} = \vM_{N-1} = \begin{bmatrix} 1 \end{bmatrix} = \vS_1$ when $j=N-1$, and the last equality holds because $\vv_j = \frac{1}{(\vM_j)_{N-j,1}} \left( \vL_j \vS_{N-j} \right)^{-1} \ve_1$ by construction.
Similarly, if $k(j) < N$ we have
\begin{align*}
    \lambda_{k(j),j}^\star = N q_{k(j)} q_j \ve_{N-k(j)+1}^\transpose \vL_j \vS_{N-j} \vv_j = \frac{Nq_{k(j)} q_j}{(\vM_j)_{N-j,N-k(j)+1}} > 0
\end{align*}
where the last identity uses $\vv_j = \frac{1}{(\vM_j)_{N-j,N-k(j)+1}} \left( \vL_j \vS_{N-j} \right)^{-1} \ve_{N-k(j)+1}$.
\hfill
\qedsymbol
\bigskip

The above proof is constructive, and thus yields an implementable numerical procedure, Meta Algorithm~\ref{meta-alg:find-Q-from-sparsity-pattern}, for computing the Q-profile associated with the sparsity pattern~\eqref{eqn:lambda-sparsity-pattern} of H-certificates.
One can then use Meta Algorithm~\ref{meta-alg:find-H-from-Q} to recover the corresponding H-matrix $H \in \cH_\star(N-1)$.

\begin{algorithm}[t]
\caption{Meta algorithm for computing $Q$ from sparsity pattern}
\label{meta-alg:find-Q-from-sparsity-pattern} 
\textbf{Input: } $k(1),k(2),\dots,k(N-1)$
\begin{algorithmic}
\State $\vv_{N-1} = \begin{bmatrix} 1 \end{bmatrix}$
\For{$j=N-2,N-3,\dots,1$}
    \State Form $\vL_j$ as in \eqref{eqn:matrix-Lj}
    \State $\vu_j = \left(\vL_j \vS_{N-j}\right)^{-1} \ve_{N-k(j)+1}$
    \State $\vv_j = \vu_j / \left(\vu_j\right)_{N-j}$
\EndFor
\State Form $\vL_0$ from $\vv_1,\dots,\vv_{N-1}$ as in \eqref{eqn:matrix-Lj}
\State $\vq = \frac{1}{N}\begin{bmatrix} \binom{N}{1} & \binom{N}{2} & \cdots & \binom{N}{N} \end{bmatrix} \vL_0^{-1}$
\For{$j=1,\dots,N-1$}
    \State $q_j = (\vq)_{N-j+1}$
    \State $\begin{bmatrix} Q(1,j) & \cdots & Q(N-j,j) \end{bmatrix}^{\transpose} = q_j \vv_j$
\EndFor
\end{algorithmic}
\textbf{Output: } $\{Q(k,j)\}_{\substack{j=1,\dots,N-1\\k=1,\dots,N-j}}$
\end{algorithm}

\section{Proof of \texorpdfstring{\cref{theorem:gluing}}{Theorem 2}}
\label{section:proof-of-gluing-theorem}

\begin{lemma}
\label{lemma:Q-lambda-identities}
Let $H = \left( h_{k,j} \right) \in \reals^{(N-1)\times (N-1)}$ be a lower triangular matrix.
Then
\begin{enumerate}[label=\normalfont(\alph*)]
    \item With the convention $Q(a,b;H) = 0$ for $a+b > N$,
    \begin{align*}
        Q(m+1,j;H) = \sum_{k=j+1}^{N-1} Q(m,k;H) \sum_{\ell=j}^{k-1} h_{\ell,j}   
    \end{align*}
    for $m,j=1,\dots,N-1$.
    \item If $H$ satisfies H-invariance~\eqref{eqn:H-invariance}, then $\sum_{j=1}^{N-1} \lambda_{N,j}^\star(H) = N-1$.
    \item If $H$ satisfies H-invariance~\eqref{eqn:H-invariance}, then for $k=1,\dots,N-1$, 
    \begin{align*}
        \sum_{j=1}^{k-1} \lambda_{k,j}^\star(H) + \sum_{j=k+1}^N \lambda_{j,k}^\star(H) = N \sum_{\ell=1}^{N-1} \sum_{m=1}^{N-1} (-1)^{\ell+m} \binom{\ell+m}{m} Q(m,k; H) Q(\ell,k; H) . 
    \end{align*}
\end{enumerate}
\end{lemma}

\begin{proof}
(a) This is a restatement of \citep[Lemma~2]{YoonRyuGrimmerInvariance2025}.

\smallskip

\noindent
(b) By \eqref{eqn:H-certificate-N}, we have
\begin{align*}
    \sum_{j=1}^{N-1}\lambda_{N,j}^\star(H)
    & = N\sum_{j=1}^{N-1} \sum_{m=1}^{N-1} (-1)^{m-1} Q(m,j;H) \\
    & = N\sum_{m=1}^{N-1} (-1)^{m-1} \sum_{j=1}^{N-1} Q(m,j;H) \\
    & = \sum_{m=1}^{N-1} (-1)^{m-1}\binom{N}{m+1} \\
    & = N-1
\end{align*}
where the third equality follows from the H-invariance condition~\eqref{eqn:H-invariance} and the last equality from $0 = \sum_{m=-1}^{N-1} (-1)^{m-1} \binom{N}{m+1} = 1 - N + \sum_{m=1}^{N-1} (-1)^{m-1} \binom{N}{m+1}$.

\smallskip

\noindent
(c) This is shown in the proof of \citep[Theorem~5]{YoonRyuGrimmerInvariance2025}, Part~3.

\end{proof}

\begin{lemma}
\label{lemma:gluing-helper-identities}
Let $H=H_1\glue H_2$ be as in \cref{definition:gluing-of-H-matrices}. Then the following properties hold.
\begin{enumerate}[label=\normalfont(\alph*)]
    \item For $m=1,\dots,N-N'$,
    \[
        Q(m,N';H)=\frac{N'(N-N'+1)}{N(m+1)}\binom{N-N'-1}{m-1},
    \]
    and for $j=N'+1,\dots,N-1$,
    \[
        Q(m,j;H)=Q(m,j-N';H_2).
    \]
    
    \item For each $m=0,1,\dots,N-N'-1$,
    \[
        \sum_{n=1}^{N-N'} (-1)^{n} \binom{m+n-1}{m} Q(n,N'; H)
        =
        \begin{cases}
            -\frac{N'}{N(N-N')} & \text{if } m=0 \\
            \frac{N'}{N(N-N')} & \text{if } m=1 \\
            0 & \text{if } 2\le m\le N-N'-1 .
        \end{cases}
    \]
\end{enumerate}
\end{lemma}

\begin{proof}
(a) The formula for $Q(m,N';H)$ is built into \cref{definition:gluing-of-H-matrices} of $H_1 \glue H_2$. For $j>N'$, the Q-functions $Q(m,j;H)$ depend only on the lower-right block of $H$, which is exactly $H_2$ with column indices shifted by $N'$.

\smallskip

\noindent
(b) Let $M = N-N'$. By part~(a), for $n=1,\dots,N-N'$, we have $Q(n,N'; H) = \frac{N'(M+1)}{N(n+1)} \binom{M-1}{n-1}$.
Then, by the elementary identity
\[
    \frac{1}{n+1}\binom{M-1}{n-1}
    = \frac{1}{M}\binom{M}{n} - \frac{1}{M(M+1)}\binom{M+1}{n+1} ,
\]
we have
\begin{align}
    & \sum_{n=1}^{N-N'} (-1)^{n} \binom{m+n-1}{m} Q(n,N'; H) \nonumber \\
    & =
    \frac{N'(M+1)}{N} \sum_{n=1}^{M} (-1)^{n} \binom{m+n-1}{m}\frac{1}{n+1} \binom{M-1}{n-1} \nonumber \\
    & = \frac{N'(M+1)}{N} \left[ 
    \frac{1}{M}\sum_{n=1}^{M} (-1)^n \binom{m+n-1}{m}\binom{M}{n}
    - \frac{1}{M(M+1)}\sum_{n=1}^{M} (-1)^n \binom{m+n-1}{m}\binom{M+1}{n+1}
    \right] . \label{eqn:glued-matrix-Q-weighted-sum-decomposed}
\end{align}
Now note that $\binom{-m-1}{n-1} = (-1)^{n-1}\binom{m+n-1}{m}$. Using this together with \cref{lemma:chu-vandermonde}, we obtain
\begin{gather*}
    \sum_{n=1}^{M} (-1)^n \binom{m+n-1}{m}\binom{M}{n}
    = -\sum_{r=0}^{M-1}\binom{-m-1}{r}\binom{M}{M-1-r} = -\binom{M-m-1}{M-1} 
    \\
    \sum_{n=1}^{M} (-1)^n \binom{m+n-1}{m}\binom{M+1}{n+1}
    = -\sum_{r=0}^{M-1}\binom{-m-1}{r}\binom{M+1}{M-1-r} = -\binom{M-m}{M-1}.
\end{gather*}
For $m=0$, plugging this into \eqref{eqn:glued-matrix-Q-weighted-sum-decomposed} gives
\begin{align*}
    \sum_{n=1}^{N-N'} (-1)^n Q(n,N';H) = \frac{N'(M+1)}{N} \left(-\frac{1}{M} + \frac{1}{M+1}\right) = -\frac{N'}{NM} = -\frac{N'}{N(N-N')} 
\end{align*}
while for $m\ge 1$, we have
\begin{align*}
    \sum_{n=1}^{N-N'} (-1)^{n} \binom{m+n-1}{m} Q(n,N'; H) = \frac{N'(M+1)}{N} \cdot \frac{1}{M(M+1)} \binom{M-m}{M-1} = \frac{N'}{NM} \binom{M-m}{M-1} ,
\end{align*}
which is nonzero only for $m=1$, which gives $\frac{N'}{NM} = \frac{N'}{N(N-N')}$.

\end{proof}

We are now ready to prove \cref{theorem:gluing}.
Let $H_1 \in \cH_\star(N'-1)$ and $H_2 \in \cH_\star(N-N'-1)$.
For $H = H_1 \glue H_2$ as defined in \cref{definition:gluing-of-H-matrices}, we will directly verify the identity~\eqref{eqn:gluing-proof-template}, restated here for convenience:
\begin{align*}
    0 & = N \sqnorm{g_N} + \inprod{g_N}{x_N - y_0} + \sum_{1\le j < k \le N'} \frac{N'}{N} \lambda_{k,j}^\star (H_1) \inprod{x_k - x_j}{g_k - g_j} + \frac{N'}{N-N'} \inprod{x_{N} - x_{N'}}{g_N - g_{N'}} \\
    & \quad + \sum_{N'+1 \le j < k \le N} \frac{N}{N-N'} \lambda_{k-N',j-N'}^\star (H_2) \inprod{x_k - x_j}{g_k - g_j} .
\end{align*}
Once this is proved, as it follows the universal proof template~\eqref{eqn:optimal-family-proof-core-identity} for optimal algorithms, we have $H\in \cH_\star(N-1)$. 
Furthermore, by uniqueness of H-certificates, 
\begin{itemize}
    \item $\lambda_{k,j}^\star (H) = \frac{N'}{N} \lambda_{k,j}^\star (H_1)$ if $1\le j < k \le N'$
    \item $\lambda_{k,j}^\star (H) = 0$ if $1\le j < N' < k$
    \item $\lambda_{N,N'}^\star(H) = \frac{N'}{N-N'}$ and $\lambda_{k,N'}^\star(H) = 0$ for $N' < k < N$
    \item $\lambda_{k,j}^\star (H) = \frac{N}{N-N'} \lambda_{k-N',j-N'}^\star (H_2)$ if $N'+1 \le j < k \le N$
\end{itemize}
so $\lambda^\star(H) = \lambda^\star(H_1) \glue \lambda^\star(H_2)$.
When $H_1, H_2$ are vertex algorithms, the above directly implies that $\cG(H) = \cG(H_1) \glue \cG(H_2)$ as arc diagrams, by definition.

Now, to prove~\eqref{eqn:gluing-proof-template}, first note that the update rule for the algorithm defined by $H$ is identical to that defined by $H_1$ up to the iterates $y_{N'-1}$ and $x_{N'} = y_{N'-1} - g_{N'}$.
Therefore, we have
\begin{align}
    & 0 = N' \sqnorm{g_{N'}} + \inprod{g_{N'}}{x_{N'} - y_0} + \sum_{1\le j < k \le N'} \lambda_{k,j}^\star (H_1) \inprod{x_k - x_j}{g_k - g_j} \nonumber \\
    & \iff \sum_{1\le j < k \le N'} \lambda_{k,j}^\star (H_1) \inprod{x_k - x_j}{g_k - g_j} = -N' \sqnorm{g_{N'}} - \inprod{g_{N'}}{x_{N'} - y_0} . \label{eqn:gluing-early-iterations-ineq-sum}
\end{align}
Next, note that
\begin{align*}
    & y_{N'-1} = y_0 + \sum_{k=1}^{N'-1} (y_k - y_{k-1}) = y_0 + \sum_{k=1}^{N'-1} \left( -\sum_{j=1}^k 2(H_1)_{k,j} g_j \right) \\
    & \iff y_0 - y_{N'-1} = \sum_{k=1}^{N'-1} \sum_{j=1}^k 2(H_1)_{k,j} g_j = \sum_{j=1}^{N'-1} \left( \sum_{k=j}^{N'-1} 2(H_1)_{k,j} \right) g_j = - \frac{N}{N-N'} \sum_{j=1}^{N'-1} 2h_{N',j} g_j .
\end{align*}
Therefore, 
\begin{align}
\label{eqn:gluing-yN'-expression}
    y_{N'} = y_{N'-1} - 2h_{N',N'} g_{N'} - \sum_{j=1}^{N'-1} 2h_{N',j} g_j = y_{N'-1} - 2h_{N',N'} g_{N'} + \frac{N-N'}{N} (y_0 - y_{N'-1}) .
\end{align}
Next, for $k=N'+1,\dots,N-1$, we have
\begin{align}
\label{eqn:gluing-true-y-sequence}
    y_{k} = y_{k-1} - 2h_{k,N'} g_{N'} - \sum_{j=N'+1}^k 2(H_2)_{k-N', j-N'} g_j ,
\end{align}
which agrees with the update rule from $H_2$ except for the additional $-2h_{k,N'} g_{N'}$ term.
To exploit the optimality proof structure for $H_2$, define $\widetilde{y}_0 = y_{N'}$, $\widetilde{g}_i = g_{N'+i}$ for $i=1,\dots,N-N'$ and let
\begin{align}
\label{eqn:gluing-fake-y-sequence}
    \widetilde{y}_i = \widetilde{y}_{i-1} - \sum_{m=1}^i 2(H_2)_{i,m} \widetilde{g}_m 
\end{align}
for $i=1,\dots,N-N'-1$, and $\widetilde{x}_{i} = \widetilde{y}_{i-1} - \widetilde{g}_i$ for $i=1,\dots,N-N'$.
Here, although $2\widetilde{g}_i$ are \textbf{not} the fixed-point residuals computed at $\widetilde{y}_{i-1}$, the sequences $\widetilde{y}_i$ and $\widetilde{x}_{i+1}$ \textit{formally} behave like the iterates and residual variables generated by $H_2$.
Therefore, we have the identity
\begin{align}
\label{eqn:gluing-fake-sequence-proof-template}
    0 = (N-N') \sqnorm{\widetilde{g}_{N-N'}} + \inprod{\widetilde{g}_{N-N'}}{\widetilde{x}_{N-N'} - \widetilde{y}_0} + \sum_{1\le j' < k' \le N-N'} \lambda_{k',j'}^\star (H_2) \inprod{\widetilde{x}_{k'} - \widetilde{x}_{j'}}{\widetilde{g}_{k'} - \widetilde{g}_{j'}} ,
\end{align}
as this holds regardless of whether $\widetilde{g}_i$ is associated with the evaluation of $\opT$ at $\widetilde{y}_{i-1}$ or not---in fact, $\widetilde{g}_i \in \cX$ can be arbitrary vectors, and as long as $\widetilde{y}_i$ are defined as \eqref{eqn:gluing-fake-y-sequence}, the identity \eqref{eqn:gluing-fake-sequence-proof-template} holds.
Now, because the actual update rule of $y_k$ is given as~\eqref{eqn:gluing-true-y-sequence}, the difference between $y_k$ and $\widetilde{y}_{k-N'}$ accumulates by the amount $-2h_{k,N'} g_{N'}$ at each iteration $k$, which implies
\begin{align*}
    \widetilde{y}_i = y_{N'+i} + \sum_{k=N'+1}^{N'+i} 2h_{k,N'} g_{N'}
\end{align*}
for $i=1,\dots,N-N'-1$.
Likewise, because $\widetilde{x}_i = \widetilde{y}_{i-1} - \widetilde{g}_i$ and $x_{N'+i} = y_{N'+i-1} - g_{N'+i} = y_{N'+i-1} - \widetilde{g}_i$, we have
\begin{align*}
    \widetilde{x}_i = x_{N'+i} + \sum_{k=N'+1}^{N'+i-1} 2h_{k,N'} g_{N'} .
\end{align*}
Plugging these into~\eqref{eqn:gluing-fake-sequence-proof-template}, and substituting $\widetilde{g}_i, \widetilde{y}_0$ with $g_{N'+i}, y_{N'}$, we obtain
\begin{align*}
    0 & = (N-N') \sqnorm{g_N} + \inprod{g_N}{x_N + \sum_{k=N'+1}^{N-1} 2h_{k,N'}g_{N'} - y_{N'}} \\
    & \quad + \sum_{1\le j' < k' \le N-N'} \lambda_{k',j'}^\star (H_2) \inprod{x_{N'+k'} - x_{N'+j'} + \sum_{\ell=N'+j'}^{N'+k'-1} 2h_{\ell,N'} g_{N'} }{g_{N' + k'} - g_{N' + j'} } \\
    & = (N-N') \sqnorm{g_N} + \inprod{g_N}{x_N - y_{N'} + \sum_{k=N'+1}^{N-1} 2h_{k,N'}g_{N'} } \\
    & \quad + \sum_{N'+1 \le j < k \le N} \lambda_{k-N',j-N'}^\star (H_2) \inprod{x_k - x_j + \sum_{\ell=j}^{k-1} 2h_{\ell,N'} g_{N'} }{g_k - g_j } 
\end{align*}
where for the last equality, we apply the changes of indices $k=N'+k'$ and $j=N'+j'$.
From this, we deduce
\begin{align*}
    & \sum_{N'+1\le j < k \le N} \lambda_{k-N',j-N'}^\star (H_2) \inprod{x_k - x_j}{g_k - g_j} \\
    & = -(N-N') \sqnorm{g_N} - \inprod{g_N}{x_N - y_{N'} + \sum_{k=N'+1}^{N-1} 2h_{k,N'}g_{N'} } - \sum_{N'+1 \le j < k \le N} \lambda_{k-N',j-N'}^\star (H_2) \inprod{\sum_{\ell=j}^{k-1} 2h_{\ell,N'} g_{N'} }{g_k - g_j } .
\end{align*}
Using this identity and \eqref{eqn:gluing-yN'-expression}, we have
\begin{align}
    & \sum_{1\le j < k \le N'} \frac{N'}{N} \lambda_{k,j}^\star (H_1) \inprod{x_k - x_j}{g_k - g_j} + \frac{N'}{N-N'} \inprod{x_{N} - x_{N'}}{g_N - g_{N'}} \nonumber \\
    & \quad + \sum_{N'+1 \le j < k \le N} \frac{N}{N-N'} \lambda_{k-N',j-N'}^\star (H_2) \inprod{x_k - x_j}{g_k - g_j} \nonumber \\
    & = -\frac{N'}{N} \left( N' \sqnorm{g_{N'}} + \inprod{g_{N'}}{x_{N'} - y_0} \right) + \frac{N'}{N-N'} \inprod{x_N - y_{N'-1} + g_{N'}}{g_{N} - g_{N'}} \nonumber \\
    & \quad - N\sqnorm{g_N} - \frac{N}{N-N'} \inprod{g_N}{x_N - y_{N'} + \sum_{k=N'+1}^{N-1} 2h_{k,N'}g_{N'} } \nonumber \\
    & \quad - \underbrace{\frac{N}{N-N'} \sum_{N'+1 \le j < k \le N} \lambda_{k-N',j-N'}^\star (H_2) \inprod{\sum_{\ell=j}^{k-1} 2h_{\ell,N'} g_{N'} }{g_k - g_j }}_{=(\mathrm{I})} \nonumber \\
    & \begin{aligned}
        & = -\frac{N'}{N} \left( (N'-1) \sqnorm{g_{N'}} + \inprod{g_{N'}}{y_{N'-1} - y_0} \right) + \frac{N'}{N-N'} \left( \inprod{x_N - y_{N'-1}}{g_N} + \inprod{g_{N'}}{g_N} - \inprod{x_N - y_{N'-1} + g_{N'}}{g_{N'}}\right) \\
        & \quad - N\sqnorm{g_N} - \frac{N}{N-N'} \inprod{g_N}{x_N - \left( y_{N'-1} - 2h_{N',N'} g_{N'} + \frac{N-N'}{N} (y_0 - y_{N'-1}) \right) + \sum_{k=N'+1}^{N-1} 2h_{k,N'} g_{N'}} - (\mathrm{I})    
    \end{aligned}
    \label{eqn:ineq-summation-simplified-first}
\end{align}
where the last equality uses \eqref{eqn:gluing-yN'-expression}.
Our final goal is to show \eqref{eqn:gluing-proof-template}, which is equivalent to showing that the expression \eqref{eqn:ineq-summation-simplified-first} simplifies to $-N\sqnorm{g_N} - \inprod{g_N}{x_N - y_0}$.
Note that collecting the coefficients of the $\inprod{x_{N} - y_{N'-1}}{g_N}$ terms gives $\frac{N'}{N-N'} - \frac{N}{N-N'} = -1$, and thus \eqref{eqn:ineq-summation-simplified-first} becomes
\begin{align}
    & -\frac{N'}{N} \left( (N'-1) \sqnorm{g_{N'}} + \inprod{g_{N'}}{y_{N'-1} - y_0} \right) + \frac{N'}{N-N'} \left( \inprod{g_{N'}}{g_N} - \inprod{x_N - y_{N'-1} + g_{N'}}{g_{N'}}\right) \nonumber \\
    & \quad - N\sqnorm{g_N} - \inprod{g_N}{x_N - y_{N'-1}} - \frac{N}{N-N'} \inprod{g_N}{ 2h_{N',N'} g_{N'} - \frac{N-N'}{N} (y_0 - y_{N'-1}) + \sum_{k=N'+1}^{N-1} 2h_{k,N'} g_{N'}} - (\mathrm{I}) \nonumber \\
    & = -\frac{N'}{N} \left( (N'-1) \sqnorm{g_{N'}} + \inprod{g_{N'}}{y_{N'-1} - y_0} \right) + \frac{N'}{N-N'} \left( \inprod{g_{N'}}{g_N} - \inprod{x_N - y_{N'-1} + g_{N'}}{g_{N'}}\right) \nonumber \\
    & \quad - N\sqnorm{g_N} - \inprod{g_N}{x_N - y_{N'-1}} + \inprod{g_N}{y_0 - y_{N'-1}} - \frac{N}{N-N'} \inprod{g_N}{ 2Q(1,N'; H) g_{N'}} - (\mathrm{I}) \nonumber \\
    & = -\frac{N'}{N} \left( (N'-1) \sqnorm{g_{N'}} + \inprod{g_{N'}}{y_{N'-1} - y_0} \right) + \frac{N'}{N-N'} \left( \inprod{g_{N'}}{g_N} - \inprod{x_N - y_{N'-1} + g_{N'}}{g_{N'}}\right) \nonumber \\
    & \quad - N\sqnorm{g_N} - \inprod{g_N}{x_N - y_0} - \frac{N}{N-N'} \inprod{g_N}{ \frac{N'(N-N'+1)}{N} g_{N'}} - (\mathrm{I}) \nonumber \\
    & \begin{aligned}
        & = - N\sqnorm{g_N} - \inprod{g_N}{x_N - y_0} - \frac{N'}{N} \left( (N'-1) \sqnorm{g_{N'}} + \inprod{g_{N'}}{y_{N'-1} - y_0} \right) - N' \inprod{g_{N'}}{g_N} \\
        & \quad - \underbrace{\frac{N'}{N-N'} \inprod{x_N - y_{N'-1} + g_{N'}}{g_{N'}}}_{(\mathrm{II})} - (\mathrm{I}) .    
    \end{aligned}
    \label{eqn:ineq-summation-simplified-second}
\end{align}
Here we use $\sum_{k=N'}^{N-1} h_{k,N'} = Q(1,N';H) = \frac{N'(N-N'+1)}{2N}$, which is exactly \cref{lemma:gluing-helper-identities}(a) with $m=1$, and simplify the coefficient of $\inprod{g_{N'}}{g_N}$ via the identity $\frac{N'}{N-N'} - \frac{N'(N-N'+1)}{N-N'} = -N'$.
Now, we can rewrite $(\mathrm{II})$ as follows:
\begin{align*}
    (\mathrm{II}) & = \frac{N'}{N-N'} \inprod{x_N - y_{N'} + (y_{N'} - y_{N'-1}) + g_{N'}}{g_{N'}} \\
    & = \frac{N'}{N-N'} \inprod{y_{N-1} - g_N - y_{N'} - 2h_{N',N'} g_{N'} + \frac{N-N'}{N}(y_0 - y_{N'-1}) + g_{N'} }{g_{N'} } \\
    & = \frac{N'}{N-N'} \inprod{y_{N-1} - y_{N'} - 2h_{N',N'} g_{N'} + g_{N'}}{g_{N'}} - \frac{N'}{N-N'} \inprod{g_N}{g_{N'}} + \frac{N'}{N} \inprod{g_{N'}}{y_0 - y_{N'-1}} \\
    & = \frac{N'}{N-N'} \inprod{\left(1 - \sum_{k=N'}^{N-1} 2h_{k,N'} \right) g_{N'} - \sum_{j=N'+1}^{N-1} \sum_{k=j}^{N-1} 2(H_2)_{k-N',j-N'} g_j}{g_{N'}} - \frac{N'}{N-N'} \inprod{g_N}{g_{N'}} + \frac{N'}{N} \inprod{g_{N'}}{y_0 - y_{N'-1}} \\
    & = -\frac{N'(N'-1)}{N} \sqnorm{g_{N'}} - \frac{N'}{N-N'} \inprod{\sum_{j=N'+1}^{N-1} 2Q(1,j-N'; H_2) g_j}{g_{N'}} - \frac{N'}{N-N'} \inprod{g_N}{g_{N'}} + \frac{N'}{N} \inprod{g_{N'}}{y_0 - y_{N'-1}} 
\end{align*}
where the second last equality uses \eqref{eqn:gluing-true-y-sequence} and the last equality uses
\begin{align*}
    \frac{N'}{N-N'} (1-2Q(1,N';H)) = \frac{N'}{N-N'} \left(1 - \frac{N'(N-N'+1)}{N}\right) = \frac{N'}{N-N'} \cdot \frac{-(N-N')(N'-1)}{N} = -\frac{N'(N'-1)}{N} .
\end{align*}
Plugging the new expression for $(\mathrm{II})$ back into \eqref{eqn:ineq-summation-simplified-second}, we obtain
\begin{align*}
    - N\sqnorm{g_N} - \inprod{g_N}{x_N - y_0} - \underbrace{\left( \frac{N'(N-N'-1)}{N-N'} \inprod{g_{N'}}{g_N} - \frac{2N'}{N-N'} \inprod{\sum_{j=N'+1}^{N-1} Q(1,j-N'; H_2) g_j}{g_{N'}} \right)}_{\mathrm{(III)}} - (\mathrm{I}) .
\end{align*}
It remains to show that $\mathrm{(I)} + \mathrm{(III)} = 0$, which will conclude the proof.

We investigate the expression $\mathrm{(I)}$ now. 
With appropriate changes of summation order and index swapping, we have
\begin{align}
    \mathrm{(I)} & = \frac{N}{N-N'} \sum_{j=N'+1}^{N-1} \sum_{k=j+1}^N \lambda_{k-N',j-N'}^\star (H_2) \inprod{\sum_{\ell=j}^{k-1} 2h_{\ell,N'} g_{N'} }{g_k - g_j } \nonumber \\
    & = \frac{N}{N-N'} \sum_{j=N'+1}^{N-1} \sum_{k=j+1}^N \lambda_{k-N',j-N'}^\star (H_2) \sum_{\ell=j}^{k-1} 2h_{\ell,N'} \inprod{g_{N'}}{g_k} \nonumber \\
    & \quad - \frac{N}{N-N'} \sum_{j=N'+1}^{N-1} \sum_{k=j+1}^N \lambda_{k-N',j-N'}^\star (H_2) \sum_{\ell=j}^{k-1} 2h_{\ell,N'} \inprod{g_{N'}}{g_j} \nonumber \\
    & = \frac{N}{N-N'} \sum_{k=N'+1}^N \sum_{j=N'+1}^{k-1} \lambda_{k-N',j-N'}^\star (H_2) \sum_{\ell=j}^{k-1} 2h_{\ell,N'} \inprod{g_{N'}}{g_k} \nonumber \\
    & \quad - \frac{N}{N-N'} \sum_{j=N'+1}^{N-1} \sum_{k=j+1}^N \lambda_{k-N',j-N'}^\star (H_2) \sum_{\ell=j}^{k-1} 2h_{\ell,N'} \inprod{g_{N'}}{g_j} \nonumber \\
    & = \frac{N}{N-N'} \sum_{j=N'+1}^{N-1} \lambda_{N-N',j-N'}^\star (H_2) \sum_{\ell=j}^{N-1} 2h_{\ell,N'} \inprod{g_{N'}}{g_N} \nonumber \\
    & \quad + \frac{N}{N-N'} \sum_{j=N'+1}^{N-1} \sum_{k=N'+1}^{j-1} \lambda_{j-N',k-N'}^\star (H_2) \sum_{\ell=k}^{j-1} 2h_{\ell,N'} \inprod{g_{N'}}{g_j} \nonumber \\
    & \quad - \frac{N}{N-N'} \sum_{j=N'+1}^{N-1} \sum_{k=j+1}^N \lambda_{k-N',j-N'}^\star (H_2) \sum_{\ell=j}^{k-1} 2h_{\ell,N'} \inprod{g_{N'}}{g_j} \nonumber \\
    &
    \begin{aligned}
        & = \frac{N}{N-N'} \sum_{j=N'+1}^{N-1} \sum_{\ell=j}^{N-1}  2h_{\ell,N'} \lambda_{N-N',j-N'}^\star (H_2) \inprod{g_{N'}}{g_N} \\
        & \quad + \frac{N}{N-N'} \sum_{j=N'+1}^{N-1} \inprod{g_{N'} }{\left[ \sum_{k=N'+1}^{j-1} \sum_{\ell=k}^{j-1} 2h_{\ell,N'} \lambda_{j-N',k-N'}^\star (H_2) - \sum_{k=j+1}^N \sum_{\ell=j}^{k-1} 2h_{\ell,N'} \lambda_{k-N',j-N'}^\star (H_2) \right] g_j } .    
    \end{aligned}
    \label{eqn:ineq-summation-simplified-third}
\end{align}
Here, to simplify the coefficient of $\inprod{g_{N'}}{g_N}$ in the first term of \eqref{eqn:ineq-summation-simplified-third} we note that
\begin{align*}
    \lambda_{N-N',j-N'}^\star (H_2) = (N-N') \sum_{m=1}^{(N-N')-(j-N')} (-1)^{m-1} Q(m, j-N'; H_2) = (N-N') \sum_{m=1}^{N-j} (-1)^{m-1} Q(m,j; H)
\end{align*}
because $Q(\ell,j; H)$ for $j \ge N'+1$ relies only on the last $N-N'-1$ columns of $H$, which agrees with $H_2$.
With this characterization, we have
\begin{align*}
    & \frac{N}{N-N'} \sum_{j=N'+1}^{N-1} \sum_{\ell=j}^{N-1}  2h_{\ell,N'} \lambda_{N-N',j-N'}^\star (H_2) \\
    & = \frac{N}{N-N'} \sum_{j=N'+1}^{N-1} \left( 2Q(1,N'; H) - \sum_{\ell=N'}^{j-1} 2h_{\ell,N'} \right) \lambda_{N-N',j-N'}^\star (H_2) \\
    & = \frac{N \cdot 2Q(1,N'; H)}{N-N'} \sum_{j=N'+1}^{N-1} \lambda_{N-N',j-N'}^\star (H_2) - N \sum_{j=N'+1}^{N-1} \sum_{\ell=N'}^{j-1} 2h_{\ell,N'} \sum_{m=1}^{N-j} (-1)^{m-1} Q(m,j; H) \\
    & = \frac{N'(N-N'+1)}{N-N'} \sum_{i=1}^{N-N'-1} \lambda_{N-N',i}^\star (H_2) - 2N \sum_{m=1}^{N-N'-1} (-1)^{m-1} \sum_{j=N'+1}^{N-m} Q(m,j; H) \sum_{\ell=N'}^{j-1} h_{\ell,N'} \\
    & \stackrel{(\ast)}{=} \frac{N'(N-N'+1)(N-N'-1)}{N-N'} - 2N \sum_{m=1}^{N-N'-1} (-1)^{m-1} Q(m+1, N'; H) \\
    & \stackrel{(\bullet)}{=} \frac{N'(N-N'+1)(N-N'-1)}{N-N'} + 2N \left( \frac{N'}{N(N-N')} - Q(1,N';H) \right) \\
    & = \frac{N'(N-N'+1)(N-N'-1)}{N-N'} + \frac{2N'}{N-N'} - N'(N-N'+1) \\
    & = -\frac{N'(N-N'-1)}{N-N'} 
\end{align*}
where $(\ast)$ uses \cref{lemma:Q-lambda-identities}(b) for $\sum_{i=1}^{N-N'-1} \lambda_{N-N',i}^\star (H_2) = N-N'-1$ and \cref{lemma:Q-lambda-identities}(a) for 
\[
    \sum_{j=N'+1}^{N-m} Q(m,j; H) \sum_{\ell=N'}^{j-1} h_{\ell,N'} = Q(m+1,N';H) ,
\]
and $(\bullet)$ uses \cref{lemma:gluing-helper-identities}(b) with $\sum_{n=1}^{N-N'} (-1)^{n-1} Q(n,N';H) = \frac{N'}{N(N-N')}$.
Applying this simplification to \eqref{eqn:ineq-summation-simplified-third} and comparing it with $\mathrm{(III)}$, we see that the $\inprod{g_{N'}}{g_N}$ terms cancel out in $\mathrm{(I)} + \mathrm{(III)}$ and only the terms of the form $\inprod{g_{N'}}{g_j}$ for $N'+1 \le j \le N-1$ remain.
To handle this, we consider the second term in \eqref{eqn:ineq-summation-simplified-third}.
In the following, all occurrences of $Q(\cdot,\cdot)$ will be $Q(\cdot,\cdot;H)$ unless otherwise specified.
For each $j=N'+1,\dots,N-1$, we have
\begin{align}
    & \sum_{k=j+1}^N \sum_{\ell=j}^{k-1} 2h_{\ell,N'} \lambda_{k-N',j-N'}^\star (H_2) \nonumber \\
    & = \sum_{k=j+1}^{N} \left( \sum_{\ell=N'}^{k-1} 2h_{\ell,N'} - \sum_{\ell=N'}^{j-1} 2h_{\ell,N'} \right) \lambda_{k-N',j-N'}^\star (H_2) \nonumber \\
    & = \sum_{k=j+1}^{N-1} \left( \sum_{\ell=N'}^{k-1} 2h_{\ell,N'} - \sum_{\ell=N'}^{j-1} 2h_{\ell,N'} \right) \lambda_{k-N',j-N'}^\star (H_2) + \left( \sum_{\ell=N'}^{N-1} 2h_{\ell,N'} - \sum_{\ell=N'}^{j-1} 2h_{\ell,N'} \right) \lambda_{N-N',j-N'}^\star (H_2) \nonumber \\
    & = (N-N') \sum_{k=j+1}^{N-1} \sum_{\ell=N'}^{k-1} 2h_{\ell,N'} \sum_{m=1}^{N-N'-1} \sum_{n=1}^{N-N'-1} (-1)^{m+n-1} \binom{m+n}{m} Q(m,j) Q(n,k) - \sum_{k=j+1}^{N-1} \left( \sum_{\ell=N'}^{j-1} 2h_{\ell,N'} \right) \lambda_{k-N',j-N'}^\star (H_2) \nonumber \\
    & \quad + (N-N') \left( \sum_{\ell=N'}^{N-1} 2h_{\ell,N'} \right) \sum_{m=1}^{N-N'-1} (-1)^{m-1} Q(m,j) - \left( \sum_{\ell=N'}^{j-1} 2h_{\ell,N'} \right) \lambda_{N-N',j-N'}^\star (H_2) \nonumber \\
    & = (N-N') \sum_{m=1}^{N-N'-1} \sum_{n=1}^{N-N'-1} (-1)^{m+n-1} \binom{m+n}{m} Q(m,j) \sum_{k=j+1}^{N-1} Q(n,k) \sum_{\ell=N'}^{k-1} 2h_{\ell,N'} \nonumber \\
    & \quad + 2(N-N') Q(1,N') \sum_{m=1}^{N-N'-1} (-1)^{m-1} Q(m,j) - \left( \sum_{\ell=N'}^{j-1} 2h_{\ell,N'} \right) \left( \sum_{k=j+1}^{N} \lambda_{k-N',j-N'}^\star (H_2) \right) \nonumber \\
    & = (N-N') \sum_{m=1}^{N-N'-1} \sum_{n=1}^{N-N'-1} (-1)^{m+n-1} \binom{m+n}{m} Q(m,j) \underbrace{\sum_{k=N'+1}^{N-1} Q(n,k) \sum_{\ell=N'}^{k-1} 2h_{\ell,N'}}_{=\,2Q(n+1,N') \text{ by \cref{lemma:Q-lambda-identities}(a)}} \nonumber \\
    & \quad - (N-N') \sum_{m=1}^{N-N'-1} \sum_{n=1}^{N-N'-1} (-1)^{m+n-1} \binom{m+n}{m} Q(m,j) \sum_{k=N'+1}^{j} Q(n,k) \sum_{\ell=N'}^{k-1} 2h_{\ell,N'} \nonumber \\
    & \quad + 2(N-N') Q(1,N') \sum_{m=1}^{N-N'-1} (-1)^{m-1} Q(m,j) - \left( \sum_{\ell=N'}^{j-1} 2h_{\ell,N'} \right) \left( \sum_{k=j+1}^{N} \lambda_{k-N',j-N'}^\star (H_2) \right) \nonumber \\
    & = 2(N-N') \left[ \sum_{m=1}^{N-N'-1} \sum_{n=1}^{N-N'-1} (-1)^{m+n-1} \binom{m+n}{m} Q(m,j) Q(n+1,N') + Q(1,N') \sum_{m=1}^{N-N'-1} (-1)^{m-1} Q(m,j) \right] \nonumber \\
    & \quad - (N-N') \sum_{m=1}^{N-N'-1} \sum_{n=1}^{N-N'-1} \sum_{k=N'+1}^{j} (-1)^{m+n-1} \binom{m+n}{m} Q(m,j) Q(n,k) \sum_{\ell=N'}^{k-1} 2h_{\ell,N'} \nonumber \\
    & \quad - \left( \sum_{\ell=N'}^{j-1} 2h_{\ell,N'} \right) \left( \sum_{k=j+1}^{N} \lambda_{k-N',j-N'}^\star (H_2) \right) \nonumber \\
    &
    \begin{aligned}
        & = 2(N-N') \left[ \sum_{m=1}^{N-N'-1} \sum_{n=1}^{N-N'} (-1)^{m+n} \binom{m+n-1}{m} Q(m,j) Q(n,N') \right] \\
        & \quad - (N-N') \sum_{m=1}^{N-N'-1} \sum_{n=1}^{N-N'-1} \sum_{k=N'+1}^{j} (-1)^{m+n-1} \binom{m+n}{m} Q(m,j) Q(n,k) \sum_{\ell=N'}^{k-1} 2h_{\ell,N'} \\
        & \quad - \left( \sum_{\ell=N'}^{j-1} 2h_{\ell,N'} \right) \left( \sum_{k=j+1}^{N} \lambda_{k-N',j-N'}^\star (H_2) \right) .
    \end{aligned}
    \label{eqn:kjp1-to-N-ell-j-to-k-1-summation}
\end{align}
On the other hand,
\begin{align*}
    & \sum_{k=N'+1}^{j-1} \sum_{\ell=k}^{j-1} 2h_{\ell,N'} \lambda_{j-N',k-N'}^\star (H_2) \\
    & = (N-N') \sum_{k=N'+1}^{j-1} \sum_{\ell=k}^{j-1} 2h_{\ell,N'} \sum_{m=1}^{N-N'-1} \sum_{n=1}^{N-N'-1} (-1)^{m+n-1} \binom{m+n}{m} Q(m,j) Q(n,k) \\
    & = (N-N') \sum_{m=1}^{N-N'-1} \sum_{n=1}^{N-N'-1} \sum_{k=N'+1}^{j-1} (-1)^{m+n-1} \binom{m+n}{m} Q(m,j) Q(n,k) \sum_{\ell=k}^{j-1} 2h_{\ell,N'} \\
    & = (N-N') \sum_{m=1}^{N-N'-1} \sum_{n=1}^{N-N'-1} \sum_{k=N'+1}^{j} (-1)^{m+n-1} \binom{m+n}{m} Q(m,j) Q(n,k) \sum_{\ell=k}^{j-1} 2h_{\ell,N'} 
\end{align*}
where in the last equality, the summation over $k$ is rewritten to include $k=j$, which makes no change since in this case, the summand is vacuous: $\sum_{\ell=k}^{j-1} 2h_{\ell,N'} = 0$.
Now combining the above with \eqref{eqn:kjp1-to-N-ell-j-to-k-1-summation}, we obtain
\begin{align*}
    & \sum_{k=N'+1}^{j-1} \sum_{\ell=k}^{j-1} 2h_{\ell,N'} \lambda_{j-N',k-N'}^\star (H_2) - \sum_{k=j+1}^N \sum_{\ell=j}^{k-1} 2h_{\ell,N'} \lambda_{k-N',j-N'}^\star (H_2) \\
    & = 2(N-N') \left[ \sum_{m=1}^{N-N'-1} \sum_{n=1}^{N-N'} (-1)^{m+n-1} \binom{m+n-1}{m} Q(m,j) Q(n,N') \right] \\
    & \quad + (N-N') \sum_{m=1}^{N-N'-1} \sum_{n=1}^{N-N'-1} \sum_{k=N'+1}^{j} (-1)^{m+n-1} \binom{m+n}{m} Q(m,j) Q(n,k) \sum_{\ell=N'}^{j-1} 2h_{\ell,N'} \\
    & \quad + \left( \sum_{\ell=N'}^{j-1} 2h_{\ell,N'} \right) \left( \sum_{k=j+1}^{N} \lambda_{k-N',j-N'}^\star (H_2) \right) \\
    & = 2(N-N') \left[ \sum_{m=1}^{N-N'-1} \sum_{n=1}^{N-N'} (-1)^{m+n-1} \binom{m+n-1}{m} Q(m,j) Q(n,N') \right] \\
    & \quad + \left( \sum_{\ell=N'}^{j-1} 2h_{\ell,N'} \right) \left( \sum_{k=N'+1}^{j-1} \lambda_{j-N',k-N'}^\star (H_2) \right) \\
    & \quad + (N-N') \left( \sum_{\ell=N'}^{j-1} 2h_{\ell,N'} \right) \sum_{m=1}^{N-N'-1} \sum_{n=1}^{N-N'-1} (-1)^{m+n-1} \binom{m+n}{m} Q(m,j) Q(n,j) \\
    & \quad + \left( \sum_{\ell=N'}^{j-1} 2h_{\ell,N'} \right) \left( \sum_{k=j+1}^{N} \lambda_{k-N',j-N'}^\star (H_2) \right) \\
    & \stackrel{(\diamond)}{=} 2(N-N') \sum_{m=1}^{N-N'-1} (-1)^{m-1} Q(m,j; H) \underbrace{\sum_{n=1}^{N-N'} (-1)^{n} \binom{m+n-1}{m} Q(n,N'; H)}_{=\begin{cases} \frac{N'}{N(N-N')} & \text{if } m=1 \\ 0 & \text{otherwise} \end{cases} \text{ by \cref{lemma:gluing-helper-identities}(b)}} \\
    & = \frac{2N'}{N} Q(1,j; H) \\
    & = \frac{2N'}{N} Q(1,j-N'; H_2) 
\end{align*}
where $(\diamond)$ uses \cref{lemma:Q-lambda-identities}(c), together with the identities $Q(m,j)Q(n,j) = Q(m,j-N';H_2) Q(n,j-N';H_2)$.
Now applying all of the above simplifications to \eqref{eqn:ineq-summation-simplified-third}, we conclude that
\begin{align*}
    \mathrm{(I)} & = -\frac{N'(N-N'-1)}{N-N'} \inprod{g_{N'}}{g_N} + \frac{N}{N-N'} \sum_{j=N'+1}^{N-1} \frac{2N'}{N} Q(1,j-N'; H_2) \inprod{g_{N'}}{g_j} \\
    & = -\frac{N'(N-N'-1)}{N-N'} \inprod{g_{N'}}{g_N} + \frac{2N'}{N-N'} \sum_{j=N'+1}^{N-1} Q(1,j-N'; H_2) \inprod{g_{N'}}{g_j} \\
    & = -\mathrm{(III)} 
\end{align*}
which completes the proof.
\hfill
\qedsymbol

\section{Proof of \texorpdfstring{\cref{theorem:lambda-of-H-dual}}{Theorem 3}}
\label{section:proof-of-lambda-inversion-theorem}

Let $\cG=\cG(H)$ be the weighted diagram of $H$ from \cref{definition:diagram-of-algorithm}. Since $H\in\cH_\star(N-1)$,
all H-certificates are nonnegative, and by
\cref{lemma:nonzero-lambda-exists}, the support graph of $\cG$ is connected.
Therefore $L(\cG)$ is the Laplacian of a weighted connected graph, so it is positive semidefinite and $\ker L(\cG) = \mathrm{Span}\{\mathbf{1}\}$.
This implies that $-\Lambda(H) = L(\cG) + E_{N,N}$ is positive definite, and in particular, invertible.

Next, define the $N\times N$ matrices 
\begin{align*}
    J = \begin{bmatrix} 
        1 \\
        1 & 1 \\
        \vdots & \vdots & \ddots \\
        1 & 1 & \cdots & 1
    \end{bmatrix} , \quad 
    \widetilde{H} = \begin{bmatrix}
        0 \\
        h_{1,1} & 0 \\
        h_{2,1} & h_{2,2} & 0 \\
        \vdots & \vdots & \ddots & \ddots \\
        h_{N-1,1} & h_{N-1,2} & \cdots & h_{N-1,N-1} & 0
    \end{bmatrix} .
\end{align*}
We will show that H-certificates, or equivalently the matrix $\Lambda(H)$ containing them, can be uniquely characterized via identity~\eqref{eqn:core-identity-matrix-version} involving the above matrices.
The idea of the proof is to manipulate \eqref{eqn:core-identity-matrix-version} and derive a corresponding identity for $H^\at$, which will provide a characterization of the H-certificates of $H^\at$.

\begin{lemma}
\label{lemma:Lambda-is-unique-solution-to-matrix-equation}
Let $H \in \reals^{(N-1)\times (N-1)}$ satisfy the H-invariance conditions~\eqref{eqn:H-invariance}. Then, there exists a unique $N\times N$ symmetric matrix $\Lambda$ satisfying $\Lambda \mathbf{1} = -e_N$ and
\begin{align}
\label{eqn:core-identity-matrix-version}
    NE_{N,N} + \Lambda + \Lambda J\widetilde{H} + (\Lambda J\widetilde{H})^\transpose = O ,
\end{align}
namely, $\Lambda = \Lambda(H)$.
\end{lemma}

\begin{proof}[Proof of \cref{lemma:Lambda-is-unique-solution-to-matrix-equation}]

For brevity, write $\lambda_{k,j}=\lambda_{k,j}^\star(H)$.
Let $y_0, g_1, \dots, g_N \in \cX$ and let $y_1, \dots, y_{N-1}$ be defined as
\begin{align*}
    y_{k+1} = y_k - \sum_{j=1}^{k+1} 2h_{k+1,j} g_{j} 
\end{align*}
for $k=0,\dots,N-2$, and $x_j = y_{j-1} - g_j$ for $j=1,\dots,N$. 
Because $H$ satisfies the H-invariance, and the H-certificates are characterized as the multipliers appearing in the ``universal proof template'', the identity \eqref{eqn:optimal-family-proof-core-identity} must hold.
Now if we define
\begin{align*}
    G = \begin{bmatrix} g_1 & g_2 & \cdots & g_N \end{bmatrix}, \quad X = \begin{bmatrix} x_1 - y_0 & x_2 - y_0 & \cdots & x_N - y_0 \end{bmatrix} \,\, \in \,\, \reals^{\dim(\cX) \times N} ,
\end{align*}
then 
\[
    \inprod{x_k - x_j}{g_k - g_j} = \inprod{(x_k - y_0) - (x_j - y_0)}{g_k - g_j} = (G^\transpose X)_{k,k} + (G^\transpose X)_{j,j} - (G^\transpose X)_{k,j} - (G^\transpose X)_{j,k} .
\]
Hence, we can rewrite \eqref{eqn:optimal-family-proof-core-identity} as
\begin{align}
    0 & = N (G^\transpose G)_{N,N} + (G^\transpose X)_{N,N} + \sum_{1\le j < k \le N} \lambda_{k,j} \left( (G^\transpose X)_{k,k} + (G^\transpose X)_{j,j} - (G^\transpose X)_{k,j} - (G^\transpose X)_{j,k} \right) \nonumber \\
    & = N (G^\transpose G)_{N,N} + (G^\transpose X)_{N,N} + \sum_{k=1}^N \left( \sum_{j=1}^{k-1} \lambda_{k,j} + \sum_{j=k+1}^N \lambda_{j,k} \right) (G^\transpose X)_{k,k} - \sum_{1\le j < k \le N}\lambda_{k,j} \left( (G^\transpose X)_{k,j} + (G^\transpose X)_{j,k} \right) \nonumber \\
    & = N (G^\transpose G)_{N,N} - \mathrm{Tr}\left( \Lambda(H) G^\transpose X \right) \nonumber \\
    & = N \cdot \mathrm{Tr}(E_{N,N} G^\transpose G) - \mathrm{Tr}\left( \Lambda(H) G^\transpose X \right) . \label{eqn:proof-template-rewritten-with-matrices}
\end{align}
Next, observe that
\begin{align*}
    x_k - y_0 = y_{k-1} - y_0 - g_k = \sum_{j=1}^{k-1} (y_j - y_{j-1}) - g_k = -\sum_{j=1}^{k-1} \sum_{\ell=1}^j 2h_{j,\ell} g_\ell - g_k = -\sum_{\ell=1}^{k-1} \left( \sum_{j=\ell}^{k-1} 2h_{j,\ell} \right) g_\ell - g_k
\end{align*}
and on the other hand, we have
\begin{align*}
    \widetilde{H} G^\transpose = 
    \begin{bmatrix}
        0 \\
        h_{1,1} g_1^\transpose \\
        h_{2,1} g_1^\transpose + h_{2,2} g_2^\transpose \\
        \vdots \\
        h_{N-1,1} g_1^\transpose + \dots + h_{N-1,N-1} g_{N-1}^\transpose
    \end{bmatrix} 
    \implies
    J\widetilde{H} G^\transpose = 
    \begin{bmatrix}
        0 \\
        h_{1,1} g_1^\transpose \\
        (h_{1,1} + h_{2,1}) g_1^\transpose + h_{2,2} g_2^\transpose \\
        \vdots \\
        (h_{1,1} + \dots + h_{N-1,1}) g_1^\transpose + \dots + h_{N-1,N-1} g_{N-1}^\transpose
    \end{bmatrix} 
\end{align*}
which shows that
\begin{align*}
    X = \begin{bmatrix} -g_1 & -2h_{1,1}g_1 - g_2 & -2(h_{1,1}+h_{2,1})g_1 - 2h_{2,2}g_2 - g_3 & \cdots & -\sum_{\ell=1}^{N-1} \sum_{j=\ell}^{N-1} 2h_{j,\ell} g_\ell - g_N \end{bmatrix} = -2G\widetilde{H}^\transpose J^\transpose - G .
\end{align*}
Substituting this into \eqref{eqn:proof-template-rewritten-with-matrices}, we obtain
\begin{align*}
    0 = N \cdot \mathrm{Tr}(E_{N,N} G^\transpose G) + \mathrm{Tr}\left( \Lambda(H) G^\transpose G (2\widetilde{H}^\transpose J^\transpose + I) \right) = N \cdot \mathrm{Tr}(E_{N,N} G^\transpose G) + \mathrm{Tr}\left( (2\widetilde{H}^\transpose J^\transpose + I) \Lambda(H) G^\transpose G \right) .
\end{align*}
Denoting the space of all symmetric $N\times N$ matrices by $\mathbb{S}^{N\times N}$ and endowing it with the inner product $\inprod{A}{B} = \mathrm{Tr}(AB)$, the previous identity can be rewritten as
\begin{align*}
    0 & = \inprod{N E_{N,N} + \frac{1}{2} \left[ (2\widetilde{H}^\transpose J^\transpose + I) \Lambda(H) + \left( (2\widetilde{H}^\transpose J^\transpose + I) \Lambda(H) \right)^\transpose \right]}{G^\transpose G} \\
    & = \inprod{N E_{N,N} + \widetilde{H}^\transpose J^\transpose \Lambda(H) + \Lambda(H)J\widetilde{H} + \Lambda(H) }{G^\transpose G}
\end{align*}
where we use $\Lambda(H) = \Lambda(H)^\transpose$.
Then the matrix $NE_{N,N} + \Lambda(H) + \Lambda(H) J\widetilde{H} + (\Lambda(H) J\widetilde{H})^\transpose$ is symmetric and has inner product $0$ with any rank-$1$ positive semidefinite matrix (because $G$ is arbitrary), which implies
\begin{align*}
    NE_{N,N} + \Lambda(H) + \Lambda(H) J\widetilde{H} + (\Lambda(H) J\widetilde{H})^\transpose = O .
\end{align*}
Finally, $\Lambda(H)\mathbf{1} = -e_N$ because $\Lambda(H) = -L(\cG) - E_{N,N}$ and $\mathbf{1} \in \ker L(\cG)$.

Conversely, if $\Lambda \in \mathbb{S}^{N\times N}$ satisfies $\Lambda \mathbf{1} = -e_N$ and \eqref{eqn:core-identity-matrix-version}, then we can reverse the above derivation and conclude that the off-diagonal entries of $\Lambda$ satisfy \eqref{eqn:optimal-family-proof-core-identity}.
Then, by uniqueness of H-certificates, those off-diagonals must be $\lambda_{k,j}$.
The condition $\Lambda\mathbf{1} = -e_N$ then determines the diagonal entries, and we conclude that $\Lambda = \Lambda(H)$. 
\end{proof}

We introduce another matrix that has useful properties for handling the algebra.

\begin{lemma}
\label{lemma:T-properties}
Let $T = \begin{bmatrix} 
    {} & {} & {} & 1 \\
    & & 1 \\
    & \iddots \\
    1 &
\end{bmatrix} \in \reals^{N\times N}$.
Then $T$ is symmetric, and satisfies
\begin{itemize}
    \item $T^2 = I_N$,
    \item $JT = TJ^\transpose$,
    \item $TM^\transpose T = M^\at$ for any lower triangular $M \in \reals^{N\times N}$,
    \item $(JT)^{-1} = TJ^{-1} = \Delta$, $\Delta^\transpose=\Delta$, and $\Delta \mathbf{1} = e_N$.
\end{itemize}
\end{lemma}

\begin{proof}[Proof of \cref{lemma:T-properties}]
We can express the entries of $T, J, M^\at$ as
\[
(T)_{ij}=\delta_{i,N+1-j}, \quad
(J)_{ij}=
\begin{cases}
    1 & \text{if } i\ge j \\
    0 & \text{if } i < j
\end{cases} ,
\quad
(M^\at)_{ij} = (M)_{N+1-j,N+1-i}.
\]
The symmetry of $T$ and $T^2 = I_N$ are obvious.
Next, 
\[
(JT)_{ij}
= \sum_{k=1}^N (J)_{ik} (T)_{kj}
= J_{i,N+1-j} 
=
\begin{cases}
    1 & \text{if } i+j\ge N+1 \\
    0 & \text{if } i+j<N+1
\end{cases} ,
\]
and similarly
\[
(TJ^\transpose)_{ij}
=\sum_{k=1}^N T_{ik}J^\transpose_{kj}
=J^\transpose_{N+1-i,j}
=J_{j,N+1-i}
=
\begin{cases}
    1 & \text{if } i+j\ge N+1 \\
    0 & \text{if } i+j<N+1
\end{cases} ,
\]
which shows that \(JT=TJ^\transpose\).

Now let \(M\in\reals^{N\times N}\) be lower triangular. Then
\[
(TM^\transpose T)_{ij}
=\sum_{k,\ell=1}^N T_{ik}(M^\transpose)_{k\ell}T_{\ell j}
=(M^\transpose)_{N+1-i,N+1-j}
=M_{N+1-j,N+1-i}
=(M^\at)_{ij},
\]
and therefore \(TM^\transpose T=M^\at\).

Finally, $(JT)^{-1}=T^{-1}J^{-1}=TJ^{-1}$, and because
\[
J^{-1}
=
\begin{bmatrix}
1      &        &        &        \\
-1     & 1      &        &        \\
       & \ddots & \ddots &        \\
       &        & -1     & 1
\end{bmatrix}
\]
we have
\[
(TJ^{-1})_{ij} 
= (J^{-1})_{N+1-i,j}
=
\begin{cases}
    1, & j=N+1-i,\\
    -1, & j=N-i,\\
    0, & \text{otherwise},
\end{cases}
\]
which agrees with \(\Delta\). 
Then clearly $\Delta^\transpose = \left((JT)^\transpose\right)^{-1} = (TJ^\transpose)^{-1} = (JT)^{-1} = \Delta$.
It is straightforward to see $\Delta \mathbf{1} = e_N$.
\end{proof}

\vspace{0.2cm}
\noindent
For simplicity, we write $\Lambda = \Lambda (H)$ from now on.
Multiplying $(JT\Lambda)^\transpose = \Lambda TJ^\transpose = \Lambda JT$ to the left and $JT\Lambda$ to the right to \eqref{eqn:core-identity-matrix-version}, we obtain
\begin{align*}
    & \Lambda JT \left( NE_{N,N} + \Lambda + \Lambda J\widetilde{H} + (\Lambda J\widetilde{H})^\transpose \right) JT\Lambda = O \\
    & \implies N \cdot \Lambda JTE_{N,N} JT \Lambda + \Lambda (JT\Lambda JT) \Lambda + \Lambda (JT\Lambda JT) T\widetilde{H} JT\Lambda + \Lambda (JT\widetilde{H}^\transpose T) (TJ^\transpose \Lambda JT)\Lambda = O 
\end{align*}
where we use symmetry of $\Lambda$ and $T$, and $T^2 = I_N$.
Multiplying the above by $\Lambda^{-1}$ from the left and right sides and using $JT = TJ^\transpose$, we obtain
\begin{align*}
    O & = N \cdot JTE_{N,N} JT + JT\Lambda JT + (JT\Lambda JT) (T\widetilde{H}T) J^\transpose + J(T\widetilde{H}^\transpose T) (JT\Lambda JT) \\
    & = N \cdot JTE_{N,N} JT + JT\Lambda JT + (JT\Lambda JT) \big(\widetilde{H}^\at\big)^\transpose J^\transpose + J\widetilde{H}^\at (JT\Lambda JT)
\end{align*}
where the last line uses $\widetilde{H}^\at = T\widetilde{H}^\transpose T$ and $\big(\widetilde{H}^\at\big)^\transpose = T\widetilde{H} T$ which follows from \cref{lemma:T-properties}. Multiplying $(JT\Lambda JT)^{-1}$ to both left and right sides of the above, we obtain
\begin{align*}
    N \cdot (\Lambda JT)^{-1} E_{N,N} (JT\Lambda)^{-1} + (JT\Lambda JT)^{-1} + \big(\widetilde{H}^\at\big)^\transpose J^\transpose (JT\Lambda JT)^{-1} + (JT\Lambda JT)^{-1} J\widetilde{H}^\at = O .
\end{align*}
Note that because $\Lambda \mathbf{1} = -e_N$, we have $(\Lambda JT)^{-1} e_N = (JT)^{-1} \Lambda^{-1} e_N = -\Delta \mathbf{1} = -e_N$, so $(\Lambda JT)^{-1} E_{N,N} = -E_{N,N}$. Hence
\begin{align*}
    (\Lambda JT)^{-1} E_{N,N} (JT\Lambda)^{-1} = (-E_{N,N}) (TJ^\transpose \Lambda)^{-1} = -\left( (\Lambda JT)^{-1} E_{N,N}\right)^\transpose = E_{N,N} ,
\end{align*}
which shows that $\Phi(\Lambda) = (JT\Lambda JT)^{-1} = \Delta \Lambda^{-1} \Delta$ satisfies
\begin{align*}
    N E_{N,N} + \Phi(\Lambda) + \big(\widetilde{H}^\at\big)^\transpose J^\transpose \Phi(\Lambda) + \Phi(\Lambda) J\widetilde{H}^\at = O .
\end{align*}
Note that $\Phi(\Lambda)$ is symmetric since $\Delta$ and $\Lambda$ are, and
\begin{align*}
    \Phi(\Lambda) \mathbf{1} = \Delta \Lambda^{-1} \Delta \mathbf{1} = \Delta \Lambda^{-1} e_N = -\Delta \mathbf{1} = -e_N .
\end{align*}
Therefore, applying \cref{lemma:Lambda-is-unique-solution-to-matrix-equation} with $H^\at$ in place of $H$ (which is possible since $H^\at$ satisfies H-invariance conditions whenever $H$ does \citep{YoonKimSuhRyu2024_optimal}), we conclude that 
\begin{align*}
    \Lambda(H^\at) = \Phi(\Lambda) = (JT\Lambda JT)^{-1} = \Delta \Lambda^{-1} \Delta
\end{align*}
as desired.
Here, we used the fact that the $N\times N$ embedding $H \mapsto \widetilde H$ is compatible with H-dual, i.e., $\widetilde{H^\at} = \widetilde{H}^\at .$

\section{Proof of \texorpdfstring{\cref{proposition:FSDM-recurrence}}{Proposition 5}}
\label{section:proof-of-fsdm-recurrence}

We prove the slightly stronger, purely algebraic statement that the stated recurrence holds for the update rule induced by \(H^{(n)}\) with \eqref{eqn:H-matrix-representation-via-g-vectors} for arbitrary vectors \(g_1,\dots,g_{2^n-1}\). 
The claim then follows by taking \(g_t=\frac12(y_{t-1}-\opT y_{t-1})\).

Recall that we denote $H^{(n)} = H_{\text{FSDM}}(2^n-1)\in \mathbb R^{(2^n-1)\times(2^n-1)}$.
We use induction on $n$ to show that \eqref{eqn:FSDM-recurrence} holds for all $1\le j\le 2^n - 1$.
When $n=1$, because $H^{(1)} = \begin{bmatrix} \frac{1}{2} \end{bmatrix}$, we have $y_1 = y_0 - 2h_{1,1} g_1 = y_0 - g_1$ and this agrees with \eqref{eqn:FSDM-recurrence}. 

Now let $n\ge 1$ and suppose that \eqref{eqn:FSDM-recurrence} holds for all $1\le j\le 2^n - 1$.
We consider the iterations $2^n \le j \le 2^{n+1} - 1$.
When $j=2^n$, then applying \cref{theorem:gluing} with $N'=2^n$ and $N=2^{n+1}$ shows that the update rule must be
\begin{align*}
    y_{2^n} = y_{2^n - 1} - 2h_{2^n, 2^n} g_{2^n} + \frac{1}{2} \left( y_0 - y_{2^n - 1} \right) = y_{2^n - 1} - 2^n g_{2^n} + \frac{1}{2} \left( y_0 - y_{2^n - 1}\right) 
\end{align*}
where we use $h_{2^n,2^n} = 2^{n-1}$, which follows from \cref{lemma:FSDM-H-matrix}. 
Next, for $2^n < j \le 2^{n+1} - 1$, by the block structure of $H^{(n+1)}$ in \cref{lemma:FSDM-H-matrix}, if we let $j' = j - 2^n$ then the $j$-th row of $H^{(n+1)}$ agrees with the $j'$-th row (up to shift of column indices by $2^n$) except for the $2^n$-th column elements $h_{j,2^{n}} = -\frac{1}{2} b_{j'}^{(n)}$.
Hence
\begin{align*}
    y_{j} = y_{2^n} - \sum_{i=2^n+1}^{j} 2h_{i,2^n} g_{2^n} - \sum_{i=2^n+1}^{j} \sum_{u=i}^j 2h_{u,i} g_i .
\end{align*}
Now if we define: for \(s=0,1,\dots,2^n-1\),
\[
    \widetilde y_s
    =
    y_{2^n+s}
    +
    \sum_{\ell=2^n+1}^{2^n+s} 2h^{(n+1)}_{\ell,2^n} g_{2^n},
    \qquad
    \widetilde g_s = g_{2^n+s}\quad(s=1,\dots,2^n-1)
\]
with the empty sum understood as zero (so that $\widetilde{y}_0 = y_{2^n}$), then
\begin{align*}
    \widetilde{y}_{j'} = \widetilde{y}_0 - \sum_{i=2^n+1}^{j} \sum_{u=i}^{j} 2h_{u,i} g_i = \widetilde{y}_0 - \sum_{i'=1}^{j'} \sum_{u=2^n+i'}^{2^n+j'} 2h_{u,2^n+i'} g_{2^n+i'} = \widetilde{y}_0 - \sum_{i'=1}^{j'} \sum_{u'=i'}^{j'} 2h_{u',i'} \widetilde{g}_{i'} 
\end{align*}
where we apply the change of indices $i'=i-2^n, u'=u-2^n$ and use $h_{u,i} = h_{u-2^n, i-2^n} = h_{u',i'}$.
Note that the above recurrence for $\widetilde{y}_{j'}$ is exactly the same as the recurrence for $y_{j'}$ with $\widetilde{g}_{i'}$ in place of $g_{i'}$.
Also note that $\nu_2(j') = \nu_2 (j-2^n) = \nu_2(j)$, which implies $m(j') = j' - 2^{\nu_2(j')} = j - 2^n - 2^{\nu_2(j)} = m(j) - 2^n$.
Hence, the binary expansion of $m(j)$ must take the form
\begin{align*}
    m(j) = 2^{n} + 2^{i_2} + \dots + 2^{i_p}
\end{align*}
with $n > i_2 > \dots > i_p$ and thus $m(j') = 2^{i_2} + \dots + 2^{i_p}$.
Then, by the induction hypothesis, we can use \eqref{eqn:FSDM-recurrence} to obtain
\begin{align*}
    \widetilde{y}_{j'} = \widetilde{y}_{j'-1} - 2^{\nu_2 (j')} \widetilde{g}_{j'} + \frac{1}{2} \left( \widetilde{y}_{m(j')} - \widetilde{y}_{j'-1} \right) + \sum_{r=1}^{p-1} 2^{\nu_2 (j') - r} \widetilde{g}_{t'_{p-r}}
\end{align*}
where $t'_r = \sum_{q=2}^{r+1} 2^{i_q} = t_{r+1} - 2^n$.
Plugging the definition of $\widetilde{y}_i, \widetilde{g}_i$ back in, the above identity becomes
\begin{align*}
    y_j + \sum_{i=2^n+1}^j 2h_{i,2^n} g_{2^n} = y_{j-1} + \sum_{i=2^n+1}^{j-1} 2h_{i,2^n} g_{2^n} - 2^{\nu_2(j)} g_j + \frac{1}{2} \left( y_{m(j)} - y_{j-1} - \sum_{i=m(j)+1}^{j-1} 2h_{i,2^n} g_{2^n} \right) + \sum_{r=1}^{p-1} 2^{\nu_2(j) - r} g_{t_{p-r+1}} 
\end{align*}
and simplifying, we obtain
\begin{align*}
    y_j & = y_{j-1} - 2^{\nu_2 (j)} g_j + \frac{1}{2} \left( y_{m(j)} - y_{j-1} \right) + \sum_{r=1}^{p-1} 2^{\nu_2(j) - r} g_{t_{p-r+1}} - \left( \frac{1}{2} \sum_{i=m(j)+1}^{j-1} 2h_{i,2^n} + 2h_{j,2^n} \right) g_{2^n} \\
    & = y_{j-1} - 2^{\nu_2 (j)} g_j + \frac{1}{2} \left( y_{m(j)} - y_{j-1} \right) + \sum_{r=1}^{p-1} 2^{\nu_2(j) - r} g_{t_{p-r+1}} + \left( \frac{1}{2} \sum_{i=m(j)+1}^{j-1} b_{i-2^n}^{(n)} + b_{j-2^n}^{(n)} \right) g_{2^n} .
\end{align*}
Therefore, the proof is complete once we show $\frac{1}{2} \sum_{i=m(j)+1}^{j-1} b_{i-2^n}^{(n)} + b_{j-2^n}^{(n)} = \frac{1}{2} \sum_{i'=m(j')+1}^{j'-1} b_{i'}^{(n)} + b_{j'}^{(n)} = 2^{\nu_2(j) - p}$. 
Here, note that $1\le j'<2^n$, $\nu_2(j) = \nu_2(j')$ and the binary expansion of $m(j')$ has $p-1$ distinct powers of 2, so we are done once we show the following lemma.

\begin{lemma}
\label{lemma:b_vector_sum}
For $u=1,\dots,2^n-1$, if $m(u) = u - 2^{\nu_2(u)} = 2^{i_1} + \dots + 2^{i_q}$ is the binary expansion ($i_1 > \dots > i_q$), then
\begin{align}
\label{eqn:b_vector_sum_formula}
    \frac{1}{2} \sum_{i=m(u)+1}^{u-1} b_{i}^{(n)} + b_{u}^{(n)} = 2^{\nu_2(u)-q-1} .
\end{align}
\end{lemma}

\begin{proof}[Proof of \cref{lemma:b_vector_sum}]
We use induction on $n$ to show that \eqref{eqn:b_vector_sum_formula} holds for all $u=1,\dots,2^n-1$.
First, for $n=1$, we have $b_1^{(1)} = \frac{1}{2}$, which agrees with \eqref{eqn:b_vector_sum_formula} because $\nu_2(1) = 0$ and $q=0$ for $u=1$ since $m(1) = 0$ has no terms in its binary expansion.

Next, let $n\ge 1$ and assume that \eqref{eqn:b_vector_sum_formula} holds for all $u=1,\dots,2^n-1$.
Note that by \cref{lemma:FSDM-H-matrix}, we have
\begin{align}
\label{eqn:b_vector_recurrence}
    b^{(n+1)} = \begin{bmatrix} \left( b^{(n)} \right)^\transpose & \frac{2^{n}+1}{4} & \left( \frac{1}{2} b^{(n)} \right)^\transpose \end{bmatrix}^\transpose ,
\end{align}
so the first $2^n-1$ entries of $b^{(n+1)}$ are exactly the same as $b^{(n)}$.
Hence, the formula \eqref{eqn:b_vector_sum_formula} immediately holds for $b^{(n+1)}$ and $v=1,\dots,2^n-1$, and we only need to consider the case $2^n \le v \le 2^{n+1}-1$.
For $v=2^n$, observe that $\nu_2(v)=n$, $m(v)=0$ and $q=0$.
From the proof of \cref{lemma:FSDM-H-matrix}, we have $\sum_{j=1}^{2^n-1} b_j^{(n)} = \sum_{j=1}^{2^n-1} a_j^{(n)} = \frac{2^n-1}{2}$, which implies
\[
    \frac{1}{2} \sum_{j=1}^{2^n-1} b_j^{(n+1)} + b_{2^n}^{(n+1)} = \frac{1}{2} \sum_{j=1}^{2^n-1} b_j^{(n)} + \frac{2^n+1}{4} = \frac{2^n-1}{4} + \frac{2^n+1}{4} = 2^{n-1} = 2^{\nu_2(v) - q - 1} ,
\]
as desired.
For $2^n < v \le 2^{n+1} - 1$, observe that $\nu_2(v) = \nu_2 (v')$ where $v' = v - 2^n$, so we have $m(v) = m(v') + 2^n$.
Also, for any $2^n < i \le 2^{n+1} - 1$, by \eqref{eqn:b_vector_recurrence}, we have $b_i^{(n+1)} = \frac{1}{2} b_{i-2^n}^{(n)}$.
Therefore,
\begin{align*}
    \frac{1}{2} \sum_{i=m(v)+1}^{v-1} b_i^{(n+1)} + b_{v}^{(n+1)} & = \frac{1}{2} \sum_{i=m(v)+1}^{v-1} \frac{1}{2} b_{i-2^n}^{(n)} + \frac{1}{2} b_{v-2^n}^{(n)} \\
    & = \frac{1}{4} \sum_{i'=m(v')+1}^{v'-1} b_{i'}^{(n)} + \frac{1}{2} b_{v'}^{(n)} = \frac{1}{2} \left( \frac{1}{2} \sum_{i'=m(v')+1}^{v'-1} b_{i'}^{(n)} + b_{v'}^{(n)}\right) = \frac{1}{2} \cdot 2^{\nu_2(v')-q} = 2^{\nu_2(v)-q-1} 
\end{align*}
where we have used the induction hypothesis for $v'$, together with the fact that $m(v') = m(v) - 2^n$ has the binary expansion given by removing $2^n$ from that of $m(v)$, so it has $q-1$ terms (powers of $2$).
This completes the induction, and hence the proof.
\end{proof}

\end{document}

%% file: macros.tex
\DeclareMathOperator{\spann}{span}

\newcommand{\cut}[1]{{}}

\newcommand{\vb}{{\mathbf{b}}}
\newcommand{\vc}{{\mathbf{c}}}
\newcommand{\vd}{{\mathbf{d}}}
\newcommand{\ve}{{\mathbf{e}}}

\newcommand{\vl}{{\mathbf{l}}}
\newcommand{\vm}{{\mathbf{m}}}

\newcommand{\vq}{{\mathbf{q}}}

\newcommand{\vt}{{\mathbf{t}}}
\newcommand{\vu}{{\mathbf{u}}}
\newcommand{\vv}{{\mathbf{v}}}
\newcommand{\vw}{{\mathbf{w}}}

\newcommand{\vA}{{\mathbf{A}}}

\newcommand{\vI}{{\mathbf{I}}}

\newcommand{\vL}{{\mathbf{L}}}
\newcommand{\vM}{{\mathbf{M}}}

\newcommand{\vP}{{\mathbf{P}}}

\newcommand{\vR}{{\mathbf{R}}}
\newcommand{\vS}{{\mathbf{S}}}
\newcommand{\vT}{{\mathbf{T}}}
\newcommand{\vU}{{\mathbf{U}}}

\newcommand{\vW}{{\mathbf{W}}}

\newcommand{\cA}{{\mathcal{A}}}

\newcommand{\cC}{{\mathcal{C}}}

\newcommand{\cE}{{\mathcal{E}}}
\newcommand{\cF}{{\mathcal{F}}}
\newcommand{\cG}{{\mathcal{G}}}
\newcommand{\cH}{{\mathcal{H}}}
\newcommand{\cI}{{\mathcal{I}}}

\newcommand{\cM}{{\mathcal{M}}}

\newcommand{\cO}{{\mathcal{O}}}

\newcommand{\cT}{{\mathcal{T}}}
\newcommand{\cU}{{\mathcal{U}}}
\newcommand{\cV}{{\mathcal{V}}}
\newcommand{\cW}{{\mathcal{W}}}
\newcommand{\cX}{{\mathcal{X}}}

\usepackage[bb=boondox]{mathalfa}

\usepackage{xparse}
\DeclareFontFamily{U}{ntxmia}{}
\DeclareFontShape{U}{ntxmia}{m}{it}{<-> ntxmia }{}
\DeclareFontShape{U}{ntxmia}{b}{it}{<-> ntxbmia }{}
\DeclareSymbolFont{lettersA}{U}{ntxmia}{m}{it}
\SetSymbolFont{lettersA}{bold}{U}{ntxmia}{b}{it}

\ExplSyntaxOn
\NewDocumentCommand{\varmathbb}{m}
 {
  \tl_map_inline:nn { #1 }
   {
    \use:c { varbb##1 }
   }
 }
\tl_map_inline:nn { ABCDEFGHIJKLMNOPQRSTUVWXYZ }
 {
  \exp_args:Nc \DeclareMathSymbol{varbb#1}{\mathord}{lettersA}{\int_eval:n { `#1+67 }}
 }
\exp_args:Nc \DeclareMathSymbol{varbbk}{\mathord}{lettersA}{169}
\ExplSyntaxOff

\newcommand*{\gra}{\mathrm{Gra}\,}

\newcommand{\reals}{\mathbb{R}}

\newcommand{\opA}{{\varmathbb{A}}}

\newcommand{\opI}{{\varmathbb{I}}}

\newcommand{\opT}{{\varmathbb{T}}}

\newcommand{\inprod}[2]{\left\langle #1,#2 \right\rangle}
\newcommand{\sqnorm}[1]{\left\| #1 \right\|^2}

\newcommand{\norm}[1]{\left\|#1\right\|}

\newcommand{\at}{\textsf{A}}

\newcommand{\fA}{\mathfrak{A}}

\newcommand{\transpose}{\mathsf{T}}

\newcommand{\glue}{\mathbin{\bowtie}}

\newcommand{\redd}[1]{{\color{red}#1}}

%% file: ref.bib
@book{bauschke2017convex,
  title={Convex Analysis and Monotone Operator Theory in Hilbert Spaces},
  author={Bauschke, Heinz H and Combettes, Patrick L},
  year={2017},
  publisher={Springer},
  edition={Second}
}

@article{Diakonikolas2020_halpern,
  title = {Halpern Iteration for Near-Optimal and Parameter-Free Monotone Inclusion and Strong Solutions to Variational Inequalities},
  author = {Diakonikolas, Jelena},
  year = {2020},
  journal = {Conference on Learning Theory}
}

@article{Halpern1967_fixed,
  title = {Fixed Points of Nonexpanding Maps},
  author = {Halpern, Benjamin},
  year = {1967},
  journal = {Bulletin of the American Mathematical Society},
  volume = {73},
  number = {6},
  pages = {957--961}
}

@article{Kim2021_accelerated,
  title = {Accelerated Proximal Point Method for Maximally Monotone Operators},
  author = {Kim, Donghwan},
  year = {2021},
  journal = {Mathematical Programming},
  volume = {190},
  number = {1{\textendash}2},
  pages = {57--87}
}

@article{KimFessler2021_optimizing,
  title = {Optimizing the Efficiency of First-Order Methods for Decreasing the Gradient of Smooth Convex Functions},
  author = {Kim, Donghwan and Fessler, Jeffrey A.},
  year = {2021},
  journal = {Journal of Optimization Theory and Applications},
  volume = {188},
  number = {1},
  pages = {192--219}
}

@article{KimOzdaglarParkRyu2023_timereversed,
  title = {Time-Reversed Dissipation Induces Duality between Minimizing Gradient Norm and Function Value},
  author = {Kim, Jaeyeon and Ozdaglar, Asuman E. and Park, Chanwoo and Ryu, Ernest K.},
  year = {2023},
  journal = {Neural Information Processing Systems}
}

@article{KimParkOzdaglarDiakonikolasRyu2023_mirror,
  title = {Mirror Duality in Convex Optimization},
  author = {Kim, Jaeyeon and Park, Chanwoo and Ozdaglar, Asuman and Diakonikolas, Jelena and Ryu, Ernest K.},
  year = {2023},
  journal = {arXiv:2311.17296},
  eprint = {2311.17296},
  archiveprefix = {arxiv}
}

@article{Krasnoselskii1955_two,
  title = {Two Remarks on the Method of Successive Approximations},
  author = {Krasnosel'skii, M. A.},
  year = {1955},
  journal = {Uspekhi Matematicheskikh Nauk},
  volume = {10},
  number = {1},
  pages = {123--127},
  ajournal = {Usp. Mat. Nauk}
}

@article{LeeKim2021_fast,
  title = {Fast Extra Gradient Methods for Smooth Structured Nonconvex-Nonconcave Minimax Problems},
  author = {Lee, Sucheol and Kim, Donghwan},
  year = {2021},
  journal = {Neural Information Processing Systems}
}

@article{Lieder2021_convergence,
  title = {On the Convergence Rate of the {{Halpern-iteration}}},
  author = {Lieder, Felix},
  year = {2021},
  journal = {Optimization Letters},
  volume = {15},
  number = {2},
  pages = {405--418}
}

@article{Mann1953_mean,
  title = {Mean Value Methods in Iteration},
  author = {Mann, William Robert},
  year = {1953},
  journal = {Proceedings of the American Mathematical Society},
  volume = {4},
  number = {3},
  pages = {506--510},
  ajournal = {Proc. Amer. Math. Soc.}
}

@article{MokhtariOzdaglarPattathil2020_unified,
  title = {A Unified Analysis of Extra-Gradient and Optimistic Gradient Methods for Saddle Point Problems: {{Proximal}} Point Approach},
  author = {Mokhtari, Aryan and Ozdaglar, Asuman and Pattathil, Sarath},
  editor = {Chiappa, Silvia and Calandra, Roberto},
  year = {2020},
  journal = {International Conference on Artificial Intelligence and Statistics}
}

@article{ParkRyu2022_exact,
  title = {Exact Optimal Accelerated Complexity for Fixed-Point Iterations},
  author = {Park, Jisun and Ryu, Ernest K.},
  year = {2022},
  journal = {International Conference on Machine Learning}
}

@book{ryuLargescaleConvexOptimization2022,
  title = {Large-Scale Convex Optimization: {{Algorithms}} \& Analyses via Monotone Operators},
  author = {Ryu, Ernest K. and Yin, Wotao},
  year = 2022,
  publisher = {Cambridge University Press},
  address = {Cambridge}
}

@article{SabachShtern2017_first,
  title = {A First Order Method for Solving Convex Bilevel Optimization Problems},
  author = {Sabach, Shoham and Shtern, Shimrit},
  year = {2017},
  journal = {SIAM Journal on Optimization},
  volume = {27},
  number = {2},
  pages = {640--660},
  publisher = {{Society for Industrial and Applied Mathematics}},
  urldate = {2021-09-26}
}

@article{SuhParkRyu2023_continuoustime,
  title = {Continuous-Time Analysis of Anchor Acceleration},
  author = {Suh, Jaewook J. and Park, Jisun and Ryu, Ernest K.},
  year = {2023},
  journal = {Neural Information Processing Systems}
}

@article{Tran-DinhLuo2021_halperntype,
  title = {Halpern-Type Accelerated and Splitting Algorithms for Monotone Inclusions},
  author = {{Tran-Dinh}, Quoc and Luo, Yang},
  year = {2021},
  journal = {arXiv:2110.08150},
  eprint = {2110.08150},
  archiveprefix = {arxiv}
}

@article{YoonRyu2021_accelerated,
  title = {Accelerated Algorithms for Smooth Convex-Concave Minimax Problems with $\mathcal{O}(1/k^2)$ Rate on Squared Gradient Norm},
  author = {Yoon, TaeHo and Ryu, Ernest K.},
  year = {2021},
  journal = {International Conference on Machine Learning},
  series = {Proceedings of Machine Learning Research}
}

@article{botFastOptimisticGradient2025,
  title = {Fast {{Optimistic Gradient Descent Ascent}} ({{OGDA}}) {{Method}} in {{Continuous}} and {{Discrete Time}}},
  author = {Bo{\c t}, Radu Ioan and Csetnek, Ern{\"o} Robert and Nguyen, Dang-Khoa},
  year = 2025,
  journal = {Foundations of Computational Mathematics},
  volume = {25},
  number = {1},
  pages = {163--222}
}

@article{YoonKimSuhRyu2024_optimal,
  title={Optimal acceleration for minimax and fixed-point problems is not unique},
  author={Yoon, TaeHo and Kim, Jaeyeon and Suh, Jaewook J and Ryu, Ernest K.},
  journal={International Conference on Machine Learning},
  year={2024}
}

@article{yoonAcceleratedMinimaxAlgorithms2025a,
  title = {Accelerated Minimax Algorithms Flock Together},
  author = {Yoon, TaeHo and Ryu, Ernest K.},
  year = 2025,
  journal = {SIAM Journal on Optimization},
  volume = {35},
  number = {1},
  pages = {180--209}
}

@article{YoonRyuGrimmerInvariance2025,
  title={H-invariance theory: A complete characterization of minimax optimal fixed-point algorithms},
  author={Yoon, TaeHo and Ryu, Ernest K.\ and Grimmer, Benjamin},
  journal={arXiv:2511.14915},
  year={2025}
}

@article{ChaikenCombinatorial1982,
  title={A combinatorial proof of the all minors matrix tree theorem},
  author={Chaiken, Seth},
  journal={SIAM Journal on Algebraic Discrete Methods},
  volume={3},
  number={3},
  pages={319--329},
  year={1982},
  publisher={SIAM}
}

@article{grimmerComposingOptimizedStepsize,
  title = {Composing Optimized Stepsize Schedules for Gradient Descent},
  author = {Grimmer, Benjamin and Shu, Kevin and Wang, Alex L.},
  journal = {Mathematics of Operations Research},
  volume = {0},
  number = {0},
  eprint = {https://pubsonline.informs.org/doi/pdf/10.1287/moor.2024.0764},
  year = {2025}
}

@article{altschulerAccelerationStepsizeHedging2025,
  title = {Acceleration by Stepsize Hedging: {{Multi-step}} Descent and the Silver Stepsize Schedule},
  author = {Altschuler, Jason M. and Parrilo, Pablo A.},
  year = 2025,
  journal = {Journal of The ACM},
  volume = {72},
  number = {2},
  publisher = {Association for Computing Machinery},
  address = {New York, NY, USA}
}

@article{altschulerAccelerationStepsizeHedging2025a,
  title = {Acceleration by Stepsize Hedging: {{Silver Stepsize Schedule}} for Smooth Convex Optimization},
  author = {Altschuler, Jason M. and Parrilo, Pablo A.},
  year = 2025,
  journal = {Mathematical Programming},
  volume = {213},
  number = {1},
  pages = {1105--1118}
}

@article{grimmerAcceleratedObjectiveGap2025,
  title = {Accelerated Objective Gap and Gradient Norm Convergence for Gradient Descent via Long Steps},
  author = {Grimmer, Benjamin and Shu, Kevin and Wang, Alex L.},
  year = 2025,
  journal = {INFORMS Journal on Optimization},
  volume = {7},
  number = {2},
  eprint = {https://doi.org/10.1287/ijoo.2024.0057},
  pages = {156--169}
}

@article{grimmerProvablyFasterGradient2024,
  title = {Provably Faster Gradient Descent via Long Steps},
  author = {Grimmer, Benjamin},
  year = 2024,
  journal = {SIAM Journal on Optimization},
  volume = {34},
  number = {3},
  pages = {2588--2608}
}

@article{zhangAcceleratedGradientDescent2026,
  title = {Accelerated Gradient Descent by Concatenation of Stepsize Schedules},
  author = {Zhang, Zehao and Jiang, Rujun},
  year = 2026,
  journal = {SIAM Journal on Optimization},
  volume = {36},
  number = {2},
  pages = {1182--1210}
}

@article{kimProofExactConvergence2024,
  title = {A Proof of the Exact Convergence Rate of Gradient Descent},
  author = {Kim, Jungbin},
  year = 2024,
  journal = {arXiv:2412.04427},
  eprint = {2412.04427},
  archiveprefix = {arXiv}
}

@article{BauschkeMoffatWang2012_firmly,
  title = {Firmly Nonexpansive Mappings and Maximally Monotone Operators: Correspondence and Duality},
  author = {Bauschke, Heinz H and Moffat, Sarah M and Wang, Xianfu},
  year = 2012,
  journal = {Set-Valued and Variational Analysis},
  volume = {20},
  pages = {131--153},
  publisher = {Springer}
}

@article{tran-dinhHalpernsFixedpointIterations2024,
  title = {From {{Halpern}}'s Fixed-Point Iterations to {{Nesterov}}'s Accelerated Interpretations for Root-Finding Problems},
  author = {{Tran-Dinh}, Quoc},
  year = 2024,
  journal = {Computational Optimization and Applications},
  volume = {87},
  number = {1},
  pages = {181--218}
}

@article{CaiSongGuzmanDiakonikolas2022_stochastic,
  title = {Stochastic {{Halpern}} Iteration with Variance Reduction for Stochastic Monotone Inclusions},
  author = {Cai, Xufeng and Song, Chaobing and Guzm{\'a}n, Crist{\'o}bal and Diakonikolas, Jelena},
  editor = {Koyejo, S. and Mohamed, S. and Agarwal, A. and Belgrave, D. and Cho, K. and Oh, A.},
  year = 2022,
  journal = {Neural Information Processing Systems}
}

@article{shuUnifiedTreatmentDuality2026,
  title = {A {{Unified Theory}} of {{H-Duality}} in {{First-Order Methods}}},
  author = {Shu, Kevin and Wang, Alex L.},
  year = 2026,
  journal = {Forthcoming}
}

@article{contrerasOptimalErrorBounds2023,
  title = {Optimal Error Bounds for Non-Expansive Fixed-Point Iterations in Normed Spaces},
  author = {Contreras, Juan Pablo and Cominetti, Roberto},
  year = 2023,
  journal = {Mathematical Programming},
  volume = {199},
  number = {1},
  pages = {343--374}
}

@article{bravo2026minimax,
  title = {Minimax-Optimal Halpern Iterations for Lipschitz Maps},
  author = {Bravo, Mario and Cominetti, Roberto and Lee, Jongmin},
  year = 2026,
  journal = {arXiv:2601.15996},
  eprint = {2601.15996},
  archiveprefix = {arXiv}
}

@article{leeAcceleratingValueIteration2023,
  title = {Accelerating Value Iteration with Anchoring},
  author = {Lee, Jongmin and Ryu, Ernest},
  editor = {Oh, A. and Naumann, T. and Globerson, A. and Saenko, K. and Hardt, M. and Levine, S.},
  year = 2023,
  journal = {Neural Information Processing Systems}
}

@article{bravoUniversalBoundsFixed2022,
  title = {Universal Bounds for Fixed Point Iterations via Optimal Transport Metrics},
  author = {Bravo, Mario and Champion, Thierry and Cominetti, Roberto},
  year = 2022,
  journal = {Applied Set-Valued Analysis \& Optimization},
  volume = {4},
  number = {3}
}

@article{black2024linear,
  title = {From Linear Programming to Colliding Particles},
  author = {Black, Alexander E and L{\"u}tjeharms, Niklas and Sanyal, Raman},
  year = 2024,
  journal = {arXiv:2405.08506},
  eprint = {2405.08506},
  archiveprefix = {arXiv}
}

@article{bravoStochasticFixedpointIterations2024,
  title = {Stochastic Fixed-Point Iterations for Nonexpansive Maps: {{Convergence}} and Error Bounds},
  author = {Bravo, Mario and Cominetti, Roberto},
  year = 2024,
  journal = {SIAM Journal on Control and Optimization},
  volume = {62},
  number = {1},
  pages = {191--219}
}

@article{bravoStochasticHalpernIteration2026,
  title = {Stochastic {{Halpern}} Iteration in Normed Spaces and Applications to Reinforcement Learning},
  author = {Bravo, Mario and Contreras, Juan Pablo},
  year = 2026,
  journal = {Mathematical Programming}
}

@article{caiVarianceReducedHalpern2024,
  title = {Variance Reduced Halpern Iteration for Finite-Sum Monotone Inclusions},
  author = {Cai, Xufeng and Alacaoglu, Ahmet and Diakonikolas, Jelena},
  year = 2024,
  journal = {International Conference on Learning Representations}
}

@article{chenNearoptimalAlgorithmsMaking2024,
  title = {Near-Optimal Algorithms for Making the Gradient Small in Stochastic Minimax Optimization},
  author = {Chen, Lesi and Luo, Luo},
  year = 2024,
  journal = {Journal of Machine Learning Research},
  volume = {25},
  number = {387},
  pages = {1--44}
}

@article{pischkeAsymptoticRegularityGeneralised2026,
  title = {Asymptotic {{Regularity}} of a {{Generalised Stochastic Halpern Scheme}}},
  author = {Pischke, Nicholas and Powell, Thomas},
  year = 2026,
  journal = {Journal of Optimization Theory and Applications},
  volume = {210},
  number = {1},
  pages = {3}
}
